\documentclass{amsart}
\usepackage{amssymb}
\usepackage{amsfonts}
\usepackage{geometry}

\newtheorem{theorem}{Theorem}
\theoremstyle{plain}

\newtheorem{corollary}[theorem]{Corollary}

\newtheorem{definition}[theorem]{Definition}

\newtheorem{lemma}[theorem]{Lemma}

\newtheorem{proposition}[theorem]{Proposition}
\newtheorem{remark}[theorem]{Remark}

\numberwithin{equation}{section}
\input{tcilatex}
\begin{document}
\title[The corona problem for kernel multiplier algebras]{The corona problem for kernel multiplier algebras}
\author[E. T. Sawyer]{Eric T. Sawyer$^\dagger$}
\address{Eric T. Sawyer, Department of Mathematics and Statistics\\
McMaster University\\
1280 Main Street West\\
Hamilton, Ontario L8S 4K1 Canada}
\thanks{$\dagger $ Research supported in part by a grant from the National
Science and Engineering Research Council of Canada.}
\email{sawyer@mcmaster.ca}
\author[B. D. Wick]{Brett D. Wick$^\ddagger$}
\address{Brett D. Wick, Department of Mathematics\\
Washington University -- St. Louis\\
One Brookings Drive\\
St. Louis, MO USA 63130}
\email{wick@math.wustl.edu}
\thanks{$\ddagger$ Research supported in part by National Science Foundation
DMS grant \# 1603246 and \#1560955.}
\date{\today }

\begin{abstract}
We prove an alternate Toeplitz corona theorem for the algebras of pointwise
kernel multipliers of Besov-Sobolev spaces on the unit ball in $\mathbb{C}%
^{n}$, and for the algebra of bounded analytic functions on certain strictly
pseudoconvex domains and polydiscs in higher dimensions as well. This
alternate Toeplitz corona theorem extends to more general Hilbert function
spaces where it does not require the complete Pick property. Instead, the
kernel functions $k_{x}\left( y\right) $ of certain Hilbert function spaces $%
\mathcal{H}$ are assumed to be invertible multipliers on $\mathcal{H}$, and
then we continue a research thread begun by Agler and McCarthy in 1999, and
continued by Amar in 2003, and most recently by Trent and Wick in 2009. In
dimension $n=1$ we prove the corona theorem for the kernel multiplier
algebras of Besov-Sobolev Banach spaces in the unit disk, extending the result for Hilbert spaces $H^\infty \cap Q_p$ by A. Nicolau and J. Xiao.  
\end{abstract}

\maketitle
\tableofcontents

\section{Introduction}

In 1962 L. Carleson \cite{Car} proved the corona theorem for the algebra of
bounded analytic functions on the unit disk. The proof used a beautiful
`corona construction' together with properties of Blaschke products. While
there is a large literature on corona theorems for domains in one complex
dimension (see e.g. \cite{Nik}), progress in higher dimensions prior to 2011
had been limited to the $H^{p}$ corona theorem on various bounded domains in 
$\mathbb{C}^{n}$, and weaker results restricting $N$ to $2$ generators. In
fact, apart from the simple cases in which the maximal ideal space of the
algebra can be identified with a compact subset of $\mathbb{C}^{n}$, no
complete corona theorem was proved in higher dimensions until the 2011
results of S. Costea and the authors in which the corona theorem was
established for the multiplier algebras $M_{H_{n}^{2}}$ and \thinspace $M_{%
\mathcal{D}_{n}}$ of the Drury-Arveson Hardy space $H_{n}^{2}$ and the
Dirichlet space $\mathcal{D}_{n}$ on the ball in $n$ dimensions. These
latter results used the abstract Toeplitz corona theorem for Hilbert
function spaces with a complete Pick kernel, see Ball, Trent and Vinnikov 
\cite{BaTrVi} and Ambrosie and Timotin \cite{AmTi}. The unresolved corona
question for the algebra of bounded analytic functions on the ball in higher
dimensions has remained a tantalizing problem for over half a century now
(see e.g. \cite{CoSaWi2} and \cite{DoKrSaTrWi} for a more detailed history
of this problem to date). We note in particular that Varopoulos \cite{Var}
gave an example of Carleson measure data in dimension $n=2$ for which there
is \emph{no} bounded solution to the $\overline{\partial }$ equation. This
poses a significant obstacle to using the $\overline{\partial }$ equation
for the multiplier problem, and suggests a more operator theoretic approach
akin to the Toeplitz corona theorem in order to solve the corona problem in
higher dimensions.

In this paper we prove in particular an alternate Toeplitz corona theorem
for all of the algebras of kernel multipliers of Besov-Sobolev spaces $%
B_{2}^{\sigma }\left( \mathbb{B}_{n}\right) $ on the ball. These spaces
include $H^{\infty }\left( \mathbb{B}_{n}\right) $, and the alternate
Toeplitz corona theorem also extends to the algebra $H^{\infty }\left(
\Omega \right) $ of bounded analytic functions on $\Omega $, where $\Omega $
is either the unit polydisc $\mathbb{D}^{n}$ in $\mathbb{C}^{n}$, a
sufficiently small $C^{\infty }$ perturbation of the unit ball $\mathbb{B}%
_{n}$, or a bounded strictly pseudoconvex homogeneous complete circular
domain in $\mathbb{C}^{n}$. Moreover, this alternate Toeplitz corona theorem
extends to the kernel multiplier space of certain Hilbert function spaces 
\emph{without} assuming the complete Pick property. This essentially shows
that whenever a Hilbert space has one of these special kernels, then the
Corona Property for its kernel multiplier algebra reduces to what we call
the Convex Poisson Property, a property which can be addressed by methods
involving solution of the $\overline{\partial }$ problem.

To illustrate in a very special case, we show that for $\Omega $ as above,
the algebra $H^{\infty }\left( \Omega \right) $ has the \emph{Corona Property%
} if and only if the Bergman space $A^{2}\left( \Omega \right) $ has the 
\emph{Convex Poisson Property} if and only if the Hardy space $H^{2}\left(
\Omega \right) $ has the \emph{Convex Poisson Property} (in order to define
the Hardy space we need to assume that $\partial \Omega $ is $C^{2}$). The
Corona Property for $H^{\infty }\left( \Omega \right) $ is this: given $%
\varphi _{1},\ldots ,\varphi _{N}\in H^{\infty }\left( \Omega \right) $
satisfying%
\begin{equation*}
1\geq \max \left\{ \left\vert \varphi _{1}\left( z\right) \right\vert
^{2},\ldots ,\left\vert \varphi _{N}\left( z\right) \right\vert ^{2}\right\}
\geq c^{2}>0,\ \ \ \ \ z\in \Omega ,
\end{equation*}%
there are a positive constant $C$ and $f_{1},\ldots ,f_{N}\in H^{\infty
}\left( \Omega \right) $ satisfying 
\begin{eqnarray*}
\max \left\{ \left\vert f_{1}\left( z\right) \right\vert ^{2},\ldots
,\left\vert f_{N}\left( z\right) \right\vert ^{2}\right\} &\leq &C^{2},\ \ \
\ \ z\in \Omega , \\
\varphi _{1}\left( z\right) f_{1}\left( z\right) +\cdots +\varphi _{N}\left(
z\right) f_{N}\left( z\right) &=&1,\ \ \ \ \ z\in \Omega .
\end{eqnarray*}%
For points $\mathbf{a}=\left( a_{1},\ldots ,a_{M}\right) \in \Omega ^{M}$
and $\mathbf{\theta }=\theta _{0},\ldots ,\theta _{M}\in \left[ 0,1\right]
^{M+1}$, and for $h\in \mathcal{H}=A^{2}\left( \Omega \right) $ or $%
H^{2}\left( \Omega \right) $, where $\Omega $ is as above, set $\widetilde{%
k_{a_{0}}}\equiv 1$ by convention. Then we have%
\begin{equation*}
\left\Vert h\right\Vert _{\mathcal{H}^{\mathbf{a,\theta }}}^{2}\equiv
\sum_{m=0}^{M}\theta _{m}\int_{\Omega }\left\vert h\right\vert
^{2}\left\vert \widetilde{k_{a_{m}}}\right\vert ^{2}dv<\infty .
\end{equation*}%
The Convex Poisson Property for the Hilbert space $\mathcal{H}$ (either
Bergman or Hardy) is then this: given $\varphi _{1},\ldots ,\varphi _{N}\in
H^{\infty }\left( \Omega \right) $ satisfying%
\begin{equation*}
1\geq \max \left\{ \left\vert \varphi _{1}\left( z\right) \right\vert
^{2},\ldots ,\left\vert \varphi _{N}\left( z\right) \right\vert ^{2}\right\}
\geq c^{2}>0,\ \ \ \ \ z\in \Omega ,
\end{equation*}%
there is a positive constant $C$ such that for all points $\mathbf{a}=\left(
a_{1},\ldots ,a_{M}\right) \in \Omega ^{M}$ and all $\mathbf{\theta }=\theta
_{0},\ldots ,\theta _{M}\in \left[ 0,1\right] ^{M+1}$ with $%
\sum_{m=0}^{M}\theta _{m}=1$, there are $f_{1},\ldots ,f_{N}\in \mathcal{H}$
satisfying 
\begin{eqnarray*}
\sum_{\ell =1}^{N}\left\Vert f_{\ell }\right\Vert _{\mathcal{H}^{\mathbf{%
a,\theta }}}^{2} &\leq &C^{2}, \\
\varphi _{1}\left( z\right) f_{1}\left( z\right) +\cdots +\varphi _{N}\left(
z\right) f_{N}\left( z\right) &=&1,\ \ \ \ \ z\in \Omega .
\end{eqnarray*}

More generally we prove an \emph{alternate} Toeplitz corona theorem for the
kernel multiplier space $K_{\mathcal{H}}$ of a Hilbert function space $%
\mathcal{H}$ whose reproducing kernel $k$ need not be a complete Pick
kernel, but must have the property (among others) that the kernel functions $%
k_{a}$ are invertible multipliers on $\mathcal{H}$. Here the Banach space $%
K_{\mathcal{H}}$ of kernel multipliers consists of those functions $\varphi
\in \mathcal{H}$ such that $\sup_{a\in \Omega }\frac{\left\Vert \varphi
k_{a}\right\Vert _{\mathcal{H}}}{\left\Vert k_{a}\right\Vert _{\mathcal{H}}}%
<\infty $. The roots of this abstract result can be traced back to a
research thread begun by Agler and McCarthy \cite{AgMc} in the bidisc using
Ando's theorem to reduce the corona problem to estimates on weighted Hardy
spaces (see also \cite[Chapters 11 and 13]{AgMc2} for a nice survey of
related prior work), and continued by work of Amar \cite{Amar} who
introduced the use of the von Neumann minimax theorem to circumvent Ando's
theorem and go beyond the bidisc, and more recently by work of Trent and the
second author \cite{TrWi} who further reduced matters to checking weights
whose densities are the modulus squared of nonvanishing $H^{\infty }$
functions whose boundary values have bounded reciprocals. In this paper we
extend and refine this approach to more general Hilbert function spaces and
reduce matters to solving the Bezout equation $\varphi \cdot f=1$ with
solutions $f$ in an \emph{analogue} of a weighted $L^{2}$ space - e.g. the
method applies to Besov-Sobolev spaces on the ball whose norms \emph{cannot}
be given as an $L^{2}$ norm with respect to some measure.

In dimension $n=1$ we can prove new Corona theorems for the kernel multiplier algebras of the Besov-Sobolev spaces $B_p^{\sigma}$ of the unit disk (See Section 7 for definitions of these spaces).  In the special case of $B_2^{\sigma}(\mathbb{D})$, the kernel multiplier algebra coincides with the Hilbert spaces $H^\infty\cap Q_{2\sigma}$, see Nicolau and Xiao \cite{NiXi}, Ess\'en and Xiao \cite{EsXi}, Aulaskari, Stegenga and Xiao \cite{AuStXi}, or Xiao \cite{Xia, Xia3} for the definitions of these spaces.

\subsection{Organization of the paper}

In Section 2 of this paper we introduce background material on a Hilbert
function space on a set $\Omega $, as well as our new definition of the
Banach space $K_{\mathcal{H}}$ of kernel multipliers on a Hilbert function
space $\mathcal{H}$. The space $K_{\mathcal{H}}$ is often an algebra and
plays the pivotal role in our alternate Toeplitz corona theorem, given as
Theorem \ref{alternate} below.  

In Section 3 we use von Neumann's minimax theorem, following the thread
begun by Amar \cite{Amar} and continued by Trent and Wick \cite{TrWi}, to
characterize those vectors of corona data $\varphi \in \oplus ^{N}L^{\infty
} $ for which solutions $f\in \oplus ^{N}K_{\mathcal{H}}$ exist to Bezout's
equation $\varphi \cdot f=1$ in $\Omega $.

In Section 4 we use this characterization to obtain an alternate Toeplitz
corona theorem for the spaces $K_{\mathcal{H}}$ that uses a Convex Poisson
Property instead of a Baby Corona Property as in the Toeplitz corona theorem
for $M_{\mathcal{H}} $ in \cite{BaTrVi} and \cite{AmTi}. Instead of
requiring a complete Pick kernel we require the following four properties:

\begin{enumerate}
\item the reproducing kernels $k_{a}$ for $\mathcal{H}$ are invertible
multipliers of $\mathcal{H}$, and the map $a\rightarrow k_{a}$ is lower
semicontinuous from $\Omega $ to $M_{\mathcal{H}}$,

\item the kernel multiplier space $K_{\mathcal{H}}$ is an algebra,

\item the constant function $1$ is in $\mathcal{H}$, and

\item the unit ball of $\mathcal{H}$ enjoys a Montel property.
\end{enumerate}

This alternate Toeplitz corona theorem thus reduces the Corona Property for $%
K_{\mathcal{H}}$ to the Convex Poisson Property for $\mathcal{H}$. As an
application, we give examples of domains $\Omega $ in $\mathbb{C}^{n}$ for
which the Corona Property for $H^{\infty }\left( \Omega \right) $ is
equivalent to the Convex Poisson Property with $\mathcal{H}$ taken to be
either the Bergman or Hardy space on $\Omega $. Section 5 is devoted to
these higher dimensional examples, which of course include the ball and
polydisc. However, we are unable to obtain any new corona theorems in higher
dimensions.

Then in Section 6, we discuss the Invertible Multiplier Property, which when
it holds, gives corona theorems in many situations as a corollary of our
alternate Toeplitz corona theorem. However, the Invertible Multiplier
Property is known to hold only for the Szeg\"{o} kernel $\frac{1}{1-%
\overline{a}z}$ in dimension $n=1$, where it can be used to prove a corona
theorem in the disk, the annulus, and any other planar domain for which the
reproducing kernel is essentially the Szeg\"{o} kernel $\frac{1}{1-\overline{%
a}z}$. In particular, we show the Invertible Multiplier Property fails for
the Szeg\"{o} kernel on both the ball and polydisc in higher dimensions.

Finally, in Section 7:

\begin{enumerate}
\item we prove that the space $K_{\mathcal{H}}$ is an algebra when $\mathcal{%
H}$ is a Besov-Sobolev Hilbert space $B_{2}^{\sigma }\left( \mathbb{B}%
_{n}\right) $ in the ball, $\sigma >0$, and conclude that $K_{\mathcal{H}}$
has the Corona Property if and only if the Convex Poisson Property holds for 
$B_{2}^{\sigma }\left( \mathbb{B}_{n}\right) $,

\item we provide a new proof of A. Nicolau and J. Xiao's result that the Corona Property holds for the algebras $K_{\mathcal{H%
}}$ when $\mathcal{H}=B_{2}^{\sigma }\left( \mathbb{D}\right) $ is a
Besov-Sobolev space in the disk with $\sigma >0$, see \cite{NiXi}; and we further demonstrate that the Corona Property holds for the algebra $K_{p}^{\sigma}$ when $0<\sigma<\frac{1}{p}$ and $1<p<\infty$, where this space is defined using the standard reproducing kernels and duality pairings for the spaces of holomorphic functions $B_p^{\sigma}(\mathbb{D})$,
\item and we show the existence of Hilbert function spaces $\mathcal{H}$ of
solutions to an elliptic PDE, to which our alternate Toeplitz corona theorem
applies, and which are \emph{not} spaces of holomorphic functions.
\end{enumerate}

\begin{remark}
After this paper appeared in print we were alerted to related work on Besov spaces by Kaptanoglu in \cite{Kap} and Beatrous and Burbea \cite{BeBu}.  In particular there is some overlap with the foundational material for Besov spaces appearing in this paper and these papers just cited.
\end{remark}

\section{Preliminaries}

We begin with a quick review of reproducing kernel Hilbert spaces, also known as 
Hilbert function spaces.

\subsection{Hilbert function spaces}

A Hilbert space $\mathcal{H}$ is said to be a \emph{Hilbert function space}
(also called a reproducing kernel Hilbert space) on a set $\Omega $ if the
elements of $\mathcal{H}$ are complex-valued functions $f$ on $\Omega $ with
the usual vector space structure, such that each point evaluation on $%
\mathcal{H}$ is a nonzero continuous linear functional, i.e. for every $x\in
\Omega $ there is a positive constant $C_{x}$ such that%
\begin{equation}
\left\vert f\left( x\right) \right\vert \leq C_{x}\left\Vert f\right\Vert _{%
\mathcal{H}},\ \ \ \ \ \forall f\in \mathcal{H},  \label{pointeval}
\end{equation}%
and there is some $f$ with $f\left( x\right) \neq 0$. Since point evaluation
at $x\in \Omega $ is a continuous linear functional, there is a unique
element $k_{x}\in \mathcal{H}$ such that%
\begin{equation*}
f\left( x\right) =\left\langle f,k_{x}\right\rangle _{\mathcal{H}}\text{ for
all }x\in \Omega .
\end{equation*}%
The element $k_{x}$ is called the reproducing kernel at $x$, and satisfies%
\begin{equation*}
k_{y}\left( x\right) =\left\langle k_{y},k_{x}\right\rangle _{\mathcal{H}},\
\ \ \ \ x,y\in \Omega .
\end{equation*}%
In particular we have%
\begin{equation*}
\left\Vert k_{x}\right\Vert _{\mathcal{H}}^{2}=\left\langle
k_{x},k_{x}\right\rangle _{\mathcal{H}}=k_{x}\left( x\right) ,
\end{equation*}%
and so the normalized reproducing kernel is given by $\widetilde{k_{x}}%
\equiv \frac{k_{x}}{\sqrt{k_{x}\left( x\right) }}$.

The function $k\left( y,x\right) \equiv \left\langle
k_{x},k_{y}\right\rangle _{\mathcal{H}}=k_{x}\left( y\right) $ is
self-adjoint ($k\left( x,y\right) =\overline{k\left( y,x\right) }$), and for
every finite subset $\left\{ x_{i}\right\} _{i=1}^{N}$ of $\Omega $, the
matrix $\left[ k\left( x_{i},x_{j}\right) \right] _{1\leq i,j\leq N}$ is
positive semidefinite:%
\begin{equation*}
\sum_{i,j=1}^{N}\xi _{i}\overline{\xi _{j}}k\left( x_{j},x_{i}\right)
=\sum_{i,j=1}^{N}\xi _{i}\overline{\xi _{j}}\left\langle
k_{x_{i}},k_{x_{j}}\right\rangle _{\mathcal{H}}=\left\langle
\sum_{i=1}^{N}\xi _{i}k_{x_{i}},\sum_{j=1}^{N}\xi _{j}k_{x_{j}}\right\rangle
_{\mathcal{H}}=\left\Vert \sum_{i=1}^{N}\xi _{i}k_{x_{i}}\right\Vert _{%
\mathcal{H}}^{2}\geq 0.
\end{equation*}%
Altogether we have shown that $k$ is a \emph{kernel function} in the
following sense.

\begin{definition}
\label{defker}A function $k:\Omega \times \Omega \rightarrow \mathbb{C}$ is
a \emph{kernel function} on $\Omega $ if $k$ is self-adjoint and positive on
the diagonal,\ and\ if for every finite subset $x=\left\{ x_{i}\right\}
_{i=1}^{N}\in \Omega ^{N}$ of $\Omega $, the matrix $\left[ k\left(
x_{i},x_{j}\right) \right] _{1\leq i,j\leq N}$ is positive semidefinite,
written $\left[ k\left( x_{i},x_{j}\right) \right] _{1\leq i,j\leq
N}\succcurlyeq 0$, i.e.%
\begin{equation}
\sum_{i,j=1}^{N}\xi _{i}\overline{\xi _{j}}k\left( x_{i},x_{j}\right) \geq
0,\ \ \ \ \ \xi \in \mathbb{C}^{N},x\in \Omega ^{N},N\geq 1.
\label{possemidef}
\end{equation}%
We write $k\succcurlyeq 0$ if $k$ is a kernel function.
\end{definition}

E. H. Moore discovered the following bijection between Hilbert function
spaces and kernel functions. Given a kernel function $k$ on $\Omega \times
\Omega $, define an inner product $\left\langle \cdot ,\cdot \right\rangle _{%
\mathcal{H}_{k}}$ on finite linear combinations $\sum_{i=1}^{N}\xi
_{i}k_{x_{i}}$ of the functions $k_{x_{i}}\left( \zeta \right) =k\left(
\zeta ,x_{i}\right) $, $\zeta \in \Omega $, by%
\begin{equation}
\left\langle \sum_{i=1}^{N}\xi _{i}k_{x_{i}},\sum_{j=1}^{N}\eta
_{j}k_{x_{j}}\right\rangle _{\mathcal{H}_{k}}=\sum_{i,j=1}^{N}\xi _{i}%
\overline{\eta _{j}}k\left( x_{j},x_{i}\right) .  \label{inpro}
\end{equation}%
If the forms in (\ref{possemidef}) are positive definite, then finite
collections of the kernel functions $k_{x_{i}}$ are linearly independent,
and the inner product in (\ref{inpro}) is well-defined. See \cite[page 19]%
{AgMc2} for a proof that (\ref{inpro}) is well-defined in general.

\begin{definition}
\label{defRKHS}Given a kernel function $k:\Omega \times \Omega \rightarrow 
\mathbb{C}$ on a set $\Omega $, define the associated Hilbert function space 
$\mathcal{H}_{k}$ to be the completion of the functions $\sum_{i=1}^{N}\xi
_{i}k_{x_{i}}$ under the norm corresponding to the inner product (\ref{inpro}%
).
\end{definition}

\begin{proposition}
\label{same kernel}The Hilbert space $\mathcal{H}_{k}$ has kernel $k$. If $%
\mathcal{H}$\ and $\mathcal{H}^{\prime }$ are Hilbert function spaces on $%
\Omega $ that have the same kernel function $k$, then there is an isometry
from $\mathcal{H}$ onto $\mathcal{H}^{\prime }$ that preserves the kernel
functions $k_{x}$, $x\in \Omega $.
\end{proposition}

We will need the notion of rescaling a kernel as given in \cite[page 25]%
{AgMc2}. Given a nonvanishing complex-valued function $\rho :\Omega
\rightarrow \mathbb{C\setminus }\left\{ 0\right\} $ and a kernel function $%
k\left( y,x\right) =k_{x}\left( y\right) $, define the $\rho $-rescaled
kernel $k^{\rho }$ by%
\begin{equation*}
k^{\rho }\left( y,x\right) =\rho \left( y\right) k\left( y,x\right) 
\overline{\rho \left( x\right) },\ \ \ \ \ x,y\in \Omega .
\end{equation*}%
It is easy to see that $k^{\rho }$ is self-adjoint and positive
semidefinite, and hence is a kernel function.  We refer to the associated
Hilbert function space $\mathcal{H}_{k^{\rho }}$ as the $\rho $-rescaling of
the Hilbert function space $\mathcal{H}_{k}$. A crucial choice of rescaling
for us below is the \emph{point} rescaling with $\rho =\frac{1}{\widetilde{%
k_{a}}}$ that results in $k_{a}^{\rho }\equiv 1$. We note that the $\delta $%
\ used in \cite{AgMc2} is our $\overline{\rho }$.

\subsubsection{Multipliers}

Let $\mathcal{H}=\mathcal{H}\left( \Omega \right) $ be a Hilbert function
space on a set $\Omega $. Let $L^{\infty }=L^{\infty }\left( \Omega \right) $
denote the space of bounded functions on $\Omega $ normed by the supremum
norm%
\begin{equation*}
\left\Vert h\right\Vert _{\infty }\equiv \sup_{x\in \Omega }\left\vert
h\left( x\right) \right\vert .
\end{equation*}%
The supremum norm is relevant here as point evaluations are continuous in $%
\mathcal{H}$, and so $\left\Vert h\right\Vert _{\infty }$ is a supremum of
moduli of continuous linear functionals. We define the space%
\begin{equation*}
\mathcal{H}^{\infty }=\mathcal{H}^{\infty }\left( \Omega \right) \equiv
\left\{ h\in \mathcal{H}:\left\Vert h\right\Vert _{\infty }<\infty \right\} =%
\mathcal{H}\cap L^{\infty }\left( \Omega \right)
\end{equation*}%
to consist of the bounded functions in $\mathcal{H}$, and we norm this space
by%
\begin{equation*}
\left\Vert h\right\Vert _{\mathcal{H}^{\infty }\left( \Omega \right) }\equiv
\max \left\{ \left\Vert h\right\Vert _{\mathcal{H}},\left\Vert h\right\Vert
_{\infty }\right\} ,
\end{equation*}%
so that $\mathcal{H}^{\infty }$ is a Banach space.

A function $\varphi $ is said to be a (pointwise) multiplier of $\mathcal{H}$
if $\varphi f\in \mathcal{H}$ for all $f\in \mathcal{H}$. The collection of
all multipliers of $\mathcal{H}$ is known to be a Banach algebra which we
denote by $M_{\mathcal{H}}$. Indeed, (see e.g. \cite{AgMc2}) if $\varphi $
is a multiplier of $\mathcal{H}$, and if we denote the linear operator of
multiplication by 
\begin{equation*}
\mathcal{M}_{\varphi }f\equiv \varphi f,
\end{equation*}%
then by the closed graph theorem 
\begin{equation*}
\left\Vert \varphi \right\Vert _{M_{\mathcal{H}}}\equiv \left\Vert \mathcal{M%
}_{\varphi }\right\Vert _{\mathcal{H}\rightarrow \mathcal{H}}\equiv
\sup_{f\in \mathcal{H}:\ f\neq 0}\frac{\left\Vert \varphi f\right\Vert _{%
\mathcal{H}}}{\left\Vert f\right\Vert _{\mathcal{H}}}<\infty .
\end{equation*}%
If in addition $1\in \mathcal{H}$ we have $\varphi \in \mathcal{H}$ and 
\begin{equation*}
\left\Vert \varphi \right\Vert _{\mathcal{H}}=\left\Vert \mathcal{M}%
_{\varphi }1\right\Vert _{\mathcal{H}}\leq \left\Vert \mathcal{M}_{\varphi
}\right\Vert _{\mathcal{H}\rightarrow \mathcal{H}}\left\Vert 1\right\Vert _{%
\mathcal{H}}\ .
\end{equation*}%
Finally and most importantly, $\varphi $ is bounded in $\Omega $ by $%
\left\Vert \mathcal{M}_{\varphi }\right\Vert _{\mathcal{H}\rightarrow 
\mathcal{H}}$, i.e. 
\begin{equation}
\left\Vert \varphi \right\Vert _{\infty }\leq \left\Vert \mathcal{M}%
_{\varphi }\right\Vert _{\mathcal{H}\rightarrow \mathcal{H}}.
\label{bound mult}
\end{equation}%
Indeed, for all $x\in \Omega $ we have 
\begin{equation*}
\left\vert \varphi \left( x\right) \right\vert \left\Vert k_{x}\right\Vert _{%
\mathcal{H}}^{2}=\left\vert \varphi \left( x\right) \right\vert k_{x}\left(
x\right) =\left\vert \left\langle \varphi k_{x},k_{x}\right\rangle _{%
\mathcal{H}}\right\vert \leq \left\Vert \varphi \right\Vert _{M_{\mathcal{H}%
}}\left\Vert k_{x}\right\Vert _{\mathcal{H}}^{2}.
\end{equation*}%
Moreover we have $\mathcal{M}_{\varphi }^{\ast }k_{x}=\overline{\varphi
\left( x\right) }k_{x}$ for all $x\in \Omega $ since%
\begin{equation*}
\left\langle f,\overline{\varphi \left( x\right) }k_{x}\right\rangle _{%
\mathcal{H}}=\varphi \left( x\right) \left\langle f,k_{x}\right\rangle _{%
\mathcal{H}}=\varphi \left( x\right) f\left( x\right) =\mathcal{M}_{\varphi
}f\left( x\right) =\left\langle \mathcal{M}_{\varphi }f,k_{x}\right\rangle _{%
\mathcal{H}}=\left\langle f,\mathcal{M}_{\varphi }^{\ast }k_{x}\right\rangle
_{\mathcal{H}}
\end{equation*}%
for all $f\in \mathcal{H}$. Thus we have shown that $M_{\mathcal{H}}$ embeds
in $\mathcal{H}^{\infty }$ with%
\begin{equation*}
\left\Vert \varphi \right\Vert _{\mathcal{H}^{\infty }}\leq \max \left\{
1,\left\Vert 1\right\Vert _{\mathcal{H}}\right\} \ \left\Vert \varphi
\right\Vert _{M_{\mathcal{H}}}\ .
\end{equation*}

\subsubsection{Kernel multipliers}

There is a Banach space $K_{\mathcal{H}}$ intermediate between $M_{\mathcal{H%
}}$ and $\mathcal{H}^{\infty }\left( \Omega \right) $ that plays a major
role in this paper, namely the Banach space $K_{\mathcal{H}}$ of \emph{kernel%
} multipliers consisting of all functions $\varphi $ on $\Omega $ for which 
\begin{equation*}
\left\Vert \varphi \right\Vert _{K_{\mathcal{H}}}\equiv \max \left\{
\left\Vert \varphi \widetilde{1}\right\Vert _{\mathcal{H}},\sup_{a\in \Omega
}\left\Vert \varphi \widetilde{k_{a}}\right\Vert _{\mathcal{H}}\right\}
<\infty ,
\end{equation*}%
where $\widetilde{1}$ is the constant function $\frac{1}{\left\Vert
1\right\Vert _{\mathcal{H}}}$ normalized to have $\mathcal{H}$-norm $1$. Let 
$\varphi \in K_{\mathcal{H}}$. Clearly, from $\widetilde{k_{a}}=\frac{k_{a}}{%
\sqrt{k_{a}\left( a\right) }}$ and the reproducing property of $k_{a}$, we
have 
\begin{equation*}
\left\vert \varphi \left( a\right) \right\vert =\frac{1}{k_{a}\left(
a\right) }\left\vert \left\langle \varphi k_{a},k_{a}\right\rangle _{%
\mathcal{H}}\right\vert =\left\vert \left\langle \varphi \widetilde{k_{a}},%
\widetilde{k_{a}}\right\rangle _{\mathcal{H}}\right\vert \leq \left\Vert
\varphi \widetilde{k_{a}}\right\Vert _{\mathcal{H}}\left\Vert \widetilde{%
k_{a}}\right\Vert _{\mathcal{H}}\leq \left\Vert \varphi \right\Vert _{K_{%
\mathcal{H}}}\ ,
\end{equation*}%
and so $K_{\mathcal{H}}$ embeds in $\mathcal{H}^{\infty }$ with 
\begin{equation*}
\left\Vert \varphi \right\Vert _{\mathcal{H}^{\infty }}\leq \max \left\{
1,\left\Vert 1\right\Vert _{\mathcal{H}}\right\} \ \left\Vert \varphi
\right\Vert _{K_{\mathcal{H}}}\ .
\end{equation*}%
Moreover, $M_{\mathcal{H}}$ embeds in $K_{\mathcal{H}}$ with $\left\Vert
\varphi \right\Vert _{K_{\mathcal{H}}}\leq \left\Vert \varphi \right\Vert
_{M_{\mathcal{H}}}$ since $\left\Vert \varphi \widetilde{k_{a}}\right\Vert _{%
\mathcal{H}}\leq \left\Vert \varphi \right\Vert _{M_{\mathcal{H}}}\left\Vert 
\widetilde{k_{a}}\right\Vert _{\mathcal{H}}=\left\Vert \varphi \right\Vert
_{M_{\mathcal{H}}}$ if $\varphi \in M_{\mathcal{H}}$. Thus we have the
embeddings%
\begin{equation*}
M_{\mathcal{H}}\hookrightarrow K_{\mathcal{H}}\hookrightarrow \mathcal{H}%
^{\infty }\hookrightarrow \mathcal{H\ },
\end{equation*}%
that show that the multiplier algebra $M_{\mathcal{H}}$ is contained in the
kernel multiplier space $K_{\mathcal{H}}$ which is contained in the space $%
\mathcal{H}^{\infty }$. Finally, we note that $M_{\mathcal{H}}$ multiplies
the spaces $K_{\mathcal{H}}$ and $\mathcal{H}^{\infty }$ as well as $%
\mathcal{H}$, i.e. that $M_{\mathcal{H}}$ is contained in both $M_{K_{%
\mathcal{H}}}$ and $M_{\mathcal{H}^{\infty }}$. Indeed, if $\varphi \in M_{%
\mathcal{H}}$ and $f\in K_{\mathcal{H}}$, then 
\begin{equation*}
\left\Vert \varphi f\right\Vert _{K_{\mathcal{H}}}=\max \left\{ \left\Vert
\varphi f\widetilde{1}\right\Vert _{\mathcal{H}},\sup_{a\in \Omega
}\left\Vert \varphi f\widetilde{k_{a}}\right\Vert _{\mathcal{H}}\right\}
\leq \left\Vert \varphi \right\Vert _{M_{\mathcal{H}}}\max \left\{
\left\Vert f\widetilde{1}\right\Vert _{\mathcal{H}},\sup_{a\in \Omega
}\left\Vert f\widetilde{k_{a}}\right\Vert _{\mathcal{H}}\right\} =\left\Vert
\varphi \right\Vert _{M_{\mathcal{H}}}\left\Vert f\right\Vert _{K_{\mathcal{H%
}}}\ ,
\end{equation*}%
and%
\begin{equation*}
\left\Vert \varphi f\right\Vert _{\mathcal{H}^{\infty }}=\max \left\{
\left\Vert \varphi f\right\Vert _{\mathcal{H}},\left\Vert \varphi
f\right\Vert _{\infty }\right\} \leq \left\Vert \varphi \right\Vert _{M_{%
\mathcal{H}}}\max \left\{ \left\Vert f\right\Vert _{\mathcal{H}},\left\Vert
f\right\Vert _{\infty }\right\} =\left\Vert \varphi \right\Vert _{M_{%
\mathcal{H}}}\left\Vert f\right\Vert _{\mathcal{H}^{\infty }}\ .
\end{equation*}

Of particular importance in this paper is the case when $K_{\mathcal{H}}$ is
an algebra. This occurs for example in the case $\mathcal{H}^{\infty }=M_{%
\mathcal{H}}$, as happens when $\mathcal{H}$ is the classical Hardy or
Bergman space on a bounded domain with $C^{2}$ boundary in $\mathbb{C}^{n}$.
We also note that $K_{\mathcal{H}}$ may be an
algebra even if $\mathcal{H}^{\infty }\neq M_{\mathcal{H}}$. For example, $%
K_{\mathcal{H}}$ is an algebra when $\mathcal{H}$ is any of the
Besov-Sobolev spaces $B_{2}^{\sigma }\left( \mathbb{B}_{n}\right) $, $\sigma
> 0$ and $n\geq 1$, of analytic functions on the ball $\mathbb{B}_{n}$.
See Subsection 7.1 below for this.

We recall at this point that the Corona Property has been proved for $M_{%
\mathcal{H}}$ when $\mathcal{H}=B_{2}^{\sigma }\left( \mathbb{B}_{n}\right) $
for $0\leq \sigma \leq \frac{1}{2}$ and $n\geq 1$; the case $n=1$ is in \cite%
{Car}, \cite{Tol}, \cite{ArBlPa} and \cite{Xia2}, and the case $n>1$ is in \cite{CoSaWi}.
In addition the Corona Property has been proved by Nicolau \cite{Nic} for
the algebra $\mathcal{H}^{\infty }$ when $\mathcal{H}=B_{2}^{0}\left( 
\mathbb{D}\right) $ is the classical Dirichlet space on the disk. In
Subsection 7.2 we use the Peter Jones solution \cite{Jo}, \cite{Jo2} to the
$\overline{\partial}$-equation in the unit disk $\mathbb{D}$, together with an adaptation of
the argument of Arcozzi, Blasi and Pau \cite{ArBlPa}, to prove the Corona
Property for the algebras $K_{B_{2}^{\sigma }\left( \mathbb{D}\right) }$. No
other corona theorems for the algebras $M_{\mathcal{H}}$, $K_{\mathcal{H}}$
or $\mathcal{H}^{\infty }$ are currently known when $\mathcal{H}%
=B_{2}^{\sigma }\left( \mathbb{B}_{n}\right) $. However, we will reduce the
Corona Property for the algebra of kernel multipliers $K_{\mathcal{H}}$
associated to $\mathcal{H}=B_{2}^{\sigma }\left( \mathbb{B}_{n}\right) $, to
a simpler property we call the Convex Poisson Property for $\mathcal{H}$.
See Theorem \ref{alternate} where this is shown to hold for more general
Hilbert function spaces $\mathcal{H}$ and their associated kernel multiplier
algebra $K_{\mathcal{H}}$.

\subsubsection{Shifted spaces and multiplier stability}

\begin{definition}
If $\mathcal{H}$ is a Hilbert function space on a set $\Omega $ with \emph{%
nonvanishing} kernel function $k$, then for each $a\in \Omega $, we define
the $a$-\emph{shifted} Hilbert space $\mathcal{H}^{a}$ to be $\mathcal{H}%
_{k^{\delta }}$ where $\delta =\frac{1}{\widetilde{k_{a}}}$, the $\frac{1}{%
\widetilde{k_{a}}}$-rescaling of $\mathcal{H}$, and where $\widetilde{k_{a}}=%
\frac{1}{\sqrt{k_{a}\left( a\right) }}k_{a}$ is the normalized reproducing
kernel for $\mathcal{H}$.
\end{definition}

\begin{lemma}
\label{weighted space}Let $\mathcal{H}$ be a Hilbert function space on a set 
$\Omega $ with nonvanishing kernel function $k$. Then the space $\mathcal{H}%
^{a}$ consists of those complex-valued functions $f$ on $\Omega $ such that $%
\widetilde{k_{a}}f\in \mathcal{H}$, and the inner product in $\mathcal{H}%
^{a} $ is given by 
\begin{equation*}
\left\langle f,g\right\rangle _{\mathcal{H}^{a}}=\left\langle \widetilde{%
k_{a}}f,\widetilde{k_{a}}g\right\rangle _{\mathcal{H}},\ \ \ \ \ f,g\in 
\mathcal{H}.
\end{equation*}%
In particular $\left\Vert f\right\Vert _{\mathcal{H}^{a}}=\left\Vert 
\widetilde{k_{a}}f\right\Vert _{\mathcal{H}}$.
\end{lemma}

\begin{proof}
Define $\mathcal{G}^{a}$ to be the linear space of functions on $\Omega $
having the form $\frac{1}{\widetilde{k_{a}}}h$ with $h\in \mathcal{H}$, and
define an inner product on $\mathcal{G}^{a}$ by $\left\langle
f,g\right\rangle _{\mathcal{G}^{a}}\equiv \left\langle \widetilde{k_{a}}f,%
\widetilde{k_{a}}g\right\rangle _{\mathcal{H}}$ for $f,g\in \mathcal{G}^{a}$%
. We have%
\begin{equation*}
k^{\delta }\left( \zeta ,\eta \right) =\frac{k\left( \zeta ,\eta \right) }{%
\widetilde{k_{a}}\left( \zeta \right) \overline{\widetilde{k_{a}}\left( \eta
\right) }}\text{ and so }k_{\eta }^{\delta }=\frac{1}{\overline{\widetilde{%
k_{a}}\left( \eta \right) }}\frac{k_{\eta }}{\widetilde{k_{a}}}\in \mathcal{G%
}^{a}.
\end{equation*}%
It is easy to see that $\mathcal{G}^{a}$ is complete in the norm derived
from this inner product, hence is a Hilbert space, and we now show that
point evaluations are continuous on $\mathcal{G}^{a}$. Indeed, if $f=\frac{1%
}{\widetilde{k_{a}}}h$ with $h\in \mathcal{H}$ then 
\begin{eqnarray*}
f\left( \eta \right) &=&\frac{1}{\widetilde{k_{a}}\left( \eta \right) }%
h\left( \eta \right) =\frac{1}{\widetilde{k_{a}}\left( \eta \right) }%
\left\langle h,k_{\eta }\right\rangle _{\mathcal{H}}=\frac{1}{\widetilde{%
k_{a}}\left( \eta \right) }\left\langle \widetilde{k_{a}}f,k_{\eta
}\right\rangle _{\mathcal{H}} \\
&=&\frac{1}{\widetilde{k_{a}}\left( \eta \right) }\left\langle \widetilde{%
k_{a}}f,\overline{\widetilde{k_{a}}\left( \eta \right) }\widetilde{k_{a}}%
k_{\eta }^{\delta }\right\rangle _{\mathcal{H}}=\left\langle \widetilde{k_{a}%
}f,\widetilde{k_{a}}k_{\eta }^{\delta }\right\rangle _{\mathcal{H}%
}=\left\langle f,k_{\eta }^{\delta }\right\rangle _{\mathcal{G}^{a}}
\end{eqnarray*}%
and so%
\begin{equation*}
\left\vert f\left( \eta \right) \right\vert \leq \left\Vert f\right\Vert _{%
\mathcal{G}^{a}}\left\Vert k_{\eta }^{\delta }\right\Vert _{\mathcal{G}%
^{a}}=\left\Vert f\right\Vert _{\mathcal{G}^{a}}\sqrt{\left\langle 
\widetilde{k_{a}}k_{\eta }^{\delta },\widetilde{k_{a}}k_{\eta }^{\delta
}\right\rangle _{\mathcal{H}}}=\left\Vert f\right\Vert _{\mathcal{G}^{a}}%
\sqrt{\left\langle \frac{k_{\eta }}{\overline{\widetilde{k_{a}}\left( \eta
\right) }},\frac{k_{\eta }}{\overline{\widetilde{k_{a}}\left( \eta \right) }}%
\right\rangle _{\mathcal{H}}}=\left\Vert f\right\Vert _{\mathcal{G}^{a}}%
\frac{\left\Vert k_{\eta }\right\Vert _{\mathcal{H}}}{\left\vert \widetilde{%
k_{a}}\left( \eta \right) \right\vert }.
\end{equation*}%
Thus $\mathcal{G}^{a}$ is a Hilbert function space on\ $\Omega $, and the
above calculation shows that the reproducing kernel for $\mathcal{G}^{a}$ is 
$k^{\delta }$. By Proposition \ref{same kernel} $\mathcal{G}^{a}=$ $\mathcal{%
H}^{a}$, and this completes the proof of the lemma.

We can give an alternate proof by computing that if $f=%
\sum_{i=1}^{J}x_{i}k_{\eta _{i}}^{\delta }\left( \zeta \right) \in \mathcal{H%
}^{a}$, then%
\begin{equation*}
\left\Vert f\right\Vert _{\mathcal{H}^{a}}^{2}=\left\langle
\sum_{i=1}^{J}x_{i}k_{\eta _{i}}^{\delta },\sum_{j=1}^{J}x_{j}k_{\eta
_{j}}^{\delta }\right\rangle _{\mathcal{H}^{a}}=\sum_{i,j=1}^{J}x_{i}%
\overline{x_{j}}\left\langle k_{\eta _{i}}^{\delta },k_{\eta _{j}}^{\delta
}\right\rangle _{\mathcal{H}^{a}}=\sum_{i,j=1}^{J}x_{i}\overline{x_{j}}%
k^{\delta }\left( \eta _{j},\eta _{i}\right) =\sum_{i,j=1}^{J}x_{i}\overline{%
x_{j}}\frac{k\left( \eta _{j},\eta _{i}\right) }{\overline{\widetilde{k_{a}}%
\left( \eta _{i}\right) }\widetilde{k_{a}}\left( \eta _{j}\right) }
\end{equation*}%
and 
\begin{eqnarray*}
\left\langle \widetilde{k_{a}}f,\widetilde{k_{a}}f\right\rangle _{\mathcal{H}%
} &=&\left\langle \sum_{i=1}^{J}x_{i}\widetilde{k_{a}}k_{\eta _{i}}^{\delta
},\sum_{j=1}^{J}x_{j}\widetilde{k_{a}}k_{\eta _{j}}^{\delta }\right\rangle _{%
\mathcal{H}}=\left\langle \sum_{i=1}^{J}x_{i}\widetilde{k_{a}}\frac{k_{\eta
_{i}}}{\overline{\widetilde{k_{a}}\left( \eta _{i}\right) }\widetilde{k_{a}}}%
,\sum_{j=1}^{J}x_{j}\widetilde{k_{a}}\frac{k_{\eta _{j}}}{\overline{%
\widetilde{k_{a}}\left( \eta _{j}\right) }\widetilde{k_{a}}}\right\rangle _{%
\mathcal{H}} \\
&=&\sum_{i,j=1}^{J}x_{i}\overline{x_{j}}\left\langle \frac{k_{\eta _{i}}}{%
\overline{\widetilde{k_{a}}\left( \eta _{i}\right) }},\frac{k_{\eta _{j}}}{%
\overline{\widetilde{k_{a}}\left( \eta _{j}\right) }}\right\rangle _{%
\mathcal{H}}=\sum_{i,j=1}^{J}x_{i}\overline{x_{j}}\frac{k_{\eta _{i}}\left(
\eta _{j}\right) }{\overline{\widetilde{k_{a}}\left( \eta _{i}\right) }%
\widetilde{k_{a}}\left( \eta _{j}\right) },
\end{eqnarray*}%
are equal. Now use that functions of the form $f=\sum_{i=1}^{J}x_{i}k_{\eta
_{i}}^{\delta }\left( \zeta \right) $ are dense in $\mathcal{H}^{a}$ by
definition.
\end{proof}

Despite the difference in norms of the shifted spaces $\mathcal{H}^{a}$, the
multiplier algebras coincide and have identical norms. This is proved in 
\cite[p.25]{AgMc2} where it is shown that rescaling a kernel leaves the
multiplier algebra and the multiplier norms unchanged. We give the simple
proof in our setting here.

\begin{lemma}
\label{stable}Let $\mathcal{H}$ be a Hilbert function space on a set $\Omega 
$ with nonvanishing kernel function. Then $M_{\mathcal{H}^{a}}=M_{\mathcal{H}%
}$ with equality of norms for all $a\in \Omega $.
\end{lemma}

\begin{proof}
Fix $a\in \Omega $. Suppose first that $\varphi \in M_{\mathcal{H}}$. We
claim that $\varphi \in M_{\mathcal{H}^{a}}$ with $\left\Vert \varphi
\right\Vert _{M_{\mathcal{H}^{a}}}\leq \left\Vert \varphi \right\Vert _{M_{%
\mathcal{H}}}$. Indeed, if $f\in \mathcal{H}^{a}$, then $f=\frac{1}{%
\widetilde{k_{a}}}g$ where $g\in \mathcal{H}$ with $\left\Vert g\right\Vert
_{\mathcal{H}}=\left\Vert f\right\Vert _{\mathcal{H}^{a}}$, and we have%
\begin{equation*}
\varphi f=\varphi \frac{1}{\widetilde{k_{a}}}g=\frac{1}{\widetilde{k_{a}}}%
\varphi g=\frac{1}{\widetilde{k_{a}}}G\ ,
\end{equation*}%
where $G\equiv \varphi g\in \mathcal{H}$ with $\left\Vert G\right\Vert _{%
\mathcal{H}}\leq \left\Vert \varphi \right\Vert _{M_{\mathcal{H}}}
\left\Vert g\right\Vert _{\mathcal{H}}$, and hence%
\begin{equation*}
\left\Vert \varphi f\right\Vert _{\mathcal{H}^{a}}=\left\Vert \widetilde{%
k_{a}}\varphi f\right\Vert _{\mathcal{H}}=\left\Vert G\right\Vert _{\mathcal{%
H}}\leq \left\Vert \varphi \right\Vert _{M_{\mathcal{H}}} \left\Vert
g\right\Vert _{\mathcal{H}}= \left\Vert \varphi \right\Vert _{M_{\mathcal{H}%
}} \left\Vert f\right\Vert _{\mathcal{H}^{a}}\ .
\end{equation*}%
This proves the claimed inequality: $\left\Vert \varphi \right\Vert _{M_{%
\mathcal{H}^{a}}}\leq \left\Vert \varphi \right\Vert _{M_{\mathcal{H}}}$.

Conversely, suppose that $\varphi \in M_{\mathcal{H}^{a}}$. We claim that $%
\varphi \in M_{\mathcal{H}}$ with $\left\Vert \varphi \right\Vert _{M_{%
\mathcal{H}}}\leq \left\Vert \varphi \right\Vert _{M_{\mathcal{H}^{a}}}$.
Indeed, if $g\in \mathcal{H}$, then $g=\widetilde{k_{a}}f$ where $f\in 
\mathcal{H}^{a}$ with $\left\Vert f\right\Vert _{\mathcal{H}^{a}}=\left\Vert
g\right\Vert _{\mathcal{H}}$, and we have%
\begin{equation*}
\varphi g=\varphi \widetilde{k_{a}}f=\widetilde{k_{a}}\varphi f=\widetilde{%
k_{a}}F
\end{equation*}%
where $F\equiv \varphi f \in \mathcal{H}^{a}$ with $\left\Vert F\right\Vert
_{\mathcal{H}^{a}}\leq \left\Vert \varphi \right\Vert _{M_{\mathcal{H}^{a}}}
\left\Vert f\right\Vert _{\mathcal{H}^{a}}$, and hence%
\begin{equation*}
\left\Vert \varphi g\right\Vert _{\mathcal{H}}=\left\Vert \widetilde{k_{a}}%
\varphi f\right\Vert _{\mathcal{H}}=\left\Vert F\right\Vert _{\mathcal{H}%
^{a}}\leq \left\Vert \varphi \right\Vert _{M_{\mathcal{H}^{a}}} \left\Vert
f\right\Vert _{\mathcal{H}^{a}}= \left\Vert \varphi \right\Vert _{M_{%
\mathcal{H}^{a}}}\left\Vert g\right\Vert _{\mathcal{H}}\ .
\end{equation*}
Hence we have: $\left\Vert \varphi \right\Vert _{M_{\mathcal{H}}}\leq
\left\Vert \varphi \right\Vert _{M_{\mathcal{H}^{a}}}$. These two
inequalities show that $M_{\mathcal{H}^{a}}=M_{\mathcal{H}}$ with equality
of norms.
\end{proof}

At this point we introduce the first main assumption needed for our
alternate Toeplitz corona theorem.

\begin{definition}
We say that the Hilbert function space $\mathcal{H}$ is \emph{multiplier
stable} if

\begin{enumerate}
\item the reproducing kernel functions $k_{x}$ are nonvanishing and are
invertible multipliers on $\mathcal{H}$, i.e. $k_{x}\in M_{\mathcal{H}}$ and 
$\frac{1}{k_{x}}\in M_{\mathcal{H}}$, for all $x\in \Omega $, and

\item the map $x\rightarrow k_{x}$ from $\Omega $ to $M_{\mathcal{H}}$ is
lower semicontinuous.
\end{enumerate}
\end{definition}

Note that we make no assumptions regarding the size of the norms of the
multipliers $k_{x}$ and $\frac{1}{k_{x}}$ in this definition. We will see
below that all the Besov-Sobolev spaces on the ball are multiplier stable,
as well as the Bergman and Hardy spaces on strictly pseudoconvex domains
with $C^{2}$ boundary. A crucial consequence of the multiplier stable
assumption is the $\mathcal{H}$-Poisson reproducing formula below.

\begin{lemma}
\label{H Poisson}Suppose $\mathcal{H}$ is a Hilbert function space on a set $%
\Omega $ with nonvanishing kernel and containing the constant functions.
Suppose furthermore that $k_{x}\in M_{\mathcal{H}}$ for all $x\in \Omega $.
Then for each $a\in \Omega $ we have the $\mathcal{H}$-Poisson reproducing
formula%
\begin{equation}
f\left( a\right) =\left\langle f,1\right\rangle _{\mathcal{H}^{a}},\ \ \ \ \
f\in \mathcal{H}\left( \Omega \right) ,\ a\in \Omega .  \label{H-reproducing}
\end{equation}
\end{lemma}

\begin{proof}
Since $k_{a}\in M_{\mathcal{H}}$ by hypothesis, the function $F_{a}\left(
w\right) \equiv \widetilde{k_{a}}\left( w\right) f\left( w\right) $ is in $%
\mathcal{H}$ for $f\in \mathcal{H}$, and so using Lemma \ref{weighted space}%
, 
\begin{equation*}
\sqrt{k_{a}\left( a\right) }f\left( a\right) =\widetilde{k_{a}}\left(
a\right) f\left( a\right) =F_{a}\left( a\right) =\left\langle
F_{a},k_{a}\right\rangle _{\mathcal{H}}=\sqrt{k_{a}\left( a\right) }%
\left\langle \widetilde{k_{a}}f,\widetilde{k_{a}}\right\rangle _{\mathcal{H}%
}=\sqrt{k_{a}\left( a\right) }\left\langle f,1\right\rangle _{\mathcal{H}%
^{a}},
\end{equation*}%
which gives (\ref{H-reproducing}).
\end{proof}

The spaces $\mathcal{H}^{a}$ are in fact all equal to $\mathcal{H}$ as sets,
with different but comparable norms (the constants of comparability need not
be bounded in $a$).

\begin{lemma}
\label{comp}Suppose $\mathcal{H}$ is multiplier stable. For $a\in \Omega $
we have comparability of the norms for $\mathcal{H}$ and $\mathcal{H}^{a}$:%
\begin{equation*}
\frac{1}{\left\Vert \frac{1}{\widetilde{k_{a}}}\right\Vert _{M_{\mathcal{H}}}%
}\left\Vert h\right\Vert _{\mathcal{H}}\leq \left\Vert h\right\Vert _{%
\mathcal{H}^{a}}\leq \left\Vert \widetilde{k_{a}}\right\Vert _{M_{\mathcal{H}%
}}\left\Vert h\right\Vert _{\mathcal{H}}\ ,\ \ \ \ \ h\in \mathcal{H}\cup 
\mathcal{H}^{a}.
\end{equation*}
\end{lemma}

\begin{proof}
Since $k_{a},\frac{1}{k_{a}}\in M_{\mathcal{H}}$ we see that 
\begin{eqnarray*}
\left\Vert f\right\Vert _{\mathcal{H}^{a}} &=&\left\Vert \widetilde{k_{a}}%
f\right\Vert _{\mathcal{H}}\leq \left\Vert \widetilde{k_{a}}\right\Vert _{M_{%
\mathcal{H}}}\left\Vert f\right\Vert _{\mathcal{H}}\ , \\
\left\Vert g\right\Vert _{\mathcal{H}} &=&\left\Vert \frac{1}{\widetilde{%
k_{a}}}\widetilde{k_{a}}g\right\Vert _{\mathcal{H}}\leq \left\Vert \frac{1}{%
\widetilde{k_{a}}}\right\Vert _{M_{\mathcal{H}}}\left\Vert \widetilde{k_{a}}%
g\right\Vert _{\mathcal{H}}=\left\Vert \frac{1}{\widetilde{k_{a}}}%
\right\Vert _{M_{\mathcal{H}}}\left\Vert g\right\Vert _{\mathcal{H}^{a}}\ ,
\end{eqnarray*}%
and these two inequalities prove the lemma.
\end{proof}

\subsection{Interpolation by rescalings}

\begin{definition}
Given a multiplier stable Hilbert function space $\mathcal{H}$ on\ $\Omega $
with kernel $k$, and $\left( \mathbf{a,\theta }\right) \in \Omega ^{M}\times %
\left[ 0,1\right] ^{M+1}$, define the Hilbert function space $\mathcal{H}^{%
\mathbf{a,\theta }}$ to be $\mathcal{H}$ with inner product given by%
\begin{equation*}
\left\langle f,g\right\rangle _{\mathcal{H}^{\mathbf{a,\theta }}}\equiv
\theta _{0}\left\langle f,g\right\rangle _{\mathcal{H}}+\sum_{m=1}^{M}\theta
_{m}\left\langle f,g\right\rangle _{\mathcal{H}^{a_{m}}},\ \ \ \ \ f,g\in 
\mathcal{H}.
\end{equation*}
\end{definition}

We recall from Lemma \ref{comp} that all of the spaces $\mathcal{H}^{a_{m}}$
are comparable, hence the inner product $\left\langle f,g\right\rangle _{%
\mathcal{H}^{\mathbf{a,\theta }}}$ is defined for $f,g\in \mathcal{H}$, and
all of the spaces $\mathcal{H}^{\mathbf{a,\theta }}$ are comparable with $%
\mathcal{H}$. We will often use the convention $k_{\eta }^{a_{0}}=k_{\eta }$
in order to simplify the sum above to $\left\langle f,g\right\rangle _{%
\mathcal{H}^{\mathbf{a,\theta }}}=\sum_{m=0}^{M}\theta _{m}\left\langle
f,g\right\rangle _{\mathcal{H}^{a_{m}}}$. We will refer to the spaces $%
\mathcal{H}^{\mathbf{a,\theta }}$ as the \emph{convex shifted} spaces
associated with $\mathcal{H}$. They are normed by $\left\Vert f\right\Vert _{%
\mathcal{H}^{\mathbf{a,\theta }}}=\sqrt{\left\langle f,f\right\rangle _{%
\mathcal{H}^{\mathbf{a,\theta }}}}$.

Let 
\begin{equation*}
\Sigma _{M}\equiv \left\{ \mathbf{\theta }=\left\{ \theta _{m}\right\}
_{m=0}^{M}\subset \left[ 0,1\right] ^{M+1}:\sum\limits_{j=0}^{M}\theta
_{j}=1\right\}
\end{equation*}%
denote the unit $\left( M+1\right) $-dimensional simplex.

\begin{description}
\item[Assumption] We now make the standing assumption, in force for the
remainder of the paper, that $\mathcal{H}$ contains the constant functions
on $\Omega $, and upon multiplying by a positive constant, we may assume
that 
\begin{equation*}
\left\Vert 1\right\Vert _{\mathcal{H}}=1.
\end{equation*}
\end{description}

\begin{corollary}
\label{constant 1} The norm of the constant function $1$ in the space $%
\mathcal{H}^{\mathbf{a,\theta }}$ is $1$ for all $\mathbf{a}\in \Omega ^{M}$
and $\mathbf{\theta }\in \Sigma _{M}$.
\end{corollary}

\begin{proof}
We have%
\begin{equation*}
\left\Vert 1\right\Vert _{\mathcal{H}^{\mathbf{a,\theta }}}^{2}=\left\langle
1,1\right\rangle _{\mathcal{H}^{\mathbf{a,\theta }}}=\sum_{m=0}^{M}\theta
_{m}\left\langle 1,1\right\rangle _{\mathcal{H}^{a_{m}}}=\sum_{m=0}^{M}%
\theta _{m}\left\langle \widetilde{k_{a_{m}}},\widetilde{k_{a_{m}}}%
\right\rangle _{\mathcal{H}}=\sum_{m=0}^{M}\theta _{m}=1.
\end{equation*}
\end{proof}

\subsubsection{Vector-valued norms}

Let $N\geq 1$ be fixed. Given Hilbert spaces $\mathcal{H}_{\ell }$ for $%
1\leq \ell \leq N$, define a complete inner product on the direct sum $%
\oplus _{\ell =1}^{N}\mathcal{H}_{\ell }$ by%
\begin{equation*}
\left\langle f,g\right\rangle _{\oplus _{\ell =1}^{N}\mathcal{H}_{\ell
}}\equiv \sum_{\ell =1}^{N}\left\langle f_{\ell },g_{\ell }\right\rangle _{%
\mathcal{H}_{\ell }}\ ,\ \ \ \ \ f=\left( f_{\ell }\right) _{\ell
=1}^{N},g=\left( g_{\ell }\right) _{\ell =1}^{N}\in \oplus _{\ell =1}^{N}%
\mathcal{H}_{\ell }\ .
\end{equation*}%
When all the spaces $\mathcal{H}_{\ell }$ are equal to the same space $%
\mathcal{H}$, we write simply $\oplus ^{N}\mathcal{H}$ in place of $\oplus
_{\ell =1}^{N}\mathcal{H}_{\ell }$.

Given Banach spaces $B_{\ell }$ for $1\leq \ell \leq N$, define a complete
norm on the direct sum $\oplus _{\ell =1}^{N}B_{\ell }$ by%
\begin{equation*}
\left\Vert \varphi \right\Vert _{\oplus _{\ell =1}^{N}B_{\ell }}\equiv
\max_{1\leq \ell \leq N}\left\Vert \varphi _{\ell }\right\Vert _{B_{\ell }}\
,\ \ \ \ \ \varphi =\left( \varphi _{\ell }\right) _{\ell =1}^{N}\in \oplus
_{\ell =1}^{N}B_{\ell }\ ,
\end{equation*}%
and again, when all the spaces $B_{\ell }$ are equal to the same space $B$,
we write simply $\oplus ^{N}B$ in place of $\oplus _{\ell =1}^{N}B_{\ell }$.
In the case when the $B_{\ell }$ are also Hilbert spaces, this definition
differs from the previous one, but the intended definition should always be
clear from the context. We will be mainly concerned with the cases $B=M_{%
\mathcal{H}}$ and $B=K_{\mathcal{H}}$.

When $B=M_{\mathcal{H}}$ is the multiplier algebra of a Hilbert function
space $\mathcal{H}$ on a set $\Omega $, there are two additional natural
norms to consider, namely the row and column norms, to which we now turn.
Let $\mathcal{H}$ be a Hilbert function space on a set $\Omega $. For $%
\varphi \in \oplus ^{N}M_{\mathcal{H}}$ define 
\begin{eqnarray*}
\mathcal{M}_{\varphi } &:&\oplus ^{N}\mathcal{H}\rightarrow \mathcal{H}\text{
by }\mathcal{M}_{\varphi }f=\sum_{\alpha =1}^{N}\varphi _{\alpha }f_{\alpha
}, \\
\mathbb{M}_{\varphi } &:&\mathcal{H}\rightarrow \oplus ^{N}\mathcal{H}\text{
by }\mathbb{M}_{\varphi }h=\left( \varphi _{\alpha }h\right) _{\alpha
=1}^{N}.
\end{eqnarray*}%
Then we have $\mathcal{M}_{\varphi }^{\ast }h=\left( \mathcal{M}_{\varphi
_{\alpha }}^{\ast }f\right) _{\alpha =1}^{N}$, and in particular,%
\begin{equation}
\mathcal{M}_{\varphi }^{\ast }\widetilde{k_{z}}=\left( \mathcal{M}_{\varphi
_{\alpha }}^{\ast }\widetilde{k_{z}}\right) _{\alpha =1}^{N}=\left( 
\overline{\varphi _{\ell }\left( z\right) }\widetilde{k_{z}}\right) _{\alpha
=1}^{N}\ .  \label{in part}
\end{equation}%
We also define the row and column norms by%
\begin{equation*}
\left\Vert \mathcal{M}_{\varphi }\right\Vert _{\limfunc{op}}\equiv
\sup_{f\neq 0}\frac{\left\Vert \mathcal{M}_{\varphi }f\right\Vert _{\mathcal{%
H}}}{\left\Vert f\right\Vert _{\oplus ^{N}\mathcal{H}}}\text{ and }%
\left\Vert \mathbb{M}_{\varphi }\right\Vert _{\limfunc{op}}\equiv
\sup_{h\neq 0}\frac{\left\Vert \mathbb{M}_{\varphi }h\right\Vert _{\oplus
^{N}\mathcal{H}}}{\left\Vert h\right\Vert _{\mathcal{H}}}.
\end{equation*}%
The two key inequalities involving these norms are:%
\begin{eqnarray*}
\left\Vert \varphi \right\Vert _{L^{\infty }\left( \ell _{N}^{2}\right) }
&\equiv &\sup_{z\in \Omega }\left( \sum_{\ell =1}^{N}\left\vert \varphi
_{\ell }\left( z\right) \right\vert ^{2}\right) ^{\frac{1}{2}}\leq
\left\Vert \mathcal{M}_{\varphi }\right\Vert _{\limfunc{op}}\ , \\
\left\Vert \varphi \right\Vert _{\oplus ^{N}\mathcal{H}} &\leq &\left\Vert 
\mathbb{M}_{\varphi }\right\Vert _{\limfunc{op}}\left\Vert 1\right\Vert _{%
\mathcal{H}}\ .
\end{eqnarray*}%
The first inequality follows from (\ref{in part}), 
\begin{eqnarray*}
\sum_{\ell =1}^{N}\left\vert \varphi _{\ell }\left( z\right) \right\vert
^{2} &=&\sum_{\ell =1}^{N}\left\vert \varphi _{\ell }\left( z\right)
\right\vert ^{2}\left\Vert \widetilde{k_{z}}\right\Vert _{\mathcal{H}%
}^{2}=\sum_{\ell =1}^{N}\left\Vert \overline{\varphi _{\ell }\left( z\right) 
}\widetilde{k_{z}}\right\Vert _{\mathcal{H}}^{2}=\left\Vert \mathcal{M}%
_{\varphi }^{\ast }\widetilde{k_{z}}\right\Vert _{\oplus ^{N}\mathcal{H}}^{2}
\\
&\leq &\left\Vert \mathcal{M}_{\varphi }^{\ast }\right\Vert _{\limfunc{op}%
}^{2}\left\Vert \widetilde{k_{z}}\right\Vert _{\mathcal{H}}^{2}=\left\Vert 
\mathcal{M}_{\varphi }\right\Vert _{\limfunc{op}}^{2}\left\Vert \widetilde{%
k_{z}}\right\Vert _{\mathcal{H}}^{2}=\left\Vert \mathcal{M}_{\varphi
}\right\Vert _{\limfunc{op}}^{2},
\end{eqnarray*}%
and the second inequality follows from%
\begin{eqnarray*}
\left\Vert \varphi \right\Vert _{\oplus ^{N}\mathcal{H}}^{2} &=&\sum_{\ell
=1}^{N}\left\Vert \varphi _{\ell }\right\Vert _{\mathcal{H}}^{2}=\sum_{\ell
=1}^{N}\left\Vert \varphi _{\ell }1\right\Vert _{\mathcal{H}}^{2}=\sum_{\ell
=1}^{N}\left\Vert \mathcal{M}_{\varphi _{\ell }}1\right\Vert _{\mathcal{H}%
}^{2}=\left\Vert \mathbb{M}_{\varphi }1\right\Vert _{\oplus ^{N}\mathcal{H}%
}^{2} \\
&\leq &\left\Vert \mathbb{M}_{\varphi }\right\Vert _{\limfunc{op}%
}^{2}\left\Vert 1\right\Vert _{\mathcal{H}}^{2}\ .
\end{eqnarray*}%
Then from our assumption that the norm of $1$ in the space $\mathcal{H}$ is $%
1$, we have both of the inequalities $\left\Vert \varphi \right\Vert
_{L^{\infty }\left( \ell ^{2}\right) }\leq \left\Vert \mathcal{M}_{\varphi
}\right\Vert _{\limfunc{op}}$ and $\left\Vert \varphi \right\Vert _{\oplus
^{N}\mathcal{H}}\leq \left\Vert \mathbb{M}_{\varphi }\right\Vert _{\limfunc{%
op}}$. We now have three norms on the Banach space $\oplus ^{N}M_{\mathcal{H}%
}$.

\begin{definition}
Given $\varphi \in \oplus ^{N}M_{\mathcal{H}}$, define the three norms 
\begin{equation*}
\left\Vert \varphi \right\Vert _{\oplus ^{N}M_{\mathcal{H}}}^{\limfunc{row}%
}\equiv \left\Vert \mathcal{M}_{\varphi }\right\Vert _{\limfunc{op}}\ ,\ \ \
\left\Vert \varphi \right\Vert _{\oplus ^{N}M_{\mathcal{H}}}^{\limfunc{column%
}}\equiv \left\Vert \mathbb{M}_{\varphi }\right\Vert _{\limfunc{op}}\ ,\ \ \
\left\Vert \varphi \right\Vert _{\oplus ^{N}M_{\mathcal{H}}}^{\max }\equiv
\max_{1\leq \ell \leq N}\left\Vert \varphi _{\ell }\right\Vert _{M_{\mathcal{%
H}}}\ .
\end{equation*}
\end{definition}

These norms are comparable since 
\begin{eqnarray*}
\left\Vert \varphi \right\Vert _{\oplus ^{N}M_{\mathcal{H}}}^{\max } &\leq
&\min \left\{ \left\Vert \varphi \right\Vert _{\oplus ^{N}M_{\mathcal{H}}}^{%
\limfunc{row}},\left\Vert \varphi \right\Vert _{\oplus ^{N}M_{\mathcal{H}}}^{%
\limfunc{column}}\right\} \\
&\leq &\max \left\{ \left\Vert \varphi \right\Vert _{\oplus ^{N}M_{\mathcal{H%
}}}^{\limfunc{row}},\left\Vert \varphi \right\Vert _{\oplus ^{N}M_{\mathcal{H%
}}}^{\limfunc{column}}\right\} \leq \sqrt{N}\left\Vert \varphi \right\Vert
_{\oplus ^{N}M_{\mathcal{H}}}^{\max }\ .
\end{eqnarray*}%
When $B=K_{\mathcal{H}}$ is the Banach space of kernel multipliers on $%
\mathcal{H}$, and in the presence of the standing assumption $\left\Vert
1\right\Vert _{\mathcal{H}}=1$, we will use the following natural norm on
the direct sum $\oplus _{\ell =1}^{N}K_{\mathcal{H}}$.

\begin{definition}
\label{def H inf}Let $\mathcal{H}$ be a Hilbert function space on a set $%
\Omega $. For $N\geq 1$ and $\varphi \in \oplus _{\ell =1}^{N}K_{\mathcal{H}%
} $ define the following norm on $\oplus ^{N}K_{\mathcal{H}}$ :%
\begin{equation*}
\left\Vert \varphi \right\Vert _{\oplus ^{N}K_{\mathcal{H}}}\equiv \max
\left\{ \left\Vert \varphi \right\Vert _{\oplus ^{N}\mathcal{H}},\sup_{a\in
\Omega }\left\Vert \varphi \widetilde{k_{a}}\right\Vert _{\oplus ^{N}%
\mathcal{H}}\right\} .
\end{equation*}
\end{definition}

In the event that $K_{\mathcal{H}}=M_{\mathcal{H}}$ isometrically, then the
norm just introduced on $\oplus ^{N}K_{\mathcal{H}}$ is comparable to the
three introduced on $\oplus ^{N}M_{\mathcal{H}}$ above:%
\begin{equation}
\left\Vert \varphi \right\Vert _{\oplus ^{N}K_{\mathcal{H}}}\leq \left\Vert
\varphi \right\Vert _{\oplus ^{N}M_{\mathcal{H}}}^{\limfunc{column}}\leq 
\sqrt{N}\left\Vert \varphi \right\Vert _{\oplus ^{N}M_{\mathcal{H}}}^{\max
}\leq \sqrt{N}\left\Vert \varphi \right\Vert _{\oplus ^{N}K_{\mathcal{H}}}\ .
\label{K and M}
\end{equation}

\subsection{A characterization of rescaling}

We end this section on preliminaries with a characterization of when two
kernels are rescalings of each other. For this we begin with a quick review
of relevant properties of positive matrices. Recall that an $n\times n$
self-adjoint matrix $A$ of complex numbers is said to be \emph{positive},
denoted $A\succcurlyeq 0$, if all of its eigenvalues are nonnegative; and
said to be \emph{strictly positive}, denoted $A\succ 0$, if all of its
eigenvalues are positive. Clearly sums of positive matrices are positive.
Moreover, given any self-adjoint $A$, there is a unitary matrix $U$ such
that $UAU^{\ast }=\limfunc{Diag}\left( \lambda _{1},\ldots ,\lambda
_{n}\right) =\left[ 
\begin{array}{cccc}
\lambda _{1} & 0 & \cdots & 0 \\ 
0 & \lambda _{2} & \ddots & \vdots \\ 
\vdots & \ddots & \ddots & 0 \\ 
0 & \cdots & 0 & \lambda _{n}%
\end{array}%
\right] $, with $\lambda _{j}\in \mathbb{R}$. A \emph{dyad} is a rank one
matrix of the form $v\otimes v^{\ast }=\left[ 
\begin{array}{cccc}
\left\vert v_{1}\right\vert ^{2} & v_{1}\overline{v_{2}} & \cdots & v_{1}%
\overline{v_{n}} \\ 
v_{2}\overline{v_{1}} & \left\vert v_{2}\right\vert ^{2} & \ddots & \vdots
\\ 
\vdots & \ddots & \ddots & v_{n-1}\overline{v_{n}} \\ 
v_{n}\overline{v_{1}} & \cdots & v_{n}\overline{v_{n-1}} & \left\vert
v_{n}\right\vert ^{2}%
\end{array}%
\right] $. Every dyad is positive, and conversely, every positive matrix $A$
is a sum $\sum_{i=1}^{I}\alpha _{i}v_{i}\otimes v_{i}^{\ast }$ of dyads $%
v_{i}\otimes v_{i}^{\ast }$ with nonnegative coefficients $\alpha _{i}$.
However, such decompositions are not in general unique. For example, any
positive sum of dyads is a positive matrix $A$, and the spectral theorem for 
$A$ gives in general a different decomposition into dyads with pairwise
orthogonal vectors (the eigenvectors of the matrix $A$). One important
consequence is that the Schur product of two positive matrices is positive.
Indeed, if $A=\sum_{i=1}^{I}\alpha _{i}v_{i}\otimes v_{i}^{\ast }$ and $%
B=\sum_{j=1}^{J}\beta _{j}w_{j}\otimes w_{j}^{\ast }$, then $A\circ
B=\sum_{i=1}^{I}\sum_{j=1}^{J}\alpha _{i}\beta _{j}\left( v_{i}\otimes
v_{i}^{\ast }\right) \circ \left( w_{j}\otimes w_{j}^{\ast }\right) $, and
it is easily verified that $\left( v_{i}\otimes v_{i}^{\ast }\right) \circ
\left( w_{j}\otimes w_{j}^{\ast }\right) =\left( v_{i}\circ w_{i}\right)
\otimes \left( v_{j}\circ w_{j}\right) ^{\ast }$.

\bigskip

Now we turn to the problem of deciding when two self-adjoint nonvanishing
functions\thinspace $K\left( x,y\right) $ and $k\left( x,y\right) $ on a
product set $\Omega \times \Omega $ are rescalings of each other. The
surprisingly simple answer depends only on the $2\times 2$ and $3\times 3$
principal submatrices of the infinite matrices $\left[ K\left( x,y\right) %
\right] _{\left( x,y\right) \in \Omega \times \Omega }$ and $\left[ k\left(
x,y\right) \right] _{\left( x,y\right) \in \Omega \times \Omega }$.

\begin{proposition}
\label{self-adjoint rescaling}Suppose that $K$ and $k$ are two self-adjoint
nonvanishing functions on a product set $\Omega \times \Omega $. Then $K$
and $k$ are rescalings of each other, i.e. there is a nonvanishing function $%
\psi $ on $\Omega $ such that $K=\left( \psi \otimes \psi ^{\ast }\right)
\circ k$ if and only if the following two conditions hold:

\begin{enumerate}
\item $\frac{\left\vert K\left( x,y\right) \right\vert ^{2}}{K\left(
x,x\right) K\left( y,y\right) }=\frac{\left\vert k\left( x,y\right)
\right\vert ^{2}}{k\left( x,x\right) k\left( y,y\right) }$ for all $x,y\in
\Omega $,

\item $\arg \frac{K\left( x,z\right) }{k\left( x,z\right) }=\arg \frac{%
K\left( x,y\right) }{k\left( x,y\right) }+\arg \frac{K\left( y,z\right) }{%
k\left( y,z\right) }$ $\left( \func{mod}2\pi \right) $ for all $x,y,z\in
\Omega $.
\end{enumerate}
\end{proposition}

\begin{proof}
If $K=\left( \psi \otimes \psi ^{\ast }\right) \circ k$, then $K\left(
x,y\right) =\psi \left( x\right) k\left( x,y\right) \overline{\psi \left(
y\right) }$ and we have%
\begin{equation*}
\frac{\left\vert K\left( x,y\right) \right\vert ^{2}}{K\left( x,x\right)
K\left( y,y\right) }=\frac{\left\vert \psi \left( x\right) k\left(
x,y\right) \overline{\psi \left( y\right) }\right\vert ^{2}}{\psi \left(
x\right) k\left( x,x\right) \overline{\psi \left( x\right) }\psi \left(
y\right) k\left( y,y\right) \overline{\psi \left( y\right) }}=\frac{%
\left\vert k\left( x,y\right) \right\vert ^{2}}{k\left( x,x\right) k\left(
y,y\right) },
\end{equation*}%
and%
\begin{eqnarray*}
\arg \frac{K\left( x,z\right) }{k\left( x,z\right) } &=&\arg \psi \left(
x\right) \overline{\psi \left( z\right) }=\arg \psi \left( x\right)
\left\vert \psi \left( y\right) \right\vert ^{2}\overline{\psi \left(
z\right) } \\
&=&\arg \psi \left( x\right) \overline{\psi \left( y\right) }+\arg \psi
\left( y\right) \overline{\psi \left( z\right) } \\
&=&\arg \frac{K\left( x,y\right) }{k\left( x,y\right) }+\arg \frac{K\left(
y,z\right) }{k\left( y,z\right) }.
\end{eqnarray*}

Conversely assume that both conditions (1) and (2) hold. Then given any
finite set of points $\left\{ x_{j}\right\} _{j=1}^{J}$ in $\Omega $, define
a matrix%
\begin{equation*}
U\equiv \left[ u_{ij}\right] _{i,j=1}^{J};\ \ \ \ \ u_{ij}=\frac{\frac{%
K\left( x_{i},x_{j}\right) }{\sqrt{K\left( x_{i},x_{i}\right) }\sqrt{K\left(
x_{j},x_{j}\right) }}}{\frac{k\left( x_{i},x_{j}\right) }{\sqrt{k\left(
x_{i},x_{i}\right) }\sqrt{k\left( x_{j},x_{j}\right) }}}.
\end{equation*}
and note that $U$ is self-adjoint since $K$ and $k$ are, and that $U$ is
unimodular by condition (1). If we set $u_{ij}=e^{i\theta _{ij}}$, then
condition (2) says that%
\begin{eqnarray}
\theta _{i\ell } &=&\arg \frac{K\left( x_{i},x_{\ell }\right) }{k\left(
x_{i},x_{\ell }\right) }=\arg \frac{K_{i\ell }}{k_{i\ell }}=\arg \frac{K_{ij}%
}{k_{ij}}+\arg \frac{K_{j\ell }}{k_{j\ell }}  \label{theta il} \\
&=&\arg \frac{K\left( x_{i},x_{j}\right) }{k\left( x_{i},x_{j}\right) }+\arg 
\frac{K\left( x_{j},x_{\ell }\right) }{k\left( x_{j},x_{\ell }\right) }%
=\theta _{ij}+\theta _{j\ell }\ .  \notag
\end{eqnarray}%
Now define $\theta _{1}=0$ and $\theta _{j}=\theta _{1j}$ for $1\leq j\leq N$%
, so that $\theta _{j}=-\theta _{j1}$ since $U$ is self-adjoint. Then $%
\theta _{1i}-\theta _{1j}=-\theta _{i1}-\theta _{1j}=-\theta _{ij}$ by (\ref%
{theta il}), and so with $\psi \left( x_{j}\right) \equiv e^{-i\theta _{j}}%
\frac{\sqrt{K\left( x_{j},x_{j}\right) }}{\sqrt{k\left( x_{j},x_{j}\right) }}
$, we obtain%
\begin{eqnarray*}
\psi \left( x_{i}\right) k\left( x_{i},x_{j}\right) \overline{\psi \left(
x_{j}\right) } &=&e^{-i\left( \theta _{1i}-\theta _{1j}\right) }\left\{ 
\frac{\sqrt{K\left( x_{i},x_{i}\right) }}{\sqrt{k\left( x_{i},x_{i}\right) }}%
k\left( x_{i},x_{j}\right) \frac{\sqrt{K\left( x_{j},x_{j}\right) }}{\sqrt{%
k\left( x_{j},x_{j}\right) }}\right\} \\
&=&e^{i\theta _{ij}}\left\{ \frac{1}{u_{ij}}K\left( x_{i},x_{j}\right)
\right\} =K\left( x_{i},x_{j}\right) .
\end{eqnarray*}%
Note that the vector $\left\{ \psi \left( x_{i}\right) \right\} _{i=1}^{J}$
is determined uniquely up to a unimodular constant.

Finally, we apply any appropriate form of transfinite induction. From what
we have done above, we get a \emph{consistent} definition of $\psi $ on an
increasing maximal chain of subsets of $\Omega $ since two dyads $\psi
_{1}\otimes \psi _{1}^{\ast }$ and $\psi _{2}\otimes \psi _{2}^{\ast }$ are
equal if and only if there is a unimodular \emph{constant} $e^{i\theta }$
such that $\psi _{1}=e^{i\theta }\psi _{2}$. Now we apply Zorn's Lemma to
get a nonvanishing function $\psi $ on $\Omega $ such that $K\left(
x,y\right) =\psi \left( x\right) k\left( x,y\right) \overline{\psi \left(
y\right) }$. If both kernels are holomorphic in their first variable $x$ and
antiholomorphic in their second variable $y$, then $\psi \left( x\right) 
\overline{\psi \left( y\right) }=\frac{K\left( x,y\right) }{k\left(
x,y\right) }$ is holomorphic and nonvanishing in $x$, which proves that $%
\psi \left( x\right) $ is holomorphic and nonvanishing in $x$. Finally, the
supremum bounds follow from the formula $\left\vert \psi \left( x\right)
\right\vert =\frac{\sqrt{K\left( x,x\right) }}{\sqrt{k\left( x,x\right) }}$.
\end{proof}

\begin{remark}
Condition (1) is equivalent to the equality of distance functions $%
d_{k}\left( x,y\right) =d_{K}\left( x,y\right) $, $x,y\in \Omega $, where
for any kernel $k$ on $\Omega $, the distance function $d_{k}$ is defined by 
\begin{equation*}
d_{k}\left( x,y\right) \equiv \sqrt{1-\frac{\left\vert k\left( x,y\right)
\right\vert ^{2}}{k\left( x,x\right) k\left( y,y\right) }}.
\end{equation*}%
Note that $d_{k}\left( x,y\right) =\sin \theta _{x,y}$ where $\theta _{x,y}$
is the angle between $k_{x}$ and $k_{y}$ in the Hilbert function space $%
\mathcal{H}_{k}$.
\end{remark}

\section{The Bezout kernel multiplier characterization}

The following definition will be used to characterize when we can solve
Bezout's equation with a vector in the space $\oplus _{\ell =1}^{N}K_{%
\mathcal{H}}$ of kernel multipliers of $\mathcal{H}$.

\begin{definition}
\label{RKCP}Let $\mathcal{H}$ be a Hilbert function space on a set $\Omega $
with nonvanishing kernel, and let $\mathcal{H}^{a}$ be the shifted Hilbert
space for $a\in \Omega $. We say that a vector $\varphi \in \oplus _{\ell
=1}^{N}L^{\infty }\left( \Omega \right) $ satisfies the\emph{\ }$\mathcal{H}%
- $\emph{convex Poisson condition} with positive constant $C$ if for every
finite collection of points $\mathbf{a}=\left( a_{1},\ldots ,a_{M}\right)
\in \Omega ^{M}$ and every collection of nonnegative numbers $\mathbf{\theta 
}=\left\{ \theta _{m}\right\} _{m=0}^{M}$ summing to $1=\sum%
\limits_{m=0}^{M}\theta _{m}$, there is a vector $g^{\mathbf{a},\mathbf{%
\theta }}\in \oplus _{\ell =1}^{N}\mathcal{H}$ satisfying%
\begin{eqnarray}
\varphi \left( z\right) \cdot g^{\mathbf{a},\mathbf{\theta }}\left( z\right)
&=&1,\ \ \ \ \ z\in \Omega ,  \label{Constant RKCP} \\
\left\Vert g^{\mathbf{a},\mathbf{\theta }}\right\Vert _{\oplus _{\ell =1}^{N}%
\mathcal{H}^{\mathbf{a,\theta }}}^{2} &=&\theta _{0}\left\Vert g^{\mathbf{a},%
\mathbf{\theta }}\right\Vert _{\oplus _{\ell =1}^{N}\mathcal{H}%
}^{2}+\sum_{m=1}^{M}\theta _{m}\left\Vert g^{\mathbf{a},\mathbf{\theta }%
}\right\Vert _{\oplus _{\ell =1}^{N}\mathcal{H}^{a_{m}}}^{2}\leq C^{2}\ . 
\notag
\end{eqnarray}%
We denote the smallest such constant $C$ by $\left\Vert \varphi \right\Vert
_{cPc}$.
\end{definition}

\subsection{Kernel multiplier solutions}

Here now is our abstract characterization of solutions to Bezout's equation,
which is of primary interest in those cases where the space of kernel
multipliers $K_{\mathcal{H}}$ is an algebra. But first we require additional
structure on our Hilbert function space $\mathcal{H}$ to substitute for
Montel's theorem in complex analysis. This leads to the second main
assumption needed for our alternate Toeplitz corona theorem.

\begin{definition}
\label{def Montel}Let $\Omega $ be a topological space. A Hilbert function
space $\mathcal{H}$ of continuous functions on $\Omega $ is said to be have
the \emph{Montel property} if there is a dense subset $S$ of $\Omega $ with
the property that for every sequence $\left\{ f_{n}\right\} _{n=1}^{\infty }$
in the unit ball of $\mathcal{H}$, there are a subsequence $\left\{
f_{n_{k}}\right\} _{k=1}^{\infty }$ and a function $g$ in the unit ball of $%
\mathcal{H}$, such that%
\begin{equation*}
\lim_{k\rightarrow \infty }f_{n_{k}}\left( x\right) =g\left( x\right) ,\ \ \
\ \ x\in S.
\end{equation*}
\end{definition}

A main ingredient in the proof is the following minimax lemma of von
Neumann, proved for example in \cite{Gam}. This lemma was introduced in this
context by Amar \cite{Amar}, and used subsequently by Trent and Wick \cite%
{TrWi} as well.

\begin{lemma}
\label{von Neumann}Suppose that $M$ is a convex compact subset of a normed
linear space, and that $P$ is a convex subset of a vector space. Let $%
\mathcal{F}:M\times P\rightarrow \left[ 0,\infty \right) $ satisfy\bigskip 
\newline
(\textbf{1}) for each fixed $p\in P$, the section $\mathcal{F}^{p}$ given by 
$\mathcal{F}^{p}\left( m\right) =\mathcal{F}\left( m,p\right) $ is concave
and continuous,\bigskip \newline
(\textbf{2}) for each fixed $m\in M$, the section $\mathcal{F}_{m}$ given by 
$\mathcal{F}_{m}\left( p\right) =\mathcal{F}\left( m,p\right) $ is
convex.\bigskip \newline
Then the following minimax equality holds:%
\begin{equation*}
\sup_{m\in M}\inf_{p\in P}\mathcal{F}\left( m,p\right) =\inf_{p\in
P}\sup_{m\in M}\mathcal{F}\left( m,p\right) .
\end{equation*}
\end{lemma}

\begin{theorem}
\label{RKCP implies Adult Toe}Let $\mathcal{H}$ be a multiplier stable
Hilbert function space of continuous functions on a separable topological
space $\Omega $, and assume that $\mathcal{H}$ has the Montel property.
Suppose that $\varphi \in \oplus _{\ell =1}^{N}L^{\infty }\left( \Omega
\right) $, and that $C>0$ is a positive constant. Then there is a vector
function $f\in \oplus ^{N}K_{\mathcal{H}}$ satisfying%
\begin{eqnarray}
\varphi \left( z\right) \cdot f\left( z\right) &=&1,\ \ \ \ \ z\in \Omega ,
\label{control} \\
\left\Vert f\right\Vert _{\oplus ^{N}K_{\mathcal{H}}} &\leq &C\ ,  \notag
\end{eqnarray}%
\emph{if and only if} $\varphi $ satisfies the $\mathcal{H}-$convex Poisson
condition in Definition \ref{RKCP} with constant $\left\Vert \varphi
\right\Vert _{\mathcal{H}-cPc}\leq C$.
\end{theorem}

\begin{description}
\item[Porism] If we drop the lower semicontinuity assumption (2) in the
definition of multiplier stability, then the proof below shows that if $%
\varphi \in \oplus _{\ell =1}^{N}L^{\infty }\left( \Omega \right) $
satisfies the $\mathcal{H}-$convex Poisson condition in Definition \ref{RKCP}%
, then there is a vector function $f\in \oplus ^{N}\mathcal{H}^{\infty }$
satisfying%
\begin{eqnarray*}
\varphi \left( z\right) \cdot f\left( z\right) &=&1,\ \ \ \ \ z\in \Omega ,
\\
\left\Vert f\right\Vert _{\oplus ^{N}\mathcal{H}^{\infty }} &\leq
&\left\Vert \varphi \right\Vert _{\mathcal{H}-cPc}\ .
\end{eqnarray*}%
However, there is no converse assertion here in general.
\end{description}

\begin{proof}
Fix for the moment an integer $M\geq 1$ and a positive constant $\varepsilon
_{M}$ to be chosen later. Now fix $\mathbf{a}=\left\{ a_{m}\right\}
_{m=1}^{M}\subset \Omega $ and consider the simplex%
\begin{equation*}
\Sigma _{M}\equiv \left\{ \mathbf{\theta }=\left\{ \theta _{m}\right\}
_{m=0}^{M}\subset \left[ 0,1\right] ^{M+1}:\sum\limits_{j=0}^{M}\theta
_{j}=1\right\} .
\end{equation*}%
For each $\mathbf{\theta }\in \Sigma _{M}$ pick $g^{\mathbf{a},\mathbf{%
\theta }}\in \mathcal{H}^{\mathbf{a},\mathbf{\theta }}\left( \Omega \right) $
as in the $cPc$ Definition \ref{RKCP} for $\varphi $, i.e.%
\begin{eqnarray*}
\varphi (z)\cdot g^{\mathbf{a},\mathbf{\theta }}(z) &=&1\text{ in }\Omega ,
\\
\left\Vert g^{\mathbf{a},\mathbf{\theta }}\right\Vert _{\oplus _{\ell =1}^{N}%
\mathcal{K}^{\mathbf{a},\mathbf{\theta }}\left( \Omega \right) } &\leq
&\left\Vert \varphi \right\Vert _{cPc}\ .
\end{eqnarray*}%
Note that 
\begin{equation}
\left\Vert g^{\mathbf{a},\mathbf{\theta }}\right\Vert _{\oplus _{\ell =1}^{N}%
\mathcal{H}^{\mathbf{a},\mathbf{\theta }^{\prime }}\left( \Omega \right)
}<\infty ,\ \ \ \ \ \text{for all }\mathbf{\theta },\mathbf{\theta }^{\prime
}\in \Sigma _{M}\ ,  \label{reason}
\end{equation}%
since as observed earlier, it follows from Lemma \ref{comp} that all of the
interpolating spaces $\mathcal{H}^{\mathbf{a},\mathbf{\theta }}\left( \Omega
\right) $ are comparable with $\mathcal{H}$. It is here that we use the full
force of our assumption that $\mathcal{H}$ is multiplier stable, as opposed
to merely assuming that the kernel functions $k_{a}$ are nonvanishing
multipliers. Indeed, with only the latter assumption, an element of $%
\mathcal{H}^{a}$ has the form $\frac{1}{\widetilde{k_{a}}}g$ for $g\in 
\mathcal{H}$, and then $\left\Vert g\right\Vert _{\mathcal{H}^{a^{\prime
}}}^{2}=\left\Vert \frac{\widetilde{k_{a^{\prime }}}}{\widetilde{k_{a}}}%
g\right\Vert _{\mathcal{H}}^{2}$ could be infinite.

Thus $\mathcal{H}^{\mathbf{a,\theta }}\subset \mathcal{H}$, and we can
define the set 
\begin{equation*}
\mathcal{C}_{\mathbf{a},M}\equiv \limfunc{convex}\limfunc{hull}\left\{ g^{%
\mathbf{a},\mathbf{\theta }}:\mathbf{\theta }\in \Sigma _{M}\left(
\varepsilon _{M}\right) \right\}
\end{equation*}%
to be the convex hull in $\oplus _{\ell =1}^{N}\mathcal{H}$ of these Bezout
solutions $g^{\mathbf{a},\mathbf{\theta }}$. We will apply the von Neumann
minimax equality in Lemma \ref{von Neumann} to the functional%
\begin{equation*}
\mathcal{F}_{\mathbf{a}}\left( \mathbf{\theta },f\right) \equiv \left\Vert
f\right\Vert _{\oplus _{\ell =1}^{N}\mathcal{H}^{\mathbf{a,\theta }%
}}^{2}=\theta _{0}\left\Vert f\right\Vert _{\oplus _{\ell =1}^{N}\mathcal{H}%
}^{2}+\sum_{m=1}^{M}\theta _{m}\left\Vert f\right\Vert _{\oplus _{\ell
=1}^{N}\mathcal{H}^{a_{m}}}^{2},
\end{equation*}%
defined for $\mathbf{\theta }\in \Sigma _{M}$, and for $f\in \mathcal{C}_{%
\mathbf{a},M}$, noting that $\mathcal{F}_{\mathbf{a}}\left( \mathbf{\theta }%
,f\right) $ is then finite by (\ref{reason}).

Both $\Sigma _{M}$ and $\mathcal{C}_{\mathbf{a},M}$ are convex, and in
addition $\Sigma _{M}$ is compact. The functional $\mathcal{F}_{\mathbf{a}%
}\left( \mathbf{\theta },f\right) $ is linear and continuous in $\mathbf{%
\theta }$, hence also concave. It is also convex in $f$ since if $\lambda
_{i}\geq 0$ and $\sum_{i=1}^{L}\lambda _{i}=1$ and $A_{i}\in \mathcal{C}_{%
\mathbf{a},M}$, then%
\begin{eqnarray*}
\mathcal{F}_{\mathbf{a}}\left( \mathbf{\theta },\sum_{i=1}^{L}\lambda
_{i}A_{i}\right) &=&\left\Vert \sum_{i=1}^{L}\lambda _{i}A_{i}\right\Vert
_{\oplus _{\ell =1}^{N}\mathcal{H}^{\mathbf{a},\mathbf{\theta }}\left(
\Omega \right) }^{2}\leq \left( \sum_{i=1}^{L}\lambda _{i}\left\Vert
A_{i}\right\Vert _{\oplus _{\ell =1}^{N}\mathcal{H}^{\mathbf{a},\mathbf{%
\theta }}\left( \Omega \right) }\right) ^{2} \\
&\leq &\sum_{i=1}^{L}\lambda _{i}\left\Vert A_{i}\right\Vert _{\oplus _{\ell
=1}^{N}\mathcal{H}^{\mathbf{a},\mathbf{\theta }}\left( \Omega \right)
}^{2}=\sum_{i=1}^{L}\lambda _{i}\mathcal{F}_{\mathbf{a}}\left( \mathbf{%
\theta },A_{i}\right) .
\end{eqnarray*}

Thus the minimax equality in Lemma \ref{von Neumann} applies to give%
\begin{eqnarray*}
\inf_{f\in \mathcal{C}_{\mathbf{a},M}}\sup_{\mathbf{\theta }\in \Sigma _{M}}%
\mathcal{F}_{\mathbf{a}}\left( \mathbf{\theta },f\right) &=&\sup_{\mathbf{%
\theta }\in \Sigma _{M}}\inf_{f\in \mathcal{C}_{\mathbf{a},M}}\mathcal{F}_{%
\mathbf{a}}\left( \mathbf{\theta },f\right) \\
&\leq &\sup_{\mathbf{\theta }\in \Sigma _{M}}\mathcal{F}_{\mathbf{a}}\left( 
\mathbf{\theta },g^{\mathbf{a},\mathbf{\theta }}\right) \\
&=&\sup_{\mathbf{\theta }\in \Sigma _{M}}\left\Vert g^{\mathbf{a},\mathbf{%
\theta }}\right\Vert _{\oplus _{\ell =1}^{N}\mathcal{H}^{\mathbf{a},\mathbf{%
\theta }}\left( \Omega \right) }^{2}\leq \left\Vert \varphi \right\Vert
_{cPc}^{2}\ .
\end{eqnarray*}%
So for each $M\geq 1$, we can pick $f^{\left( M\right) }\in \mathcal{C}_{%
\mathbf{a},M}\subset \mathcal{H}^{\mathbf{a},\mathbf{\theta }}\left( \Omega
\right) $ so that $f^{\left( M\right) }$ is almost optimal for $\inf_{f\in 
\mathcal{C}_{\mathbf{a},M}}$ in the display above, more precisely, 
\begin{equation*}
\sup_{\mathbf{\theta }\in \Sigma _{M}}\mathcal{F}_{\mathbf{a}}\left( \mathbf{%
\theta },f^{\left( M\right) }\right) <\left( 1+\varepsilon _{M}\right)
\left\Vert \varphi \right\Vert _{cPc}^{2}.
\end{equation*}%
We also have $\varphi (z)\cdot f^{\left( M\right) }(z)=1$. Let $\mathbf{e}%
_{j}=\left( \theta _{0},\theta _{1},\ldots ,\theta _{M}\right) \in \Sigma
_{M}$ where $\theta _{j}=1$ and $\theta _{k}=0$ for $k\neq j$. Then we have 
\begin{equation}
\left\Vert f^{\left( M\right) }\right\Vert _{\oplus _{\ell =1}^{N}\mathcal{H}%
\left( \Omega \right) }^{2}=\mathcal{F}_{\mathbf{a}}\left( \mathbf{e}%
_{0},f^{\left( M\right) }\right) \leq \sup_{\mathbf{\theta }\in \Sigma _{M}}%
\mathcal{F}_{\mathbf{a}}\left( \mathbf{\theta },f^{\left( M\right) }\right)
\leq \left( 1+\varepsilon _{M}\right) \left\Vert \varphi \right\Vert
_{cPc}^{2}\ ,  \label{H2 normal}
\end{equation}%
and for each $\ell \geq 1$ we have,%
\begin{equation}
\left\Vert f^{\left( M\right) }\right\Vert _{\oplus _{\ell =1}^{N}\mathcal{H}%
^{a_{\ell }}\left( \Omega \right) }^{2}=\left\Vert f^{\left( M\right)
}\right\Vert _{\oplus _{\ell =1}^{N}\mathcal{H}^{a_{\ell }}\left( \Omega
\right) }^{2}\leq \mathcal{F}_{\mathbf{a}}\left( \widetilde{\mathbf{e}_{\ell
}},f^{\left( M\right) }\right) \leq \sup_{\mathbf{\theta }\in \Sigma _{M}}%
\mathcal{F}_{\mathbf{a}}\left( \mathbf{\theta },f^{\left( M\right) }\right)
\leq \left( 1+\varepsilon _{M}\right) \left\Vert \varphi \right\Vert
_{cPc}^{2}\ .  \label{pointwise control}
\end{equation}%
Now we use the multiplier stable hypothesis together with the $\mathcal{H}$%
-Poisson reproducing formula (\ref{H-reproducing}) in Lemma \ref{H Poisson},
and then the Cauchy-Schwarz inequality for inner products, to obtain the
pointwise estimates%
\begin{eqnarray*}
\left\vert f^{\left( M\right) }\left( a_{m}\right) \right\vert ^{2}
&=&\left\vert \left\langle f^{\left( M\right) },1\right\rangle _{\oplus
_{\ell =1}^{N}\mathcal{H}^{a_{m}}}\right\vert ^{2}\leq \left\langle
f^{\left( M\right) },f^{\left( M\right) }\right\rangle _{\oplus _{\ell
=1}^{N}\mathcal{H}^{a_{m}}}\left\langle 1,1\right\rangle _{\oplus _{\ell
=1}^{N}\mathcal{H}^{a_{m}}} \\
&=&\left\Vert f^{\left( M\right) }\right\Vert _{\oplus _{\ell =1}^{N}%
\mathcal{H}^{a_{m}}\left( \Omega \right) }^{2}\left\langle \widetilde{%
k_{a_{m}}},\widetilde{k_{a_{m}}}\right\rangle _{\oplus _{\ell =1}^{N}%
\mathcal{H}}\leq \left( 1+\varepsilon _{M}\right) \left\Vert \varphi
\right\Vert _{cPc}^{2}\ .
\end{eqnarray*}

Now by separability of $\Omega $, we can choose a dense set $S=\left\{
a_{m}\right\} _{m=1}^{\infty }$ in $\Omega $ as in Definition \ref{def
Montel}, and then choose the positive constants $\varepsilon _{M}$ so small
that 
\begin{equation*}
\lim_{M\rightarrow \infty }\varepsilon _{M}=0.
\end{equation*}%
From (\ref{H2 normal}) and the fact that $\mathcal{H}$, and so also $\oplus
_{\ell =1}^{N}\mathcal{H}$, has the Montel property, we conclude that there
is a vector function $f\in \oplus _{\ell =1}^{N}\mathcal{H}$ with 
\begin{equation*}
\left\Vert f\right\Vert _{\oplus _{\ell =1}^{N}\mathcal{H}}^{2}\leq
\lim_{M\rightarrow \infty }\left( 1+\varepsilon _{M}\right) \left\Vert
\varphi \right\Vert _{cPc}^{2}=\left\Vert \varphi \right\Vert _{cPc}^{2}
\end{equation*}%
and 
\begin{equation*}
f\left( a_{m}\right) =\lim_{M\rightarrow \infty }f^{\left( M\right) }\left(
a_{m}\right) ,\ \ \ \ \ m\geq 1.
\end{equation*}%
Thus we have%
\begin{equation*}
1=\lim_{M\rightarrow \infty }1=\lim_{M\rightarrow \infty }\varphi \left(
a_{m}\right) \cdot f^{\left( M\right) }\left( a_{m}\right) =\varphi \left(
a_{m}\right) \cdot \lim_{M\rightarrow \infty }f^{\left( M\right) }\left(
a_{m}\right) =\varphi \left( a_{m}\right) \cdot f\left( a_{m}\right) 
\end{equation*}%
for all $m\geq 1$, and hence $f\in \oplus _{\ell =1}^{N}\mathcal{H}$ is a
solution to the Bezout equation $\varphi \left( z\right) \cdot f\left(
z\right) =1$ for $z\in \Omega $ since $\varphi \cdot f$ is continuous and $S$
is dense. We also have from (\ref{pointwise control}) the pointwise estimate 
\begin{equation*}
\left\vert f\left( a_{m}\right) \right\vert =\lim_{M\rightarrow \infty
}\left\vert f^{\left( M\right) }\left( a_{m}\right) \right\vert \leq \lim
\sup_{M\rightarrow \infty }\left( 1+\varepsilon _{M}\right) \left\Vert
\varphi \right\Vert _{cPc}=\left\Vert \varphi \right\Vert _{cPc}
\end{equation*}%
for each $m\geq 1$. Since $f\in \oplus _{\ell =1}^{N}\mathcal{H}$ satisfies $%
\left\Vert f\right\Vert _{\oplus _{\ell =1}^{N}\mathcal{H}}\leq \left\Vert
\varphi \right\Vert _{cPc}$ and is continuous in $\Omega $, and since $%
S=\left\{ a_{m}\right\} _{m=1}^{\infty }$ is dense in $\Omega $, we thus
have 
\begin{equation*}
\left\Vert f\right\Vert _{\oplus _{\ell =1}^{N}\mathcal{H}^{\infty }}=\max
\left\{ \left\Vert f\right\Vert _{\oplus _{\ell =1}^{N}\mathcal{H}%
},\left\Vert f\right\Vert _{L^{\infty }\left( \ell _{N}^{2}\right) }\right\}
\leq \left\Vert \varphi \right\Vert _{cPc}\ ,
\end{equation*}%
since $\left\Vert f\right\Vert _{L^{\infty }\left( \ell _{N}^{2}\right)
}=\sup_{z\in \Omega }\left\vert f\left( z\right) \right\vert $.

Finally, using the lower semicontinuity assumption (2) in the definition of
multiplier stability (which has not been needed until now), we also have the
stronger estimate%
\begin{equation}
\left\Vert f\right\Vert _{\oplus _{\ell =1}^{N}K_{\mathcal{H}}}=\max \left\{
\left\Vert f\right\Vert _{\mathcal{H}},\sup_{a\in \Omega }\left\Vert f%
\widetilde{k_{a}}\right\Vert _{\mathcal{H}}\right\} =\max \left\{ \left\Vert
f\right\Vert _{\mathcal{H}},\sup_{a\in \Omega }\left\Vert f\right\Vert _{%
\mathcal{H}^{a}}\right\} \leq \left\Vert \varphi \right\Vert _{cPc}.
\label{stronger estimate}
\end{equation}%
Indeed, upon letting $M\rightarrow \infty $ in (\ref{pointwise control}), we
obtain%
\begin{equation*}
\max \left\{ \left\Vert f\right\Vert _{\mathcal{H}},\sup_{1\leq m<\infty
}\left\Vert f\right\Vert _{\mathcal{H}^{a_{m}}}\right\} \leq \left\Vert
\varphi \right\Vert _{cPc}.
\end{equation*}%
The map $a\rightarrow \widetilde{k_{a}}$ is lower semicontinuous from $%
\Omega $ to $M_{\mathcal{H}}$ by assumption, and it follows that $\left\Vert
f\right\Vert _{\mathcal{H}^{a}}=\left\Vert f\widetilde{k_{a}}\right\Vert _{%
\mathcal{H}}=\left\Vert \mathcal{M}_{\widetilde{k_{a}}}f\right\Vert _{%
\mathcal{H}}$ is a lower semicontinuous function of $a\in \Omega $. The
proof of (\ref{stronger estimate}) is now completed using that $S=\left\{
a_{m}\right\} _{m=1}^{\infty }$ is dense in $\Omega $, so that $\left\Vert
f\right\Vert _{\mathcal{H}^{a}}\leq \lim \inf_{a_{m}\rightarrow a}\left\Vert
f\right\Vert _{\mathcal{H}^{a_{m}}}\leq \left\Vert \varphi \right\Vert
_{cPc} $ for all $a\in \Omega $.

The converse assertion is straightforward. Suppose that $f\in \oplus _{\ell
=1}^{N}K_{\mathcal{H}}$ solves $\varphi \cdot f=1$ in $\Omega $. Then $%
\left\Vert f\widetilde{k_{a_{m}}}\right\Vert _{\oplus _{\ell =1}^{N}\mathcal{%
H}}\leq \left\Vert f\right\Vert _{\oplus _{\ell =1}^{N}K_{\mathcal{H}}}$ and
so 
\begin{eqnarray*}
\left\Vert f\right\Vert _{\oplus _{\ell =1}^{N}\mathcal{H}^{\mathbf{a,\theta 
}}}^{2} &=&\theta _{0}\left\Vert f\right\Vert _{\oplus _{\ell =1}^{N}%
\mathcal{H}}^{2}+\sum_{m=1}^{M}\theta _{m}\left\Vert f\right\Vert _{\oplus
_{\ell =1}^{N}\mathcal{H}^{a_{m}}}^{2} \\
&=&\theta _{0}\left\Vert f\right\Vert _{\oplus _{\ell =1}^{N}\mathcal{H}%
}^{2}+\sum_{m=1}^{M}\theta _{m}\left\Vert f\widetilde{k_{a_{m}}}\right\Vert
_{\oplus _{\ell =1}^{N}\mathcal{H}}^{2} \\
&\leq &\theta _{0}\left\Vert f\right\Vert _{\oplus _{\ell =1}^{N}K_{\mathcal{%
H}}}^{2}+\sum_{m=1}^{M}\theta _{m}\left\Vert f\right\Vert _{\oplus _{\ell
=1}^{N}K_{\mathcal{H}}}^{2}=\left\Vert f\right\Vert _{\oplus _{\ell
=1}^{N}K_{\mathcal{H}}}^{2}\ .
\end{eqnarray*}
\end{proof}

\section{The alternate Toeplitz corona theorem}

We begin by establishing notation, in particular various corona properties.

\subsection{Corona Properties}

Fix $N\in \mathbb{N}$ throughout this subsection. Given a Hilbert function
space $\mathcal{H}$ on a set $\Omega $, denote by $\oplus ^{N}K_{\mathcal{H}%
} $ and $\oplus ^{N}M_{\mathcal{H}}$ the direct sum of the kernel multiplier
spaces $K_{\mathcal{H}}$ and the multiplier algebras $M_{\mathcal{H}}$
respectively, equipped with the norms $\left\Vert \varphi \right\Vert
_{\oplus ^{N}K_{\mathcal{H}}}$ and $\left\Vert \varphi \right\Vert _{\oplus
^{N}M_{\mathcal{H}}}^{\limfunc{column}}$ respectively, as introduced in a
previous subsection. We often write $\left\Vert \varphi \right\Vert _{\oplus
^{N}M_{\mathcal{H}}}=\left\Vert \varphi \right\Vert _{\oplus ^{N}M_{\mathcal{%
H}}}^{\limfunc{column}}$ from now on. Next we define two `baby' properties:

\begin{enumerate}
\item the first solves $\varphi \cdot f=k_{a}$ with $f\in \mathcal{H}$ for
all $a\in \Omega $ and all $\varphi \in K_{\mathcal{H}}$ in the kernel
multiplier space,

\item the second solves $\varphi \cdot f=h$ with $f\in \mathcal{H}$ for all $%
h\in \mathcal{H}$ and all $\varphi \in M_{\mathcal{H}}$ in the multiplier
algebra.
\end{enumerate}

Note that we typically define \textbf{conditions} that a vector $\varphi $
of corona data might have, and we typically define \textbf{properties} that
a space might have in terms of these conditions.

\subsubsection{Kernel Corona Property}

\begin{definition}
\label{BKP}Suppose $\mathcal{H}$ is a Hilbert function space on a set $%
\Omega $.

\begin{enumerate}
\item Let $C>0$. We say that a vector $\varphi =\left( \varphi _{1},\ldots
,\varphi _{N}\right) \in \oplus ^{N}L^{\infty }\left( \Omega \right) $
satisfies the $\mathcal{H-}$\emph{kernel corona\textbf{\ }condition} with
constant $C$ if for each $a\in \Omega $ there are $f_{1},\ldots ,f_{N}\in 
\mathcal{H}$ satisfying 
\begin{eqnarray}
\left\Vert f\right\Vert _{\oplus ^{N}\mathcal{H}}^{2} &=&\left\Vert
f_{1}\right\Vert _{\mathcal{H}}^{2}+\cdots +\left\Vert f_{N}\right\Vert _{%
\mathcal{H}}^{2}\leq C^{2}\left\Vert k_{a}\right\Vert _{\mathcal{H}}^{2},
\label{corcon k} \\
\left( \varphi \cdot f\right) \left( z\right) &=&\varphi _{1}\left( z\right)
f_{1}\left( z\right) +\cdots +\varphi _{N}\left( z\right) f_{N}\left(
z\right) =k_{a}\left( z\right) ,\ \ \ \ \ z\in \Omega .  \notag
\end{eqnarray}

\item Let $c,C>0$. We say that $\mathcal{H}$ has the \emph{Kernel Corona%
\textbf{\ }Property} with constants $c,C$ if for every vector $\varphi
=\left( \varphi _{1},\ldots ,\varphi _{N}\right) \in \oplus ^{N}K_{\mathcal{H%
}}$ satisfying both 
\begin{equation*}
\left\Vert \varphi \right\Vert _{\oplus ^{N}K_{\mathcal{H}}}\leq 1
\end{equation*}%
and 
\begin{equation}
\left\vert \varphi _{1}\left( z\right) \right\vert ^{2}+\cdots +\left\vert
\varphi _{N}\left( z\right) \right\vert ^{2}\geq c^{2}>0,\ \ \ \ \ z\in
\Omega ,  \label{lowerphi}
\end{equation}%
the vector $\varphi =\left( \varphi _{1},\ldots ,\varphi _{N}\right) $
satisfies the $\mathcal{H}$-kernel corona condition with constant $C$.
\end{enumerate}
\end{definition}

\subsubsection{Baby Corona Property}

\begin{definition}
\label{BCP}Suppose $\mathcal{H}$ is a Hilbert function space on a set $%
\Omega $.

\begin{enumerate}
\item Let $C>0$. We say that a vector $\varphi =\left( \varphi _{1},\ldots
,\varphi _{N}\right) \in \oplus ^{N}L^{\infty }\left( \Omega \right) $
satisfies the $\mathcal{H-}$\emph{baby corona\textbf{\ }condition} with
constant $C$ if for each $h\in \mathcal{H}$ there are $f_{1},\ldots
,f_{N}\in \mathcal{H}$ satisfying 
\begin{eqnarray}
\left\Vert f\right\Vert _{\oplus ^{N}\mathcal{H}}^{2} &=&\left\Vert
f_{1}\right\Vert _{\mathcal{H}}^{2}+\cdots +\left\Vert f_{N}\right\Vert _{%
\mathcal{H}}^{2}\leq C^{2}\left\Vert h\right\Vert _{\mathcal{H}}^{2},
\label{corcon} \\
\left( \varphi \cdot f\right) \left( z\right) &=&\varphi _{1}\left( z\right)
f_{1}\left( z\right) +\cdots +\varphi _{N}\left( z\right) f_{N}\left(
z\right) =h\left( z\right) ,\ \ \ \ \ z\in \Omega .  \notag
\end{eqnarray}

\item Let $c,C>0$. We say that $\mathcal{H}$ has the \emph{Baby Corona
Property }with constants $c,C$ if for every vector $\varphi =\left( \varphi
_{1},\ldots ,\varphi _{N}\right) \in \oplus ^{N}M_{\mathcal{H}}$ satisfying
both 
\begin{equation*}
\left\Vert \varphi \right\Vert _{\oplus ^{N}M_{\mathcal{H}}}\leq 1
\end{equation*}%
and 
\begin{equation*}
\left\vert \varphi _{1}\left( z\right) \right\vert ^{2}+\cdots +\left\vert
\varphi _{N}\left( z\right) \right\vert ^{2}\geq c^{2}>0,\ \ \ \ \ z\in
\Omega ,
\end{equation*}%
the vector $\varphi =\left( \varphi _{1},\ldots ,\varphi _{N}\right) $
satisfies the $\mathcal{H}$-baby corona condition with constant $C$.
\end{enumerate}
\end{definition}

\subsubsection{Corona Property}

Given a Hilbert function space $\mathcal{H}$ on $\Omega $, and an algebra $%
\mathcal{A}$ contained in $\mathcal{H}^{\infty }=\mathcal{H}\cap L^{\infty
}\left( \Omega \right) $, we define the Corona Property for the algebra $%
\mathcal{A}$ as follows. There is a small abuse of notation here since the
definition we give will depend on\thinspace $N$ and the norm $\left\Vert
\cdot \right\Vert _{\oplus ^{N}\mathcal{A}}$ that we use for the direct sum $%
\oplus ^{N}\mathcal{A}$. Thus we should really define the Corona Property
for the triple $\left( \mathcal{A},N,\left\Vert \cdot \right\Vert _{\oplus
^{N}\mathcal{A}}\right) $, but we will often suppress the dependence on $N$
and $\left\Vert \cdot \right\Vert _{\oplus ^{N}\mathcal{A}}$ and only
specify them when needed.

\begin{definition}
\label{ACP}Suppose $N\geq 2$, $\mathcal{H}$ is a Hilbert function space on a
set $\Omega $, $\mathcal{A}$ is an algebra contained in $\mathcal{H}^{\infty
}$, and the direct sum $\oplus ^{N}\mathcal{A}$ is normed by $\left\Vert
\cdot \right\Vert _{\oplus ^{N}\mathcal{A}}$.

\begin{enumerate}
\item Let $C>0$. We say that the vector $\varphi =\left( \varphi _{1},\ldots
,\varphi _{N}\right) \in \oplus ^{N}\mathcal{A}$ satisfies the $\mathcal{A}$%
\emph{-corona condition }with constant $C$ if for each $h\in \mathcal{A}$
there are $f_{1},\ldots ,f_{N}\in \mathcal{A}$ satisfying 
\begin{eqnarray}
\left\Vert f\right\Vert _{\oplus ^{N}\mathcal{A}}^{2} &\leq &C^{2}\left\Vert
h\right\Vert _{\mathcal{A}}^{2},  \label{corcon adult} \\
\varphi \cdot f\left( z\right) &=&h\left( z\right) ,\ \ \ \ \ z\in \Omega . 
\notag
\end{eqnarray}

\item Let $c,C>0$. We say that $\mathcal{A}$ has the \emph{Corona Property}
with constants $c,C$ if for every vector $\varphi =\left( \varphi
_{1},\ldots ,\varphi _{N}\right) \in \oplus ^{N}\mathcal{A}$ satisfying both 
$\left\Vert \varphi \right\Vert _{\oplus ^{N}\mathcal{A}}\leq 1$ and (\ref%
{lowerphi}), the vector $\varphi =\left( \varphi _{1},\ldots ,\varphi
_{N}\right) $ satisfies the $\mathcal{A}$-corona condition with constant $C$.
\end{enumerate}
\end{definition}

\begin{remark}
If there are $f_{1},\ldots ,f_{N}\in \mathcal{A}$ satisfying (\ref{corcon
adult}) with $h=1$, then we can multiply the equation through by $h\in 
\mathcal{A}$, and use that $\mathcal{A}$ is an algebra to obtain $%
f_{1}h,\ldots ,f_{N}h\in \mathcal{A}$, and hence that $\mathcal{A}$
satisfies the Corona Property.
\end{remark}

\subsection{Convex shifted spaces}

At this point we pause to note that (\ref{Constant RKCP}) holds for all the 
\emph{extreme} cases $\mathbf{\theta }=\mathbf{e}_{m}$ provided a slight
strengthening of it holds for $\mathbf{\theta }=\mathbf{e}_{0}$. More
precisely we have the following lemma.

\begin{lemma}
\label{shifted kernels}Suppose that a vector $\varphi \in \oplus _{\ell
=1}^{N}L^{\infty }\left( \Omega \right) $ satisfies the $\mathcal{H}$-kernel
corona condition (\ref{corcon k}) with constant $C$. Then for every $a\in
\Omega $, there is $g\in \oplus ^{N}\mathcal{H}^{a}$ (depending on $a$) with 
$\left\Vert g\right\Vert _{\oplus ^{N}\mathcal{H}^{a}}^{2}\leq C^{2}$ such
that $\varphi \cdot g=1$.
\end{lemma}

\begin{proof}
Given $a\in \Omega $ there is by assumption a vector $f=\left( f_{1},\ldots
,f_{N}\right) \in \oplus ^{N}\mathcal{H}$ satisfying 
\begin{eqnarray*}
\varphi \cdot f\left( z\right) &=&\varphi _{1}\left( z\right) f_{1}\left(
z\right) +\cdots +\varphi _{N}\left( z\right) f_{N}\left( z\right) =%
\widetilde{k_{a}}\left( z\right) ,\ \ \ \ \ z\in \Omega , \\
\left\Vert f\right\Vert _{\oplus ^{N}\mathcal{H}}^{2} &=&\left\Vert
f_{1}\right\Vert _{\mathcal{H}}^{2}+\cdots +\left\Vert f_{N}\right\Vert _{%
\mathcal{H}}^{2}\leq C^{2}.
\end{eqnarray*}%
Now divide both sides of the Bezout equation above by $\widetilde{k_{a}}$ to
obtain%
\begin{equation*}
\varphi _{1}\left( z\right) \frac{f_{1}\left( z\right) }{\widetilde{k_{a}}%
\left( z\right) }+\cdots +\varphi _{N}\left( z\right) \frac{f_{N}\left(
z\right) }{\widetilde{k_{a}}\left( z\right) }=1,\ \ \ \ \ z\in \Omega .
\end{equation*}%
Then with $g_{\ell }\equiv \frac{f_{\ell }}{\widetilde{k_{a}}}\in \mathcal{H}%
^{a}$ for $1\leq \ell \leq N$ (here we use the definition of membership in $%
\mathcal{H}^{a}$), we have 
\begin{equation*}
\left\Vert g_{\ell }\right\Vert _{\mathcal{H}^{a}}^{2}=\left\Vert g_{\ell }%
\widetilde{k_{a}}\right\Vert _{\mathcal{H}}^{2}=\left\Vert \frac{f_{\ell }}{%
\widetilde{k_{a}}}\widetilde{k_{a}}\right\Vert _{\mathcal{H}}^{2}=\left\Vert
f_{\ell }\right\Vert _{\mathcal{H}}^{2}\ ,
\end{equation*}%
and so%
\begin{eqnarray*}
\varphi \cdot g\left( z\right) &=&\varphi _{1}\left( z\right) g_{1}\left(
z\right) +\cdots +\varphi _{N}\left( z\right) g_{N}\left( z\right) =1,\ \ \
\ \ z\in \Omega , \\
\left\Vert g\right\Vert _{\oplus ^{N}\mathcal{H}^{a}}^{2} &=&\left\Vert
g_{1}\right\Vert _{\mathcal{H}^{a}}^{2}+\cdots +\left\Vert g_{N}\right\Vert
_{\mathcal{H}^{a}}^{2}=\left\Vert f_{1}\right\Vert _{\mathcal{H}}^{2}+\cdots
+\left\Vert f_{N}\right\Vert _{\mathcal{H}}^{2}\leq C^{2}.
\end{eqnarray*}
\end{proof}

We do not know if the $\mathcal{H}$-kernel corona condition (\ref{corcon k})
with constant $C$ implies the `full' convex Poisson condition with constant $%
C$, or even with a larger positive constant $C^{\prime }$. However, the $%
\mathcal{H}-$convex Poisson condition holds for every $\varphi \in \oplus
_{\ell =1}^{N}H^{\infty }\left( \Omega \right) $ satisfying the $\mathcal{H}$%
-baby corona condition when $\mathcal{H}=H^{2}\left( \mathbb{D}\right) $ is
the classical Hardy space in the disk. This can be proved using an
appropriate outer function in place of the reproducing kernel $\widetilde{%
k_{a}}$ used in the proof above, and is carried out in Lemma \ref{Hardy
outer} in Section 6 below. The outer function in question is actually an
invertible multiplier on $H^{2}\left( \mathbb{D}\right) $, and this leads to
the following definition.

\begin{definition}
Let $\mathcal{H}=\mathcal{H}_{k}$ be a multiplier stable Hilbert function
space on a set $\Omega $ with reproducing kernel $k$, and containing the
constant functions. We say that the kernel $k$ has the \emph{Invertible
Multiplier Property} if for every $\left( \mathbf{a,\theta }\right) \in
\Omega ^{M}\times \Sigma _{M}\left( 0\right) $, there is a normalized
invertible multiplier $\widetilde{k_{\mathbf{a,\theta }}}\in M_{\mathcal{H}}$
such that%
\begin{equation}
\left\langle f,g\right\rangle _{\mathcal{H}^{\mathbf{a,\theta }%
}}=\left\langle \widetilde{k_{\mathbf{a,\theta }}}f,\widetilde{k_{\mathbf{%
a,\theta }}}g\right\rangle _{\mathcal{H}},\ \ \ \ \ f,g\in \mathcal{H}.
\label{outer}
\end{equation}
\end{definition}

The normalized invertible multiplier $\widetilde{k_{\mathbf{a,\theta }}}$ in
(\ref{outer}) is uniquely determined up to a unimodular constant.

\begin{lemma}
If two normalized invertible multipliers $\varphi _{1},\varphi _{2}\in M_{%
\mathcal{H}}$ satisfy%
\begin{equation}
\left\langle \varphi _{1}f,\varphi _{1}g\right\rangle _{\mathcal{H}%
}=\left\langle \varphi _{2}f,\varphi _{2}g\right\rangle _{\mathcal{H}}\text{
for all }f,g\in \mathcal{H},  \label{unique mult}
\end{equation}%
then $\varphi _{1}=A\varphi _{2}$ for some unimodular constant $A$.
\end{lemma}

\begin{proof}
If equality holds in (\ref{unique mult}), then 
\begin{equation*}
\left\langle \frac{\varphi _{1}}{\varphi _{2}}f,\frac{\varphi _{1}}{\varphi
_{2}}g\right\rangle _{\mathcal{H}}=\left\langle \varphi _{1}\left( \frac{1}{%
\varphi _{2}}f\right) ,\varphi _{1}\left( \frac{1}{\varphi _{2}}g\right)
\right\rangle _{\mathcal{H}}=\left\langle \varphi _{2}\left( \frac{1}{%
\varphi _{2}}f\right) ,\varphi _{2}\left( \frac{1}{\varphi _{2}}g\right)
\right\rangle _{\mathcal{H}}=\left\langle f,g\right\rangle _{\mathcal{H}}\ ,
\end{equation*}%
and so in particular,%
\begin{equation*}
\left\langle \frac{\varphi _{1}}{\varphi _{2}}f,\frac{\varphi _{1}}{\varphi
_{2}}k_{z}\right\rangle _{\mathcal{H}}=\left\langle f,k_{z}\right\rangle _{%
\mathcal{H}}=f\left( z\right) ,\ \ \ \ \ z\in \Omega .
\end{equation*}%
Thus the kernel $K_{z}\left( w\right) \equiv \frac{\varphi _{1}\left(
w\right) }{\varphi _{2}\left( w\right) }\frac{\overline{\varphi _{1}\left(
z\right) }}{\overline{\varphi _{2}\left( z\right) }}k_{z}\left( w\right) $
satisfies%
\begin{equation*}
\left\langle f,K_{z}\right\rangle _{\mathcal{H}}=\frac{\varphi _{1}\left(
z\right) }{\varphi _{2}\left( z\right) }\left\langle f,\frac{\varphi _{1}}{%
\varphi _{2}}k_{z}\right\rangle _{\mathcal{H}}=\frac{\varphi _{1}\left(
z\right) }{\varphi _{2}\left( z\right) }\left\langle \frac{\varphi _{1}}{%
\varphi _{2}}\left( \frac{\varphi _{2}}{\varphi _{1}}f\right) ,\frac{\varphi
_{1}}{\varphi _{2}}k_{z}\right\rangle _{\mathcal{H}}=\frac{\varphi
_{1}\left( z\right) }{\varphi _{2}\left( z\right) }\left( \frac{\varphi _{2}%
}{\varphi _{1}}f\right) \left( z\right) =f\left( z\right)
\end{equation*}%
for all $z\in \Omega $. By uniqueness of kernel functions, we obtain $\frac{%
\varphi _{1}\left( w\right) }{\varphi _{2}\left( w\right) }\frac{\overline{%
\varphi _{1}\left( z\right) }}{\overline{\varphi _{2}\left( z\right) }}%
k_{z}\left( w\right) =k_{z}\left( w\right) $, hence $\frac{\varphi
_{1}\left( w\right) }{\varphi _{2}\left( w\right) }\frac{\overline{\varphi
_{1}\left( z\right) }}{\overline{\varphi _{2}\left( z\right) }}=1$, since $%
k_{z}\left( w\right) $ is nonvanishing by multiplier stability. Thus we
conclude that there is a nonzero constant $A$ such that $\frac{\varphi
_{1}\left( w\right) }{\varphi _{2}\left( w\right) }=A$ for all $w\in \Omega $%
. Finally, since both multipliers are normalized we have $1=\left\Vert
\varphi _{1}\right\Vert _{\mathcal{H}}=\left\Vert A\varphi _{2}\right\Vert _{%
\mathcal{H}}=\left\vert A\right\vert \left\Vert \varphi _{2}\right\Vert _{%
\mathcal{H}}=\left\vert A\right\vert $.
\end{proof}

The identity (\ref{outer}) is equivalent to the assertion that the
reproducing kernels $k^{\mathbf{a,\theta }}$ and $k$ of the Hilbert spaces $%
\mathcal{H}^{\mathbf{a,\theta }}$ and $\mathcal{H}$ respectively, are
rescalings of each other. Indeed, the reproducing kernel corresponding to
the inner product $\left\langle \widetilde{k_{\mathbf{a,\theta }}}f,%
\widetilde{k_{\mathbf{a,\theta }}}g\right\rangle _{\mathcal{H}}$ is $\frac{%
k_{w}\left( z\right) }{\overline{\widetilde{k_{\mathbf{a,\theta }}}\left(
w\right) }\widetilde{k_{\mathbf{a,\theta }}}\left( z\right) }$. It turns out
that the Invertible Multiplier Property is extremely rare - the $1$%
-dimensional Szeg\"{o} kernel has it, and this is essentially the only
kernel we know with this property. See Section 6 below where we show that
the Invertible Multiplier Property holds for the Hardy space on the disk,
but fails for the Hardy space on both the ball and polydisc in higher
dimensions.

We now show that if $\mathcal{H}$ satisfies the\ Invertible Multiplier
Property, then the $\mathcal{H}$-baby corona condition is sufficient for the 
$\mathcal{H}$-convex Poisson corona condition.

\begin{lemma}
\label{outer property}Let $\mathcal{H}$ be a multiplier stable Hilbert
function space on a set $\Omega $, whose kernel $k$ has the Invertible
Multiplier Property. Then a vector $\varphi \in \oplus _{\ell
=1}^{N}L^{\infty }\left( \Omega \right) $ satisfies the\emph{\ }$\mathcal{H}%
- $convex Poisson condition in Definition \ref{RKCP} with positive constant $%
C$ if $\varphi \in \oplus _{\ell =1}^{N}L^{\infty }\left( \Omega \right) $
satisfies the $\mathcal{H}$-baby corona condition (\ref{corcon}) in
Definition \ref{BCP} with constant $C$.
\end{lemma}

\begin{proof}
Repeat the proof of Lemma \ref{shifted kernels} using the invertible
multiplier $\widetilde{k_{\mathbf{a,\theta }}}$ in place of $\widetilde{k_{a}%
}$ and use (\ref{outer}).
\end{proof}

\subsection{Formulation and proof}

In the case that $M_{\mathcal{H}}$ is the multiplier algebra of a Hilbert
function space $\mathcal{H}$ with a \emph{complete Pick} kernel $k$, there
is a characterization of the Corona Property for $M_{\mathcal{H}}$ in terms
of matrix-valued kernel positivity conditions involving $k$. This results in
the Toeplitz corona theorem which asserts the equivalence of the Baby Corona
Property for $\mathcal{H}$ and the Corona Property for its multiplier
algebra $M_{\mathcal{H}}$, and with the \textbf{same} constants $c,C$ - see 
\cite{BaTrVi}, \cite{AmTi} and also \cite[Theorem 8.57]{AgMc2}. Here is a
special case of the Toeplitz corona theorem as given in \cite[Theorem 8.57]%
{AgMc2}.

\begin{description}
\item[Toeplitz corona theorem] \label{ToeCorThm}Let $\mathcal{H}$ be a
Hilbert function space in a set $\Omega $ with an irreducible complete
Nevanlinna-Pick kernel. Let $C>0$ and $N\in \mathbb{N}$ and let $\varphi
=\left( \varphi _{1},\ldots,\varphi _{N}\right) \in \oplus ^{N}M_{\mathcal{H}%
}$. Then $\varphi $ satisfies the $M_{\mathcal{H}}$-corona condition (\ref%
{corcon adult}) with constant $C$ in Definition \ref{ACP} \emph{if and only
if} $\varphi $ satisfies the $\mathcal{H}$-baby corona condition (\ref%
{corcon}) with constant $C$ in Definition \ref{BCP}.
\end{description}

In this paper we will use Theorem \ref{RKCP implies Adult Toe} to obtain an
analogue of this theorem for the kernel multiplier space $K_{\mathcal{H}}$
when it is an algebra. The role of the Baby Corona Property for $\mathcal{H}$
will be played by the following property.

\begin{definition}
\label{CPC}Let $\mathcal{H}$ be a Hilbert function space with kernel $k$ on
a set $\Omega $, and let $c,C>0$. We say that the space $\mathcal{H}$ has
the \emph{Convex Poisson Property} with positive constants $c,C$ if for all
vectors $\varphi \in \oplus ^{N}K_{\mathcal{H}}$ satisfying $\left\Vert
\varphi \right\Vert _{\oplus ^{N}K_{\mathcal{H}}}\leq 1$ and (\ref{lowerphi}%
), the vector $\varphi $ satisfies the $\mathcal{H}-$convex Poisson
condition in Definition \ref{RKCP} with constant $C$.
\end{definition}

Here is our alternate Toeplitz corona theorem.

\begin{theorem}
\label{alternate}Suppose that $\mathcal{H}$ is a multiplier stable Hilbert
function space of continuous functions on $\Omega $ that contains the
constant functions, and enjoys the Montel property. Suppose further that the
space of kernel multipliers $K_{\mathcal{H}}$ is an algebra.

\begin{enumerate}
\item Then $K_{\mathcal{H}}$, with the direct sum $\oplus ^{N}K_{\mathcal{H}%
} $ normed by $\left\Vert \cdot \right\Vert _{\oplus ^{N}K_{\mathcal{H}}}$,
satisfies the Corona Property with positive constants $c,C$ \emph{if and
only if\ }$\mathcal{H}$ satisfies the Convex Poisson Property with positive
constants $c,C$.

\item Suppose in addition that $\mathcal{H}$ satisfies the Invertible
Multiplier Property and that $M_{\mathcal{H}}=K_{\mathcal{H}}$
isometrically. Equip the direct sum $\oplus ^{N}M_{\mathcal{H}}$ with the
norm $\left\Vert \cdot \right\Vert _{\oplus ^{N}M_{\mathcal{H}}}$.

\begin{enumerate}
\item Then $\mathcal{H}$ satisfies the Baby Corona Property with constants $%
c,C$ if $M_{\mathcal{H}}$ satisfies the Corona Property with the constants $%
c,C$.

\item Conversely, $M_{\mathcal{H}}$ satisfies the Corona Property with
constants $c,C\sqrt{N}$ if $\mathcal{H}$ satisfies the Baby Corona Property
with constants $c,C$.
\end{enumerate}
\end{enumerate}
\end{theorem}

Armed with Theorem \ref{RKCP implies Adult Toe}, Lemma \ref{outer property}
and Corollary \ref{constant 1}, it is now an easy matter to prove Theorem %
\ref{alternate}.

\begin{proof}[Proof of Theorem \protect\ref{alternate}]
The first assertion follows immediately from Theorem \ref{RKCP implies Adult
Toe} and definitions. Now we turn to the second assertion. Clearly the Baby
Corona Property with constants $c,C$ holds for $\mathcal{H}$ if the
multiplier algebra $M_{\mathcal{H}}$ satisfies the Corona Property with
constants $c,C$. Conversely, assume the Baby Corona Property with constants $%
c,C$ holds for $\mathcal{H}$ as above. We must show that the Corona Property
holds for $M_{\mathcal{H}}$ with constants $c,C\sqrt{N}$. So fix a vector $%
\varphi \in M_{\mathcal{H}}$ with $\left\Vert \varphi \right\Vert _{\oplus
_{\ell =1}^{N}M_{\mathcal{H}}}\leq 1$ and $\left\vert \varphi \left(
z\right) \right\vert \geq c>0$ for $z\in \Omega $. Then the Baby Corona
Property for $\mathcal{H}$ with constants $c,C$ implies that the vector $%
\varphi $ satisfies the $\mathcal{H}$-baby corona condition with constant $C$%
, i.e. for each $h\in \mathcal{H}$ there is $f\in \oplus ^{N}\mathcal{H}$
satisfying 
\begin{equation*}
\left\Vert f\right\Vert _{\oplus ^{N}\mathcal{H}}^{2}\leq C^{2}\left\Vert
h\right\Vert _{\mathcal{H}}^{2}\text{ and }\left( \varphi \cdot f\right)
\left( z\right) =h\left( z\right) ,\ \ \ \ \ z\in \Omega .
\end{equation*}%
It follows from Lemma \ref{outer property} that $\varphi $ satisfies the%
\emph{\ }$\mathcal{H}$-convex Poisson condition in Definition \ref{RKCP}
with constant $C$. Now apply Theorem \ref{RKCP implies Adult Toe} to
conclude that there is a vector function $f\in \oplus ^{N}K_{\mathcal{H}}$
such that 
\begin{equation*}
\left\Vert f\right\Vert _{\oplus ^{N}K_{\mathcal{H}}}\leq C\text{ and }%
\left( \varphi \cdot f\right) \left( z\right) =1,\ \ \ \ \ z\in \Omega .
\end{equation*}%
By (\ref{K and M}) we thus conclude that $\left\Vert f\right\Vert _{\oplus
^{N}M_{\mathcal{H}}}\leq C\sqrt{N}$, and this shows that $M_{\mathcal{H}}$
satisfies the Corona Property with constants $c,C\sqrt{N}$.
\end{proof}

\begin{remark}
In the event that $\mathcal{H}=H^{2}\left( \Omega \right) $ is the Hardy
space on a bounded domain with $C^{2}$ boundary in $\mathbb{C}^{n}$ (see
Section 5 for definitions), then we have%
\begin{eqnarray*}
\left\Vert \varphi \right\Vert _{\oplus _{\ell =1}^{N}M_{\mathcal{H}}}^{%
\limfunc{column}} &=&\left\Vert \mathbb{M}_{\varphi }\right\Vert _{\limfunc{%
op}}\equiv \sup_{h\neq 0}\frac{\left\Vert \mathbb{M}_{\varphi }h\right\Vert
_{\oplus ^{N}\mathcal{H}}}{\left\Vert h\right\Vert _{\mathcal{H}}}%
=\sup_{h\neq 0}\frac{\sqrt{\sum_{\alpha =1}^{N}\left\Vert \mathcal{M}%
_{\varphi _{\alpha }}h\right\Vert _{H^{2}\left( \Omega \right) }^{2}}}{%
\left\Vert h\right\Vert _{H^{2}\left( \Omega \right) }} \\
&=&\sup_{h\neq 0}\frac{\sqrt{\int_{\partial \Omega }\left( \sum_{\alpha
=1}^{N}\left\vert \varphi _{\alpha }\right\vert ^{2}\right) \left\vert
h\right\vert ^{2}d\sigma }}{\sqrt{\int_{\partial \Omega }\left\vert
h\right\vert ^{2}d\sigma }}=\left\Vert \sum_{\alpha =1}^{N}\left\vert
\varphi _{\alpha }\right\vert ^{2}\right\Vert _{L^{\infty }\left( \partial
\Omega \right) }=\left\Vert \varphi \right\Vert _{L^{\infty }\left( \ell
_{N}^{2}\right) } \\
&=&\sup_{a\in \Omega }\frac{\sqrt{\int_{\partial \Omega }\left( \sum_{\alpha
=1}^{N}\left\vert \varphi _{\alpha }\right\vert ^{2}\right) \left\vert 
\widetilde{k_{a}}\right\vert ^{2}d\sigma }}{\sqrt{\int_{\partial \Omega
}\left\vert \widetilde{k_{a}}\right\vert ^{2}d\sigma }}=\sup_{a\in \Omega }%
\sqrt{\int_{\partial \Omega }\left( \sum_{\alpha =1}^{N}\left\vert \varphi
_{\alpha }\right\vert ^{2}\right) \left\vert \widetilde{k_{a}}\right\vert
^{2}d\sigma }\leq \left\Vert \varphi \right\Vert _{\oplus _{\ell =1}^{N}K_{%
\mathcal{H}}}
\end{eqnarray*}%
and so $\left\Vert \varphi \right\Vert _{\oplus _{\ell =1}^{N}M_{\mathcal{H}%
}}^{\limfunc{column}}=\left\Vert \varphi \right\Vert _{\oplus _{\ell
=1}^{N}K_{\mathcal{H}}}=\left\Vert \varphi \right\Vert _{L^{\infty }\left(
\ell _{N}^{2}\right) }$. Thus in the case $\mathcal{H}=H^{2}\left( \Omega
\right) $, we can replace the constant $C\sqrt{N}$ in part (2) of Theorem %
\ref{alternate} with the constant $C$, which shows that the Corona Property
for $M_{\mathcal{H}}$ holds with the \textbf{same} constants $c,C$ for which
the Baby Corona Property holds for $\mathcal{H}$.
\end{remark}

\section{The alternate Toeplitz corona theorem for Bergman and Hardy spaces
in $\mathbb{C}^{n}$}

Let $\Omega $ be a bounded domain in $\mathbb{C}^{n}$. The Bergman space $%
A^{2}\left( \Omega \right) $ consists of those $f\in H\left( \Omega \right) $
such that%
\begin{equation*}
\int_{\Omega }\left\vert f\right\vert ^{2}d\widetilde{V}<\infty ,
\end{equation*}%
where $d\widetilde{V}=\frac{1}{\left\vert \Omega \right\vert }dV$ and $dV$
is Lebesgue measure on $\Omega $. We then have that $\left\Vert f\right\Vert
_{A^{2}\left( \Omega \right) }\equiv \sqrt{\int_{\Omega }\left\vert
f\right\vert ^{2}d\widetilde{V}}$ defines a norm on $A^{2}\left( \Omega
\right) $. Now point evaluations are bounded, hence continuous, on $%
A^{2}\left( \Omega \right) $ by the mean value property,%
\begin{eqnarray*}
\left\vert f\left( z\right) \right\vert &=&\left\vert \frac{1}{\left\vert
B\left( z,d_{\partial \Omega }\left( z\right) \right) \right\vert }%
\int_{B\left( z,d_{\partial \Omega }\left( z\right) \right) }f\left(
w\right) dV\left( w\right) \right\vert \\
&\leq &\frac{1}{c_{n}d_{\partial \Omega }\left( z\right) ^{n}}\left(
\int_{\Omega }\left\vert f\left( w\right) \right\vert ^{2}dV\left( w\right)
\right) ^{\frac{1}{2}}=c_{n}d_{\partial \Omega }\left( z\right)
^{-n}\left\Vert f\right\Vert _{A^{2}\left( \Omega \right) }.
\end{eqnarray*}%
Thus the Bergman space $A^{2}\left( \Omega \right) $ is a reproducing kernel
Hilbert space, and we denote by $k_{z}\left( w\right) $ the reproducing
kernel for $A^{2}\left( \Omega \right) $ with respect to the inner product%
\begin{equation*}
\left\langle f,g\right\rangle _{A^{2}\left( \Omega \right) }\equiv
\int_{\Omega }f\overline{g}d\widetilde{V},\ \ \ \ \ f,g\in A^{2}\left(
\Omega \right) .
\end{equation*}%
We then have%
\begin{equation*}
f\left( z\right) =\left\langle f,k_{z}\right\rangle _{A^{2}\left( \Omega
\right) }=\int_{\Omega }f\left( w\right) \overline{k_{z}\left( w\right) }d%
\widetilde{V}\left( w\right) ,\ \ \ \ \ z\in \Omega ,\ f\in A^{2}\left(
\Omega \right) .
\end{equation*}%
Note also that $M_{A^{2}\left( \Omega \right) }=H^{\infty }\left( \Omega
\right) $, the algebra of bounded analytic functions on $\Omega $, and that
in fact with $\mathcal{H}=A^{2}\left( \Omega \right) $, we have $M_{\mathcal{%
H}}=K_{\mathcal{H}}=\mathcal{H}^{\infty }=H^{\infty }\left( \Omega \right) $.

\begin{remark}
If $\Omega $ is strictly pseudoconvex, then by Theorem 2 in C. Fefferman 
\cite{Fef} (see also Boutet de Monvel and Sj\"{o}strand \cite{BdMSj}) we
have that $k_{a}$ is bounded for each $a\in \Omega $, and moreover that $%
\left\Vert k_{a}-k_{b}\right\Vert _{\infty }=\sup_{z\in \Omega }\left\vert
k_{a}\left( z\right) -k_{b}\left( z\right) \right\vert $ tends to $0$ as $%
b\rightarrow a$ for each $a\in \Omega $. Thus the map $a\rightarrow k_{a}$
is a \emph{continuous} map from $\Omega $ to the multiplier algebra $%
H^{\infty }\left( \Omega \right) =M_{A^{2}\left( \Omega \right) }$. While
the Bergman kernel functions are also nonvanishing in the ball, H. Boas has
shown that in a generic sense, the Bergman kernels of strictly pseudoconvex
domains have zeroes (\cite{Boas}; see also Skwarczynski \cite[Skw]{Skw}, and
Boas \cite{Boas2} for a nice survey on Lu Qi-Keng's problem), and in a paper
with Fu and Straube \cite{BFS}, they constructed specific examples of such
domains, including even some strictly convex smooth Reinhardt domains.
\end{remark}

Now we recall the definition of the Hardy spaces $H^{2}\left( \Omega \right) 
$ for a bounded domain $\Omega $ in $\mathbb{C}^{n}$ with $C^{2}$ boundary.
The Hardy space $H^{2}\left( \Omega \right) $ consists of those $f\in
H\left( \Omega \right) $ such that%
\begin{equation*}
\sup_{\varepsilon >0}\int_{\partial \Omega _{\varepsilon }}\left\vert
f\right\vert ^{2}d\sigma _{\varepsilon }<\infty ,
\end{equation*}%
where $\Omega _{\varepsilon }\equiv \left\{ z\in \Omega :\rho \left(
z\right) <-\varepsilon \right\} $ and $\rho $ is an appropriate defining
function for $\Omega $, and where $\sigma _{\varepsilon }$ is surface
measure on $\partial \Omega _{\varepsilon }$. We then have that%
\begin{equation*}
\left\Vert f\right\Vert _{H^{2}\left( \Omega \right) }\equiv \sqrt{%
\int_{\partial \Omega }\left\vert f^{\ast }\right\vert ^{2}d\sigma }
\end{equation*}%
defines a norm on $H^{2}\left( \Omega \right) $, where $\sigma $ is surface
measure on $\partial \Omega $, and where the nontangential boundary limits $%
f^{\ast }$ exist a.e. $\left[ \sigma \right] $ on $\partial \Omega $. We
also note that 
\begin{equation*}
\int_{\Omega }\left\vert f\right\vert ^{2}dA\leq C\sup_{\varepsilon
>0}\int_{\partial \Omega _{\varepsilon }}\left\vert f\right\vert ^{2}d\sigma
_{\varepsilon }<\infty
\end{equation*}%
for some constant $C$ depending only on $\Omega $, and this shows that the
Bergman space norm $\left\Vert f\right\Vert _{A^{2}\left( \Omega \right)
}\equiv \sqrt{\int_{\Omega }\left\vert f\right\vert ^{2}dA}$ is dominated by
a multiple of the Hardy space norm $\left\Vert f\right\Vert _{H^{2}\left(
\Omega \right) }$:%
\begin{equation*}
\left\Vert f\right\Vert _{A^{2}\left( \Omega \right) }\leq C_{\Omega
}\left\Vert f\right\Vert _{H^{2}\left( \Omega \right) }.
\end{equation*}%
Since point evaluations are continuous on $A^{2}\left( \Omega \right) $, it
follows that they are also continuous on $H^{2}\left( \Omega \right) $. Thus
the Hardy space $H^{2}\left( \Omega \right) $ is a reproducing kernel
Hilbert space, and we denote by $k_{z}\left( w\right) $ the reproducing
kernel for $H^{2}\left( \Omega \right) $ with respect to the inner product%
\begin{equation*}
\left\langle f,g\right\rangle _{H^{2}\left( \Omega \right) }\equiv
\int_{\partial \Omega }f^{\ast }\overline{g^{\ast }}d\sigma ,\ \ \ \ \
f,g\in H^{2}\left( \Omega \right) .
\end{equation*}%
We then have%
\begin{equation*}
f\left( z\right) =\left\langle f,k_{z}\right\rangle _{H^{2}\left( \Omega
\right) }=\int_{\partial \Omega }f^{\ast }\left( w\right) \overline{%
k_{z}^{\ast }\left( w\right) }d\sigma \left( w\right) ,\ \ \ \ \ z\in \Omega
,f\in H^{2}\left( \Omega \right) .
\end{equation*}%
We will typically suppress the star in the superscript and simply write $%
\left\langle f,g\right\rangle _{H^{2}\left( \Omega \right) }\equiv
\int_{\partial \Omega }f\overline{g}d\sigma $ and $\left\langle
f,k_{z}\right\rangle _{H^{2}\left( \Omega \right) }=\int_{\partial \Omega }f%
\overline{k_{z}}d\sigma $. In the case of the Hardy space on the polydisc $%
\mathbb{D}^{n}$, where $\partial \mathbb{D}^{n}$ is not $C^{2}$, integration
is restricted to the distinguished boundary $\mathbb{T}^{n}$, a subset of
measure zero of $\partial \Omega $.

The advantage of the Hardy space over the Bergman space is the equality of
vector norms $\left\Vert \varphi \right\Vert _{\oplus ^{N}H^{2}\left( \Omega
\right) }^{\limfunc{column}}=\left\Vert \varphi \right\Vert _{\oplus
^{N}H^{2}\left( \Omega \right) }^{\func{row}}=\left\Vert \varphi \right\Vert
_{L^{2}\left( \ell _{N}^{2}\right) }$, while the advantage of the Bergman
space over the Hardy space is the greater generality of the domains $\Omega $
for which it is defined.

\subsection{The Bergman and Hardy space on the unit ball}

First note that the reproducing kernels $k_{a}\left( w\right) =$ $\frac{1}{%
\left( 1-\overline{a}\cdot w\right) ^{n+1}}$ for the standard inner product
on $A^{2}\left( \mathbb{B}_{n}\right) $ are invertible multipliers of $%
A^{2}\left( \mathbb{B}_{n}\right) $ (with possibly large multiplier norm $%
\left\Vert k_{a}\right\Vert _{\infty }=\frac{1}{\left( 1-\left\vert
a\right\vert \right) ^{n+1}}$, but this norm plays no role here) depending
continuously on $a\in \mathbb{B}_{n}$, that $M_{A^{2}\left( \mathbb{B}%
_{n}\right) }=H^{\infty }\left( \mathbb{B}_{n}\right) $, that $A^{2}\left( 
\mathbb{B}_{n}\right) $ contains the constants, and finally that $%
A^{2}\left( \mathbb{B}_{n}\right) $ satisfies the Montel property precisely
by Montel's theorem for holomorphic functions. Similar considerations apply
to the Hardy space $H^{2}\left( \mathbb{B}_{n}\right) $. From part (2) of
Theorem \ref{alternate} we obtain the following Corollary.

\begin{corollary}
The Corona Property holds for $H^{\infty }\left( \mathbb{B}_{n}\right) $ if
and only if $A^{2}\left( \mathbb{B}_{n}\right) $ satisfies the Convex
Poisson Property if and only if $H^{2}\left( \mathbb{B}_{n}\right) $
satisfies the Convex Poisson Property.
\end{corollary}

\subsection{The Hardy and Bergman space on the unit polydisc}

For the polydisc, we will again apply our alternate Toeplitz corona theorem
to the Hardy space $H^{2}\left( \mathbb{D}^{n}\right) $ and the Bergman
space $A^{2}\left( \mathbb{D}^{n}\right) $. First, we note that the
multiplier algebra of $H^{2}\left( \mathbb{D}^{n}\right) $ is $H^{\infty
}\left( \mathbb{D}^{n}\right) $, and that the reproducing kernels $%
k_{a}=\prod_{j=1}^{n}\frac{1}{(1-\overline{a}_{j}z_{j})}$ for the standard
inner product on $H^{2}\left( \mathbb{D}^{n}\right) $ are invertible
multipliers of $H^{2}\left( \mathbb{D}^{n}\right) $ and that the map $%
a\rightarrow k_{a}$ is continuous from $\mathbb{D}^{n}$ to $H^{\infty
}\left( \mathbb{D}^{n}\right) $. Similar considerations apply to the Bergman
space $A^{2}\left( \mathbb{D}^{n}\right) $.

\begin{corollary}
The corona theorem holds for $H^{\infty }\left( \mathbb{D}^{n}\right) $ if
and only if $H^{2}\left( \mathbb{D}^{n}\right) $ satisfies the Convex
Poisson Property if and only if $A^{2}\left( \mathbb{D}^{n}\right) $
satisfies the Convex Poisson Property.
\end{corollary}

\subsection{General domains}

From part (2) of Theorem \ref{alternate}, we obtain the corona theorem for $%
H^{\infty }\left( \Omega \right) $ for any bounded domain $\Omega $ in $%
\mathbb{C}^{n}$ for which the Convex Poisson Property holds for the Bergman
space $A^{2}\left( \Omega \right) $, and for which the Bergman space $%
A^{2}\left( \Omega \right) $ is multiplier stable (note that $K_{A^{2}\left(
\Omega \right) }=H^{\infty }\left( \Omega \right) $, contains constants and
satisfies the Montel property).

Greene and Krantz \cite{GrKr} show that for domains $\Omega $ in a
sufficiently small $C^{\infty }$ neighbourhood of the unit ball $\mathbb{B}%
_{n}$ (or any other strictly pseudoconvex domain $D$ with smooth boundary
for which $\inf_{z,w\in D}k_{D}\left( z,w\right) >0$), we have the lower
bound $\inf_{z,w\in \Omega }k_{\Omega }\left( z,w\right) >0$.

\begin{corollary}
Let\thinspace $D$ be a strictly pseudoconvex domain with smooth boundary for
which $\inf_{z,w\in D}k_{D}\left( z,w\right) >0$, where $k_{D}$ is the
Bergman kernel for $D$. Then for $\Omega $ in a sufficiently small $%
C^{\infty }$ neighbourhood of $D$, the corona theorem holds for $H^{\infty
}\left( \Omega \right) $ if and only if the Convex Poisson Property holds
for $A^{2}\left( \Omega \right) $.
\end{corollary}

If $f\in \limfunc{Aut}\left( \Omega \right) $, then we have the Bergman
kernel transformation law (see e.g. \cite{Boas2})%
\begin{equation*}
K\left( z,w\right) =\det f^{\prime }\left( z\right) K\left( f\left( z\right)
,f\left( w\right) \right) \overline{\det f^{\prime }\left( w\right) },\ \ \
\ \ z,w\in \Omega .
\end{equation*}%
If $\Omega $ is a complete circular domain, then $K\left( z,0\right) $ is a
nonzero constant $K_{0}$ (see Bell \cite{Bell}; see also \cite{Boas2}). Now
assume in addition that $\Omega $ is homogeneous and fix $\zeta \in \Omega $%
. Then with $f\in \limfunc{Aut}\left( \Omega \right) $ such that $f\left(
\zeta \right) =0$, the transformation law shows that%
\begin{equation*}
K\left( z,\zeta \right) =\det f^{\prime }\left( z\right) K\left( f\left(
z\right) ,f\left( \zeta \right) \right) \overline{\det f^{\prime }\left(
\zeta \right) }=\det f^{\prime }\left( z\right) K_{0}\overline{\det
f^{\prime }\left( \zeta \right) }.
\end{equation*}%
If $g=f^{-1}$ we have $\det f^{\prime }\left( z\right) =\frac{1}{\det
g^{\prime }\left( z\right) }$, and from another application of the
transformation law with $\eta =f\left( 0\right) $,%
\begin{equation*}
\overline{K\left( z,\eta \right) }=K\left( \eta ,z\right) =\det g^{\prime
}\left( \eta \right) K\left( g\left( \eta \right) ,g\left( z\right) \right) 
\overline{\det g^{\prime }\left( z\right) }=\det g^{\prime }\left( \eta
\right) \overline{K_{0}}\overline{\det g^{\prime }\left( z\right) },
\end{equation*}%
we see that%
\begin{eqnarray*}
K\left( z,\zeta \right) &=&\det f^{\prime }\left( z\right) K_{0}\overline{%
\det f^{\prime }\left( \zeta \right) }=\frac{1}{\det g^{\prime }\left(
z\right) }K_{0}\overline{\det f^{\prime }\left( \zeta \right) }=\frac{%
\overline{\det g^{\prime }\left( \eta \right) }K_{0}}{K\left( z,\eta \right) 
}K_{0}\overline{\det f^{\prime }\left( \zeta \right) } \\
&=&K_{0}^{2}\frac{\overline{\det g^{\prime }\left( \eta \right) }\overline{%
\det f^{\prime }\left( \zeta \right) }}{K\left( z,\eta \right) },
\end{eqnarray*}%
and hence that%
\begin{equation*}
\inf_{z\in \Omega }\left\vert K\left( z,\zeta \right) \right\vert \geq
\left\vert K_{0}\right\vert ^{2}\frac{\left\vert \det g^{\prime }\left( \eta
\right) \det f^{\prime }\left( \zeta \right) \right\vert }{\left\Vert
K_{\eta }\right\Vert _{\infty }}>0
\end{equation*}%
since $\left\Vert K_{\eta }\right\Vert _{\infty }<\infty $ if $\Omega $ is a
strictly pseudoconvex domain. The continuity of the map $\zeta \rightarrow
K\left( z,\zeta \right) $ also follows from these calculations. This gives
the following corollary to Theorem \ref{alternate}.

\begin{corollary}
Let\ $\Omega $ be a bounded strictly pseudoconvex homogeneous complete
circular domain with smooth boundary in $\mathbb{C}^{n}$. Then the corona
theorem holds for $H^{\infty }\left( \Omega \right) $ if and only if the
Convex Poisson Property holds for $A^{2}\left( \Omega \right) $.
\end{corollary}

\begin{remark}
The above smoothness assumption on the boundaries of $D$ and $\Omega $ can
be relaxed, but we will not pursue that here.
\end{remark}

\begin{remark}
In all of the above corollaries, the Baby Corona Property for $H^{2}\left(
\Omega \right) $ and $A^{2}\left( \Omega \right) $ is known to hold - see 
\cite{KrLi} for strictly pseudoconvex domains as above, and see \cite{Lin}, 
\cite{Li}, \cite{Lin2}, \cite{Tren}, and \cite{TrWi2} for the polydisc.
\end{remark}

\begin{remark}
\label{other spaces}For the pseudoconvex domain constructed by Sibony \cite%
{Sib}, in which the Corona Property for $H^{\infty }\left( \Omega \right) $
fails, the Szeg\"{o} kernel functions cannot be invertible multipliers
depending in a lower semicontinuous way on the pole. Indeed, $M_{H^{2}\left(
\Omega \right) }=H^{\infty }\left( \Omega \right) $ and $H^{2}\left( \Omega
\right) $ contains constants and satisfies the Montel property. Moreover, by
a result of Andersson and Carlsson \cite{AnCa}, the baby corona theorem
holds for $H^{2}\left( \Omega \right) $ on pseudoconvex domains with smooth
boundary, and Theorem \ref{alternate} now shows that the Szeg\"{o} kernels
cannot be bounded away from both $0$ and $\infty $ and satisfy the
semicontinuity assumption on the Sibony domain.
\end{remark}

\begin{remark}
There is a partial result for certain `polydomains' in $\mathbb{C}^{n}$. The
Bergman space $A^{2}\left( \Omega \right) $ satisfies $K_{A^{2}\left( \Omega
\right) }=M_{A^{2}\left( \Omega \right) }=H^{\infty }\left( \Omega \right) $%
, and clearly contains the constants and is a Montel space. If the Bergman
spaces $A^{2}\left( \Omega _{j}\right) $, $\Omega _{j}\subset \mathbb{C}%
^{n_{j}}$, are multiplier stable for $1\leq j\leq J$, then the Bergman space 
$A^{2}\left( \Omega \right) $ for the polydomain $\Omega
=\prod\limits_{j=1}^{J}\Omega _{j}\subset \mathbb{C}^{n}$, $%
n=\sum_{j=1}^{J}n_{j}$, is also multiplier stable (since the kernel function
of $A^{2}\left( \Omega \right) $ is the product of the kernel functions for $%
A^{2}\left( \Omega _{j}\right) $, and the multiplier norm is the supremum
norm). Thus in this case $M_{A^{2}\left( \Omega \right) }$ satisfies the
Corona Property if and only if $A^{2}\left( \Omega \right) $ satisfies the
Convex Poisson Property.
\end{remark}

\section{Invertible Multiplier Property for Hardy spaces}

Here we begin by showing that the Invertible Multiplier Property holds for
the Hardy space $H^{2}\left( \mathbb{D}\right) $ on the disk.

\begin{lemma}
\label{Hardy outer}The Szeg\"{o} kernel for the Hardy space $\mathcal{H}%
=H^{2}\left( \mathbb{D}\right) $ has the Invertible Multiplier Property.
\end{lemma}

\begin{proof}
Take $\widetilde{k_{\mathbf{a,\theta }}}$ to be the outer function%
\begin{equation*}
\widetilde{k_{\mathbf{a,\theta }}}\left( z\right) \equiv \exp \left\{ \frac{1%
}{2\pi }\int_{0}^{2\pi }\frac{e^{it}+z}{e^{it}-z}\ln \left(
\sum_{m=0}^{M}\theta _{m}\left\vert \widetilde{k_{a_{m}}}\left(
e^{it}\right) \right\vert ^{2}\right) dt\right\} ,
\end{equation*}%
which by \cite[Theorem 17.16 (b)]{Rud4} satisfies $\left\vert \widetilde{k_{%
\mathbf{a,\theta }}}\left( e^{it}\right) \right\vert
^{2}=\sum_{m=0}^{M}\theta _{m}\left\vert \widetilde{k_{a_{m}}}\left(
e^{it}\right) \right\vert ^{2}$ for almost every $0\leq t<2\pi $ (the
equality actually holds for all $t$ by continuity of the $\widetilde{%
k_{a_{m}}}$ on the boundary). Here the normalized reproducing kernel $%
\widetilde{k_{a_{m}}}\left( z\right) $ is given by $\widetilde{k_{a_{m}}}%
\left( z\right) =\frac{\sqrt{1-\left\vert a_{m}\right\vert ^{2}}}{1-%
\overline{a_{m}}z}$. Then $\widetilde{k_{\mathbf{a,\theta }}}$ and $\frac{1}{%
\widetilde{k_{\mathbf{a,\theta }}}}$ are bounded in the disk and we have%
\begin{eqnarray*}
\left\langle \widetilde{k_{\mathbf{a,\theta }}}f,\widetilde{k_{\mathbf{%
a,\theta }}}g\right\rangle _{\mathcal{H}} &=&\int_{0}^{2\pi }\widetilde{k_{%
\mathbf{a,\theta }}}\left( e^{it}\right) f^{\ast }\left( e^{it}\right) \ 
\overline{\widetilde{k_{\mathbf{a,\theta }}}\left( e^{it}\right) g^{\ast
}\left( e^{it}\right) }\ dt \\
&=&\int_{0}^{2\pi }f^{\ast }\left( e^{it}\right) \ \overline{g^{\ast }\left(
e^{it}\right) }\ \left\vert \widetilde{k_{\mathbf{a,\theta }}}\left(
e^{it}\right) \right\vert ^{2}dt=\sum_{m=0}^{M}\theta _{m}\int_{0}^{2\pi
}f^{\ast }\left( e^{it}\right) \ \overline{g^{\ast }\left( e^{it}\right) }\
\left\vert \widetilde{k_{a_{m}}}\left( e^{it}\right) \right\vert ^{2}dt \\
&=&\sum_{m=0}^{M}\theta _{m}\left\langle \widetilde{k_{a_{m}}}f,\widetilde{%
k_{a_{m}}}g\right\rangle _{H^{2}\left( \mathbb{D}\right)
}=\sum_{m=0}^{M}\theta _{m}\left\langle f,g\right\rangle _{\mathcal{H}%
^{a_{m}}}=\left\langle f,g\right\rangle _{\mathcal{H}^{\mathbf{a,\theta }}}\
.
\end{eqnarray*}%
In particular $\left\Vert \widetilde{k_{\mathbf{a,\theta }}}\right\Vert
_{H^{2}\left( \mathbb{D}\right) }=\left\Vert 1\right\Vert _{\mathcal{H}^{%
\mathbf{a,\theta }}}=1$, and so we have shown that $\widetilde{k_{\mathbf{%
a,\theta }}}$ is a normalized invertible multiplier satisfying (\ref{outer}).
\end{proof}

Now we show that the Invertible Multiplier Property \emph{fails} for the Szeg%
\"{o} kernel for the Hardy space on the ball in higher dimensions by showing
that when at least two of the $\theta _{m}$ are positive, then there are no
invertible multipliers whose real parts have boundary values equal to $\ln
\left( \sum_{m=0}^{M}\theta _{m}\left\vert \widetilde{k_{a_{m}}}\right\vert
^{2}\right) $ almost everywhere. This should be contrasted with results of
Aleksandrov \cite{Ale} and L\NEG{o}w \cite{Low} that yield \emph{nonvanishing%
} multipliers whose real parts \emph{do} have boundary values equal to $\ln
\left( \sum_{m=0}^{M}\theta _{m}\left\vert \widetilde{k_{a_{m}}}\right\vert
^{2}\right) $ almost everywhere. \ In the proof of Theorem \ref{RKCP implies
Adult Toe} above (see (\ref{reason})) it was essential that $\widetilde{k_{%
\mathbf{a,\theta }}}$ was an \emph{invertible} multiplier, as opposed to a
more general \emph{nonvanishing} multiplier. But the reciprocals of the
Aleksandrov and L\NEG{o}w multipliers are not bounded in the ball (only
their boundary values are bounded almost everywhere on the sphere), and in
fact the reciprocals do not belong to any reasonable class of holomorphic
functions on the ball. These nonvanishing multipliers can be thought of as
ball analogues of products of \emph{singular} inner functions with outer
functions in the disk.

\begin{lemma}
\label{ball failure}The Hardy space $\mathcal{H}=H^{2}\left( \mathbb{B}%
_{n}\right) $ on the unit ball in $\mathbb{C}^{n}$ fails to have the
Invertible Multiplier Property when $n>1$. The failure is spectacular in
that for \emph{any} $\left( \mathbf{a,\theta }\right) \in \Omega ^{M}\times
\Sigma _{M}$ with at least two of the $\theta _{m}$ positive, there is no
normalized invertible multiplier $\widetilde{k_{\mathbf{a,\theta }}}\in M_{%
\mathcal{H}}=H^{\infty }\left( \mathbb{B}_{n}\right) $ satisfying (\ref%
{outer}).
\end{lemma}

\begin{proof}
We prove the case when $M=1$, $\theta _{0}=\theta _{1}=\frac{1}{2}$ and $%
a_{1}=\alpha e_{1}=\left( \alpha ,0,\ldots,0\right) \in \mathbb{B}_{n}$, and
leave the general case to the interested reader. We assume, in order to
derive a contradiction, that that there is a normalized invertible
multiplier $\varphi \in M_{\mathcal{H}}$ such that%
\begin{equation*}
\left\langle f,g\right\rangle _{\mathcal{H}}+\left\langle
k_{a_{1}}f,k_{a_{1}}g\right\rangle _{\mathcal{H}}=2\left\langle
f,g\right\rangle _{\mathcal{H}^{\mathbf{a,\theta }}}=\left\langle \varphi
f,\varphi g\right\rangle _{\mathcal{H}},\ \ \ \ \ f,g\in \mathcal{H},
\end{equation*}%
where we have absorbed the factor $2$ into $\varphi $ for convenience.
Unraveling notation this becomes%
\begin{equation*}
\int_{\partial \mathbb{B}_{n}}f\overline{g}\left( 1+\left\vert \widetilde{%
k_{a_{1}}}\right\vert ^{2}\right) d\sigma =\int_{\partial \mathbb{B}_{n}}f%
\overline{g}\left\vert \varphi \right\vert ^{2}d\sigma ,\ \ \ \ \ f,g\in
H^{2}\left( \mathbb{B}_{n}\right) .
\end{equation*}%
In particular we obtain that%
\begin{equation*}
1+\left\vert \widetilde{k_{a_{1}}}\right\vert ^{2}=\left\vert \varphi
\right\vert ^{2}\text{ on }\partial \mathbb{B}_{n}
\end{equation*}%
by the Stone-Weierstrass theorem, since the algebra $\mathcal{A}$ generated
by restrictions to $\partial \mathbb{B}_{n}$ of $f$ and $\overline{g}$ for $%
f,g\in A\left( \mathbb{B}_{n}\right) $, is self-adjoint, separates points
and contains the constants on $\partial \mathbb{B}_{n}$; and thus $\mathcal{A%
}$ is dense in $C\left( \partial \mathbb{B}_{n}\right) $. Since $\varphi $
is an \emph{invertible} multiplier, we claim it has a holomorphic logarithm $%
F=\log \varphi $ in $H^{\infty }\left( \mathbb{B}_{n}\right) $ whose
boundary values satisfy 
\begin{equation*}
2\func{Re}F=2\func{Re}\log \varphi =\ln \left\vert \varphi \right\vert ^{2}.
\end{equation*}%
Indeed, to see this we consider the dilates $F_{r}\left( z\right) \equiv
F\left( rz\right) $ of $F$ which are clearly in $H^{\infty }\left( \mathbb{B}%
_{n}\right) $ (although not uniformly in $0<r<1$), and whose real parts are
uniformly bounded in $0<r<1$:%
\begin{equation*}
\func{Re}F_{r}\left( z\right) =\func{Re}F\left( rz\right) =\func{Re}\log
\varphi \left( rz\right) =\ln \left\vert \varphi \left( rz\right)
\right\vert \leq \left\Vert \ln \left\vert \varphi \right\vert \right\Vert
_{\infty }<\infty ,
\end{equation*}%
where the boundedness of $\ln \left\vert \varphi \right\vert $ follows from
the maximum principle since $\ln \left\vert \varphi \right\vert $ is
continuous in $\overline{\mathbb{B}_{n}}$ and harmonic in $\mathbb{B}_{n}$.
Then the Koranyi-Vagi theorem gives the uniform $L^{p}$ estimate (see e.g. 
\cite[inequality (1) on page 125]{Rud2})%
\begin{equation*}
\int_{\partial \mathbb{B}_{n}}\left\vert F_{r}\right\vert ^{p}d\sigma \leq
M_{p,n}\int_{\partial \mathbb{B}_{n}}\left\vert \func{Re}F_{r}\right\vert
^{p}d\sigma \leq M_{p,n}\left\Vert \ln \left\vert \varphi \right\vert
\right\Vert _{\infty }^{p},\ \ \ \ \ 1<p<\infty .
\end{equation*}%
Thus $F\in H^{p}\left( \mathbb{B}_{n}\right) $ for all $1<p<\infty $, and in
particular $F$ is in the Hardy space $H^{2}\left( \mathbb{B}_{n}\right) $.
Thus the restriction $h$ of $\ln \left\vert \varphi \right\vert ^{2}=\ln
\left( 1+\left\vert \widetilde{k_{a_{1}}}\right\vert ^{2}\right) $ to $%
\partial \mathbb{B}_{n}$ is the boundary value function of the real part $%
\func{Re}F$ of a holomorphic function $F$ in $H^{2}\left( \mathbb{B}%
_{n}\right) $. It follows that for almost every $\zeta \in \partial \mathbb{B%
}_{n}$, the slice function $F_{\zeta }\left( \lambda \right) =F\left(
\lambda \zeta \right) $ defined on the slice $S_{\zeta }=\left\{ z\in 
\mathbb{B}_{n}:z=\lambda \zeta ,\lambda \in \mathbb{D}\right\} $, is in $%
H^{2}\left( S_{\zeta }\right) \approx H^{2}\left( \mathbb{D}\right) $ and
has boundary values equal to $h\mid _{\partial S_{\zeta }}$ almost
everywhere. In particular then, the integral of $h=\ln \left\vert \varphi
\right\vert ^{2}$ on the boundary of such a slice $S_{\zeta }$, with respect
to Haar measure $dm$ on $\mathbb{T}=\partial S_{\zeta }$, must equal the
value of $h$ at the origin, i.e. the constant $\ln \left\vert \varphi \left(
0\right) \right\vert ^{2}$. Thus we have shown that the function%
\begin{equation*}
G\left( \zeta \right) \equiv \int_{\partial S_{\zeta }}\ln \left(
1+\left\vert \widetilde{k_{a_{1}}}\right\vert ^{2}\right)
dm=\int_{\left\vert \lambda \right\vert =1}\ln \left( 1+\left\vert \frac{%
\sqrt{1-\left\vert \alpha \right\vert ^{2}}}{1-\alpha \lambda \zeta _{1}}%
\right\vert ^{2n}\right) dm\left( \lambda \right)
\end{equation*}%
equals the constant $\ln \left\vert \varphi \left( 0\right) \right\vert ^{2}$
almost everywhere on the sphere $\partial \mathbb{B}_{n}$. However, it is
clear from the integral on the right hand side that the function $G$ is
continuous on the sphere $\partial \mathbb{B}_{n}$. In particular the map 
\begin{equation*}
g\left( z\right) \equiv G\left( z,0,\ldots,0,\sqrt{1-\left\vert z\right\vert
^{2}}\right) =\int_{\left\vert \lambda \right\vert =1}\ln \left(
1+\left\vert \frac{\sqrt{1-\left\vert \alpha \right\vert ^{2}}}{1-\alpha
\lambda z}\right\vert ^{2n}\right) dm\left( \lambda \right)
\end{equation*}%
is constant for $z\in \mathbb{D}$.

We now derive a contradiction by calculating that $g$ is \emph{not} constant
provided that $\alpha \neq 0$. In fact, $g$ is clearly twice differentiable
in the disk, and we will now show that $g$ has nonvanishing Laplacian in the
disk (this approach is suggested by a theorem of Forelli - see e.g. \cite[%
Theorem 4.4.4]{Rud3} - that shows $u:\mathbb{B}_{n}\rightarrow \mathbb{R}$
is the real part of a holomorphic function on the ball if and only if $u$ is
harmonic in each slice and smooth near the origin). We have 
\begin{equation*}
g\left( z\right) =\int_{\left\vert \lambda \right\vert =1}\ln A\left(
\lambda z\right) dm\left( \lambda \right) ,
\end{equation*}%
where 
\begin{equation*}
A\left( z\right) \equiv 1+\left\vert \frac{\sqrt{1-\left\vert \alpha
\right\vert ^{2}}}{1-\alpha z}\right\vert ^{2n}=1+\left( 1-\left\vert \alpha
\right\vert ^{2}\right) ^{n}\left\vert \left( 1-\alpha z\right)
^{-n}\right\vert ^{2}\equiv 1+c_{\alpha ,n}\left\vert \left( 1-\alpha
z\right) ^{-n}\right\vert ^{2}.
\end{equation*}%
Then we have%
\begin{eqnarray*}
\frac{\partial }{\partial z}A\left( \lambda z\right) &=&c_{\alpha ,n}%
\overline{\left( 1-\alpha \lambda z\right) ^{-n}}\left( -n\right) \left(
1-\alpha \lambda z\right) ^{n+1}\left( -\alpha \lambda \right) =c_{\alpha
,n}n\alpha \lambda \left\vert \left( 1-\alpha \lambda z\right)
^{-n}\right\vert ^{2}\left( 1-\alpha \lambda z\right) \\
&=&c_{\alpha ,n}n\alpha \lambda \left( 1-\alpha \lambda z\right) \left(
A\left( \lambda z\right) -1\right) ,
\end{eqnarray*}%
and of course $\frac{\partial }{\partial \overline{z}}A\left( \lambda
z\right) =\overline{\frac{\partial }{\partial z}A\left( \lambda z\right) }$
because $A$ is real. Thus%
\begin{eqnarray*}
&&\frac{\partial }{\partial \overline{z}}\frac{\partial }{\partial z}\ln
A\left( \lambda z\right) =\frac{\partial }{\partial \overline{z}}\left\{
c_{\alpha ,n}n\alpha \lambda \left( 1-\alpha \lambda z\right) \left( 1-\frac{%
1}{A\left( \lambda z\right) }\right) \right\} \\
&=&-c_{\alpha ,n}n\alpha \lambda \left( 1-\alpha \lambda z\right) \frac{%
\partial }{\partial \overline{z}}\frac{1}{A\left( \lambda z\right) }%
=c_{\alpha ,n}n\alpha \lambda \left( 1-\alpha \lambda z\right) \frac{1}{%
A\left( \lambda z\right) ^{2}}\frac{\partial }{\partial \overline{z}}A\left(
\lambda z\right) \\
&=&c_{\alpha ,n}n\alpha \lambda \left( 1-\alpha \lambda z\right) \frac{1}{%
A\left( \lambda z\right) ^{2}}\overline{\left\{ n\alpha \lambda \left(
1-\alpha \lambda z\right) \left( A\left( \lambda z\right) -1\right) \right\} 
}=c_{\alpha ,n}\left\vert n\alpha \lambda \left( 1-\alpha \lambda z\right)
\right\vert ^{2}\frac{A\left( z\right) -1}{A\left( \lambda z\right) ^{2}},
\end{eqnarray*}%
which is strictly positive for $\left\vert \lambda \right\vert =1$ if $%
\alpha \neq 0$. This shows that 
\begin{equation*}
\frac{1}{4}\bigtriangleup g\left( z\right) =\frac{\partial }{\partial 
\overline{z}}\frac{\partial }{\partial z}F\left( z\right) =\int_{\left\vert
\lambda \right\vert =1}\left\{ \frac{\partial }{\partial \overline{z}}\frac{%
\partial }{\partial z}\ln A\left( \lambda z\right) \right\} dm\left( \lambda
\right) >0
\end{equation*}%
for $z\in \mathbb{D}$, and this shows that $g$ is not constant in the disk,
providing the claimed contradiction. This completes the proof of Lemma \ref%
{ball failure}.
\end{proof}

Essentially the same argument as above shows that the Hardy space $%
H^{2}\left( \mathbb{D}^{n}\right) $ on the polydisc $\mathbb{D}^{n}$ also
fails to have the Invertible Multiplier Property when $n>1$.

\begin{lemma}
\label{polydisc failure}The Hardy space $\mathcal{H}=H^{2}\left( \mathbb{D}%
^{n}\right) $ on the polydisc $\mathbb{D}^{n}$ in $\mathbb{C}^{n}$ fails to
have the Invertible Multiplier Property when $n>1$.
\end{lemma}

\begin{proof}
We prove the Invertible Multiplier Property fails in the case when $M=1$, $%
\theta _{0}=\theta _{1}=\frac{1}{2}$ and $a_{1}=\alpha e_{1}=\left( \alpha
,\alpha ,\ldots,0\right) \in \mathbb{D}^{n}$ with $\alpha \neq 0$. We
assume, in order to derive a contradiction, that that there is a normalized
invertible multiplier $\varphi \in M_{\mathcal{H}}$ such that%
\begin{equation*}
\left\langle f,g\right\rangle _{\mathcal{H}}+\left\langle
k_{a_{1}}f,k_{a_{1}}g\right\rangle _{\mathcal{H}}=2\left\langle
f,g\right\rangle _{\mathcal{H}^{\mathbf{a,\theta }}}=\left\langle \varphi
f,\varphi g\right\rangle _{\mathcal{H}},\ \ \ \ \ f,g\in \mathcal{H},
\end{equation*}%
where again we have absorbed the factor $2$ into $\varphi $ for convenience.
Unraveling notation this becomes%
\begin{equation*}
\int_{\mathbb{T}^{n}}f\overline{g}\left( 1+\left\vert \widetilde{k_{a_{1}}}%
\right\vert ^{2}\right) dm=\int_{\mathbb{T}^{n}}f\overline{g}\left\vert
\varphi \right\vert ^{2}dm,\ \ \ \ \ f,g\in H^{2}\left( \mathbb{D}%
^{n}\right) .
\end{equation*}%
In particular we obtain that%
\begin{equation*}
1+\left\vert \widetilde{k_{a_{1}}}\right\vert ^{2}=\left\vert \varphi
\right\vert ^{2}\text{ on }\mathbb{T}^{n}
\end{equation*}%
by the Stone-Weierstrass theorem, just as in the proof of Lemma \ref{ball
failure}. Since $\varphi $ is an \emph{invertible} multiplier, the argument
used in the proof of Lemma \ref{ball failure} shows it has a holomorphic
logarithm $F=\log \varphi $ in $H^{2}\left( \mathbb{D}^{n}\right) $ with
boundary values $\ln \left\vert \varphi \right\vert ^{2}=2\func{Re}\log
\varphi =2\func{Re}F$. Thus the restriction $h$ of $\ln \left( 1+\left\vert 
\widetilde{k_{a_{1}}}\right\vert ^{2}\right) $ to $\mathbb{T}^{n}$ is the
distinguished boundary value function of the real part $\func{Re}F$ of a
holomorphic function $F$ in $H^{2}\left( \mathbb{D}^{n}\right) $. It follows
that for almost every $\zeta \in \mathbb{T}^{n}$, the slice function $%
F_{\zeta }\left( \lambda \right) =F\left( \lambda \zeta \right) $ defined on
the slice $S_{\zeta }=\left\{ z\in \mathbb{D}^{n}:z=\lambda \zeta ,\lambda
\in \mathbb{D}\right\} $, is in $H^{2}\left( S_{\zeta }\right) \approx
H^{2}\left( \mathbb{D}\right) $ and has boundary values equal to $h\mid _{%
\mathbb{T}^{n}}$. In particular then, the integral of $h$ on the boundary of
such a slice $S_{\zeta }$, with respect to Haar measure $dm$ on $\mathbb{T}%
=\partial S_{\zeta }$, must equal the constant $\ln \left\vert \varphi
\left( 0\right) \right\vert ^{2}$. Thus we have shown that the function%
\begin{equation*}
G\left( \zeta \right) \equiv \int_{\partial S_{\zeta }}\ln \left(
1+\left\vert \widetilde{k_{a_{1}}}\right\vert ^{2}\right)
dm=\int_{\left\vert \lambda \right\vert =1}\ln \left( 1+\left\vert \frac{%
\sqrt{1-\left\vert \alpha \right\vert ^{2}}}{1-\alpha \lambda \zeta _{1}}%
\frac{\sqrt{1-\left\vert \alpha \right\vert ^{2}}}{1-\alpha \lambda \zeta
_{2}}\right\vert ^{2}\right) dm\left( \lambda \right)
\end{equation*}%
equals the constant $\ln \left\vert \varphi \left( 0\right) \right\vert ^{2}$
almost everywhere on the distinguished boundary $\mathbb{T}^{n}$. In
particular the map 
\begin{equation*}
g\left( z\right) \equiv G\left( z,z,0,\ldots,0\right) =\int_{\left\vert
\lambda \right\vert =1}\ln \left( 1+\left\vert \left( \frac{\sqrt{%
1-\left\vert \alpha \right\vert ^{2}}}{1-\alpha \lambda z}\right)
^{2}\right\vert ^{2}\right) dm\left( \lambda \right)
\end{equation*}%
is the constant $\ln \left\vert \varphi \left( 0\right) \right\vert ^{2}$
almost everywhere on $\mathbb{D}$, and since $g$ is clearly continuous in $%
\mathbb{D}$, it is constant in $\mathbb{D}$. But this function $g$ is the
same as the function $g$ obtained in the previous proof for $n=2$. We showed
there that this function is not constant, and this completes the proof of
Lemma \ref{polydisc failure}.
\end{proof}

We emphasize that $H^{2}\left( \Omega \right) $ for certain bounded finitely
connected planar domains are essentially the only spaces we currently know
to have the Invertible Multiplier Property ($H^{2}\left( \Omega \right) $
has this property if $\Omega $ is simply connected by the Riemann mapping
theorem).

\section{Smoothness spaces and nonholomorphic spaces}

In the first subsection, we show that if $\mathcal{H}$ is a Besov-Sobolev
space $B_{2}^{\sigma }\left( \mathbb{B}_{n}\right) $ on the ball for $\sigma
>0$ and $n\geq 1$, then the kernel multiplier space $K_{\mathcal{H}}$ is an
algebra. Moreover, the reproducing kernels satisfy the hypotheses of the
alternate Toeplitz corona theorem \ref{alternate}, so that as a consequence
of this subsection we obtain the following corollary.

\begin{corollary}
Let $\sigma >0$ and $n\geq 1$. The kernel multiplier algebra $K_{2}^{\sigma
}\left( \mathbb{B}_{n}\right)$ satisfies the Corona Property if and only if
the Convex Poisson Property holds for $B_{2}^{\sigma }\left( \mathbb{B}%
_{n}\right) $.
\end{corollary}

Then in the second subsection we make the obvious extension of kernel multiplier algebras for the Banach spaces of analytic functions $B_p^{\sigma}(\mathbb{D})$ and then prove the Corona Property for the kernel
multiplier algebras $K_{B_{p}^{\sigma }\left( \mathbb{D}\right) }$ on the
disk when $0\leq \sigma < \frac{1}{p}$. We recall here that the algebra $%
B_{2}^{0}\left( \mathbb{D}\right)^{\infty}=H^{\infty }\left( \mathbb{D}\right)
\cap B_{2}^{0}\left( \mathbb{D}\right) $ of bounded Dirichlet space functions
was shown to have the Corona Property by Nicolau \cite{Nic}, who used the
difficult theory of best approximation in $VMO$ due to Peller and Hruscev 
\cite{PeHr}.

Finally, in the third subsection, we demonstrate that there are many Hilbert
function spaces to which our alternate Toeplitz corona theorem applies, that
are \emph{not} spaces of holomorphic functions.

\subsection{The kernel multiplier algebras for Besov-Sobolev spaces}

We begin by extending some of the background material developed in \cite%
{ArRoSa2} for the Besov spaces $B_{p}\left( \mathbb{B}_{n}\right) $ to the
Besov-Sobolev spaces $B_{p}^{\sigma }\left( \mathbb{B}_{n}\right) $, $\sigma
\geq 0$. Some of this material appears in \cite{ArRoSa3}.

\subsubsection{Besov-Sobolev spaces}

Recall the invertible \textquotedblleft radial\textquotedblright\ operators $%
R^{\gamma ,t}:H\left( \mathbb{B}_{n}\right) \rightarrow H\left( \mathbb{B}%
_{n}\right) $ given in \cite{Zhu} by 
\begin{equation*}
R^{\gamma ,t}f\left( z\right) =\sum_{k=0}^{\infty }\frac{\Gamma \left(
n+1+\gamma \right) \Gamma \left( n+1+k+\gamma +t\right) }{\Gamma \left(
n+1+\gamma +t\right) \Gamma \left( n+1+k+\gamma \right) }f_{k}\left(
z\right) ,
\end{equation*}%
provided neither $n+\gamma $ nor $n+\gamma +t$ is a negative integer, and
where $f\left( z\right) =\sum_{k=0}^{\infty }f_{k}\left( z\right) $ is the
homogeneous expansion of $f$. If the inverse of $R^{\gamma ,t}$ is denoted $%
R_{\gamma ,t}$, then Proposition 1.14 of \cite{Zhu} yields 
\begin{eqnarray}
R^{\gamma ,t}\left( \frac{1}{\left( 1-\overline{w}\cdot z\right)
^{n+1+\gamma }}\right) &=&\frac{1}{\left( 1-\overline{w}\cdot z\right)
^{n+1+\gamma +t}},  \label{Zhuidentity} \\
R_{\gamma ,t}\left( \frac{1}{\left( 1-\overline{w}\cdot z\right)
^{n+1+\gamma +t}}\right) &=&\frac{1}{\left( 1-\overline{w}\cdot z\right)
^{n+1+\gamma }},  \notag
\end{eqnarray}%
for all $w\in \mathbb{B}_{n}$. Thus for any $\gamma $, $R^{\gamma ,t}$ is
approximately differentiation of order $t$. From Theorem 6.1 and Theorem 6.4
of \cite{Zhu} we have that the derivatives $R^{\gamma ,m}f\left( z\right) $
are \textquotedblleft $L^{p}$ norm equivalent\textquotedblright\ to $%
\sum_{k=0}^{m-1}\left\vert f^{\left( k\right) }\left( 0\right) \right\vert
+f^{\left( m\right) }\left( z\right) $ for $m$ large enough.

\begin{proposition}
\label{Bequiv}(analogue of Theorem 6.1 and Theorem 6.4 of \cite{Zhu})
Suppose that $0<p<\infty $, $0\leq \sigma <\infty $, $n+\gamma $ is not a
negative integer, and $f\in H\left( \mathbb{B}_{n}\right) $. Then the
following four conditions are equivalent: 
\begin{eqnarray*}
\left( 1-\left\vert z\right\vert ^{2}\right) ^{m+\sigma }f^{\left( m\right)
}\left( z\right) &\in &L^{p}\left( d\lambda _{n}\right) \,\text{for \emph{%
some }}m+\sigma >\frac{n}{p},m\in \mathbb{N}, \\
\left( 1-\left\vert z\right\vert ^{2}\right) ^{m+\sigma }f^{\left( m\right)
}\left( z\right) &\in &L^{p}\left( d\lambda _{n}\right) \text{ for \emph{all 
}}m+\sigma >\frac{n}{p},m\in \mathbb{N}, \\
\left( 1-\left\vert z\right\vert ^{2}\right) ^{m+\sigma }R^{\gamma
,m}f\left( z\right) &\in &L^{p}\left( d\lambda _{n}\right) \text{ for \emph{%
some }}m+\sigma >\frac{n}{p},m+n+\gamma \notin -\mathbb{N}, \\
\left( 1-\left\vert z\right\vert ^{2}\right) ^{m+\sigma }R^{\gamma
,m}f\left( z\right) &\in &L^{p}\left( d\lambda _{n}\right) \,\text{for \emph{%
all }}m+\sigma >\frac{n}{p},m+n+\gamma \notin -\mathbb{N}.
\end{eqnarray*}%
Moreover, with $\rho \left( z\right) =1-\left\vert z\right\vert ^{2}$, we
have for $1<p<\infty $, 
\begin{equation}
C^{-1}\left\Vert \rho ^{m_{1}+\sigma }R^{\gamma ,m_{1}}f\right\Vert
_{L^{p}\left( d\lambda _{n}\right) }\leq \sum_{k=0}^{m_{2}-1}\left\vert
f^{\left( k\right) }\left( 0\right) \right\vert +\left( \int_{\mathbb{B}%
_{n}}\left\vert \left( 1-\left\vert z\right\vert ^{2}\right) ^{m_{2}+\sigma
}f^{m_{2}}\left( z\right) \right\vert ^{p}d\lambda _{n}\left( z\right)
\right) ^{\frac{1}{p}}\leq C\left\Vert \rho ^{m_{1}+\sigma }R^{\gamma
,m_{1}}f\right\Vert _{L^{p}\left( d\lambda _{n}\right) }  \label{normequiv}
\end{equation}%
for all $m_{1}+\sigma ,m_{2}+\sigma >\frac{n}{p}$, $m_{1}+n+\gamma \notin -%
\mathbb{N}$, $m_{2}\in \mathbb{N}$, and where the constant $C$ depends only
on $\sigma $, $m_{1}$, $m_{2}$, $n$, $\gamma $ and $p$.
\end{proposition}

\begin{definition}
We define the analytic Besov-Sobolev spaces $B_{p}^{\sigma }\left( \mathbb{B}%
_{n}\right) $ on the ball $\mathbb{B}_{n}$ by taking $\gamma =0$ and $m=%
\frac{n+1}{p}$ and setting 
\begin{equation}
B_{p}^{\sigma }=B_{p}^{\sigma }\left( \mathbb{B}_{n}\right) =\left\{ f\in
H\left( \mathbb{B}_{n}\right) :\left\| \rho ^{m+\sigma }R^{0,m}f\right\|
_{L^{p}\left( d\lambda _{n}\right) }<\infty \right\} .  \label{Besovspacedef}
\end{equation}
\end{definition}

We will indulge in the usual abuse of notation by using $\left\Vert
f\right\Vert _{B_{p}^{\sigma }\left( \mathbb{B}_{n}\right) }$ to denote any
of the norms appearing in (\ref{normequiv}).

\subsubsection{Reproducing kernels}

For $\alpha >-1$, let $\left\langle \cdot ,\cdot \right\rangle _{\alpha }$
denote the inner product for the weighted Bergman space $A_{\alpha }^{2}$: 
\begin{equation*}
\left\langle f,g\right\rangle _{\alpha }=\int_{\mathbb{B}_{n}}f\left(
z\right) \overline{g\left( z\right) }d\nu _{\alpha }\left( z\right)
,\;\;\;\;f,g\in A_{\alpha }^{2},
\end{equation*}%
where $d\nu _{\alpha }\left( z\right) =\left( 1-\left\vert z\right\vert
^{2}\right) ^{\alpha }dV(z)$. Recall that $K_{w}^{\alpha }\left( z\right)
=K^{\alpha }\left( z,w\right) =\left( 1-\overline{w}\cdot z\right)
^{-n-1-\alpha }$ is the reproducing kernel for $A_{\alpha }^{2}$ (Theorem
2.7 in \cite{Zhu}): 
\begin{equation*}
f\left( w\right) =\left\langle f,K_{w}^{\alpha }\right\rangle _{\alpha
}=\int_{\mathbb{B}_{n}}f\left( z\right) \overline{K_{w}^{\alpha }\left(
z\right) }d\nu _{\alpha }\left( z\right) ,\;\;\;\;\;f\in A_{\alpha }^{2}.
\end{equation*}%
This formula continues to hold as well for $f\in A_{\alpha }^{p}$, $%
1<p<\infty $, since the polynomials are dense in $A_{\alpha }^{p}$.

The proof of Corollary 6.5 of \cite{Zhu} shows that $R^{\gamma ,\frac{%
n+1+\alpha }{p}-\sigma }$ is a bounded invertible operator from $%
B_{p}^{\sigma }$ onto $A_{\alpha }^{p}$, provided that neither $n+\gamma $
nor $n+\gamma +\frac{n+1+\alpha }{p}-\sigma $ is a negative integer. It
turns out to be convenient to take $\gamma =\alpha -\frac{n+1+\alpha }{p}%
+\sigma $ here (with this choice we can explicitly compute certain formulas
- see (\ref{comp}) below), and thus we single out the special operators 
\begin{equation*}
\mathcal{R}_{t}^{\alpha }=R^{\alpha -t,t}.
\end{equation*}
Note that the operators $\mathcal{R}_{t}^{\alpha }$ and their inverses $%
\left( \mathcal{R}_{t}^{\alpha }\right) ^{-1}\mathcal{=}R_{\alpha -t,t}$ are
self-adjoint with respect to $\left\langle \cdot ,\cdot \right\rangle
_{\alpha }$ since the monomials are orthogonal with respect to $\left\langle
\cdot ,\cdot \right\rangle _{\alpha }$ (see (1.21) and (1.23) in \cite{Zhu}%
), and the operators act on the homogeneous expansion of $f$ by multiplying
the homogeneous coefficients of $f$ by certain positive constants. The next
definition is motivated by the fact that $\mathcal{R}_{\frac{n+1+\alpha }{p}%
-\sigma }^{\alpha }$ is a bounded invertible operator from $B_{p}^{\sigma }$
onto $A_{\alpha }^{p}$, and that $\mathcal{R}_{\frac{n+1+\alpha }{p^{\prime }%
}-\sigma }^{\alpha }$ is a bounded invertible operator from $B_{p^{\prime
}}^{\sigma }$ onto $A_{\alpha }^{p^{\prime }}$, provided that neither $%
n+\alpha $, $n+\alpha -\frac{n+1+\alpha }{p}+\sigma $ nor $n+\alpha -\frac{%
n+1+\alpha }{p^{\prime }}+\sigma $ is a negative integer. Note that this
proviso holds in particular for $\alpha >-1$, $\sigma \geq 0$.

\begin{definition}
For $\alpha >-1$, $\sigma \geq 0$ and $1<p<\infty $, we define a pairing $%
\left\langle \cdot ,\cdot \right\rangle _{\alpha ,p}^{\sigma }$ for $%
B_{p}^{\sigma }$ and $B_{p^{\prime }}^{\sigma }$ using $\left\langle \cdot
,\cdot \right\rangle _{\alpha }$ as follows: 
\begin{eqnarray*}
\left\langle f,g\right\rangle _{\alpha ,p}^{\sigma } &=&\left\langle 
\mathcal{R}_{\frac{n+1+\alpha }{p}-\sigma }^{\alpha }f,\mathcal{R}_{\frac{%
n+1+\alpha }{p^{\prime }}-\sigma }^{\alpha }g\right\rangle _{\alpha }=\int_{%
\mathbb{B}_{n}}\mathcal{R}_{\frac{n+1+\alpha }{p}-\sigma }^{\alpha }f\left(
z\right) \overline{\mathcal{R}_{\frac{n+1+\alpha }{p^{\prime }}-\sigma
}^{\alpha }g\left( z\right) }d\nu _{\alpha }\left( z\right) \\
&=&\int_{\mathbb{B}_{n}}\left\{ \left( 1-\left| z\right| ^{2}\right) ^{\frac{%
n+1+\alpha }{p}}\mathcal{R}_{\frac{n+1+\alpha }{p}-\sigma }^{\alpha }f\left(
z\right) \right\} \overline{\left\{ \left( 1-\left| z\right| ^{2}\right) ^{%
\frac{n+1+\alpha }{p^{\prime }}}\mathcal{R}_{\frac{n+1+\alpha }{p^{\prime }}%
-\sigma }^{\alpha }g\left( z\right) \right\} }d\lambda _{n}\left( z\right) .
\end{eqnarray*}
\end{definition}

Clearly we have 
\begin{equation*}
\left| \left\langle f,g\right\rangle _{\alpha ,p}^{\sigma }\right| \leq
\left\| f\right\| _{B_{p}^{\sigma }}\left\| g\right\| _{B_{p^{\prime
}}^{\sigma }}
\end{equation*}
by H\"{o}lder's inequality. By Theorem 2.12 of \cite{Zhu}, we also have that
every continuous linear functional $\Lambda $ on $B_{p}^{\sigma }$ is given
by $\Lambda f=\left\langle f,g\right\rangle _{\alpha ,p}^{\sigma }$ for a
unique $g\in B_{p^{\prime }}^{\sigma }$ satisfying 
\begin{equation}
\left\| g\right\| _{B_{p^{\prime }}^{\sigma }}=\sup_{\left\| f\right\|
_{B_{p}^{\sigma }}=1}\left| \left\langle f,g\right\rangle _{\alpha
,p}^{\sigma }\right| .  \label{attain}
\end{equation}
Indeed, if $\Lambda \in \left( B_{p}^{\sigma }\right) ^{*}$, then $\Lambda
\circ \left( \mathcal{R}_{\frac{n+1+\alpha }{p}-\sigma }^{\alpha }\right)
^{-1}\in \left( A_{\alpha }^{p}\right) ^{*}$, and by Theorem 2.12 of \cite%
{Zhu}, there is $G\in A_{\alpha }^{p^{\prime }}$ with $\left\| G\right\|
_{A_{\alpha }^{p^{\prime }}}=\left\| \Lambda \right\| $ such that $\Lambda
\circ \left( \mathcal{R}_{\frac{n+1+\alpha }{p}-\sigma }^{\alpha }\right)
^{-1}F=\left\langle F,G\right\rangle _{\alpha }$ for all $F\in A_{\alpha
}^{p}$. If we set $g=\left( \mathcal{R}_{\frac{n+1+\alpha }{p^{\prime }}%
-\sigma }^{\alpha }\right) ^{-1}G$, then we have $\left\| g\right\|
_{B_{p^{\prime }}^{\sigma }}=\left\| G\right\| _{A_{\alpha }^{p^{\prime
}}}=\left\| \Lambda \right\| $ and with $F=\mathcal{R}_{\frac{n+1+\alpha }{p}%
-\sigma }^{\alpha }f$, we also have 
\begin{equation*}
\Lambda f=\Lambda \circ \left( \mathcal{R}_{\frac{n+1+\alpha }{p}-\sigma
}^{\alpha }\right) ^{-1}F=\left\langle F,G\right\rangle _{\alpha
}=\left\langle \mathcal{R}_{\frac{n+1+\alpha }{p}-\sigma }^{\alpha }f,%
\mathcal{R}_{\frac{n+1+\alpha }{p^{\prime }}-\sigma }^{\alpha
}g\right\rangle _{\alpha }=\left\langle f,g\right\rangle _{\alpha
,p}^{\sigma }
\end{equation*}
for all $f\in B_{p}^{\sigma }$. Then (\ref{attain}) follows from 
\begin{equation*}
\left\| g\right\| _{B_{p^{\prime }}^{\sigma }}=\left\| \Lambda \right\|
=\sup_{\left\| f\right\| _{B_{p}^{\sigma }}=1}\left| \Lambda \left( f\right)
\right| =\sup_{\left\| f\right\| _{B_{p}^{\sigma }}=1}\left| \left\langle
f,g\right\rangle _{\alpha ,p}^{\sigma }\right| .
\end{equation*}

With $K_{w}^{\alpha }\left( z\right) $ the reproducing kernel for $A_{\alpha
}^{2}$, we now claim that the kernel 
\begin{equation}
k_{w}^{\sigma ,\alpha ,p}\left( z\right) =\left( \mathcal{R}_{\frac{%
n+1+\alpha }{p^{\prime }}-\sigma }^{\alpha }\right) ^{-1}\left( \mathcal{R}_{%
\frac{n+1+\alpha }{p}-\sigma }^{\alpha }\right) ^{-1}K_{w}^{\alpha }\left(
z\right)  \label{kerdef}
\end{equation}%
satisfies the following reproducing formula for $B_{p}^{\sigma }$: 
\begin{equation}
f\left( w\right) =\left\langle f,k_{w}^{\sigma ,\alpha ,p}\right\rangle
_{\alpha ,p}^{\sigma }=\int_{\mathbb{B}_{n}}\mathcal{R}_{\frac{n+1+\alpha }{p%
}-\sigma }^{\alpha }f\left( z\right) \overline{\mathcal{R}_{\frac{n+1+\alpha 
}{p^{\prime }}-\sigma }^{\alpha }k_{w}^{\sigma ,\alpha ,p}\left( z\right) }%
d\nu _{\alpha }\left( z\right) ,\;\;\;\;\;f\in B_{p}^{\sigma }.  \label{kw}
\end{equation}%
Indeed, for $f$ a polynomial, we have 
\begin{eqnarray*}
f\left( w\right) &=&\left\langle f,K_{w}^{\alpha }\right\rangle _{\alpha
}=\left\langle \left( \mathcal{R}_{\frac{n+1+\alpha }{p}-\sigma }^{\alpha
}\right) ^{-1}\mathcal{R}_{\frac{n+1+\alpha }{p}-\sigma }^{\alpha
}f,K_{w}^{\alpha }\right\rangle _{\alpha }=\left\langle \mathcal{R}_{\frac{%
n+1+\alpha }{p}-\sigma }^{\alpha }f,\left( \mathcal{R}_{\frac{n+1+\alpha }{p}%
-\sigma }^{\alpha }\right) ^{-1}K_{w}^{\alpha }\right\rangle _{\alpha } \\
&=&\left\langle \mathcal{R}_{\frac{n+1+\alpha }{p}-\sigma }^{\alpha }f,%
\mathcal{R}_{\frac{n+1+\alpha }{p^{\prime }}-\sigma }^{\alpha }\left( 
\mathcal{R}_{\frac{n+1+\alpha }{p^{\prime }}-\sigma }^{\alpha }\right)
^{-1}\left( \mathcal{R}_{\frac{n+1+\alpha }{p}-\sigma }^{\alpha }\right)
^{-1}K_{w}^{\alpha }\right\rangle _{\alpha } \\
&=&\left\langle f,\left( \mathcal{R}_{\frac{n+1+\alpha }{p^{\prime }}-\sigma
}^{\alpha }\right) ^{-1}\left( \mathcal{R}_{\frac{n+1+\alpha }{p}-\sigma
}^{\alpha }\right) ^{-1}K_{w}^{\alpha }\right\rangle _{\alpha ,p}^{\sigma }.
\end{eqnarray*}%
We now obtain the claim since the polynomials are dense in $B_{p}^{\sigma }$
and the kernels $k_{w}^{\sigma ,\alpha ,p}$ are in $B_{p^{\prime }}^{\sigma
} $ for each fixed $w\in \mathbb{B}_{n}$. Thus we have proved the following
theorem.

\begin{theorem}
\label{pairing}Let $1<p<\infty $, $\sigma \geq 0$ and $\alpha >-1$. Then the
dual space of $B_{p}^{\sigma }$ can be identified with $B_{p^{\prime
}}^{\sigma }$ under the pairing $\left\langle \cdot ,\cdot \right\rangle
_{\alpha ,p}^{\sigma }$, and the reproducing kernel $k_{w}^{\sigma ,\alpha
,p}$ for this pairing is given by (\ref{kerdef}).
\end{theorem}

From (\ref{kerdef}) and (\ref{Zhuidentity}) we have 
\begin{eqnarray}
\mathcal{R}_{\frac{n+1+\alpha }{p^{\prime }}-\sigma }^{\alpha }k_{w}^{\sigma
,\alpha ,p}\left( z\right) &=&\left( \mathcal{R}_{\frac{n+1+\alpha }{p}%
-\sigma }^{\alpha }\right) ^{-1}K_{w}^{\alpha }\left( z\right)
\label{comput} \\
&=&R_{\alpha -\frac{n+1+\alpha }{p}+\sigma ,\frac{n+1+\alpha }{p}-\sigma
}\left( \left( 1-\overline{w}\cdot z\right) ^{-\left( n+1+\alpha \right)
}\right)  \notag \\
&=&\left( 1-\overline{w}\cdot z\right) ^{-\frac{n+1+\alpha }{p^{\prime }}%
-\sigma }.  \notag
\end{eqnarray}

\subsubsection{Kernel multiplier spaces}

Define the Banach space $K_{p}^{\sigma }\left( \mathbb{B}_{n}\right) $ of
kernel multipliers of $B_{p}^{\sigma }\left( \mathbb{B}_{n}\right) $ to
consist of those holomorphic functions in the ball $\mathbb{B}_{n}$ such
that 
\begin{equation*}
\left\Vert \varphi \right\Vert _{K_{p}^{\sigma }\left( \mathbb{B}_{n}\right)
}\equiv \sup_{a\in \mathbb{B}_{n}}\frac{\left\Vert \varphi k_{a}^{\sigma
,p^{\prime}}\right\Vert _{B_{p}^{\sigma }\left( \mathbb{B}_{n}\right) }}{\left\Vert
k_{a}^{\sigma ,p^{\prime}}\right\Vert _{B_{p}^{\sigma }\left( \mathbb{B}_{n}\right) }}%
<\infty ,
\end{equation*}%
where $k_{a}^{\sigma ,p^{\prime}}\left( z\right) $ is a reproducing kernel for $%
B_{p^{\prime}}^{\sigma }\left( \mathbb{B}_{n}\right) $, e.g., $k_{w}^{\sigma ,\alpha ,p^{\prime}}\left(
z\right) $ - any admissible choice of $\alpha $ can be used here. Standard
arguments using the reproducing kernel $k_{a}^{\sigma ,p^{\prime}}\left( z\right) $
show that $K_{p}^{\sigma }\left( \mathbb{B}_{n}\right) $ embeds in $%
H^{\infty }\left( \mathbb{B}_{n}\right) $.   Indeed, 

\begin{equation*}
\left\vert \varphi \left( a\right) \right\vert =\frac{1}{ k_{a}^{\sigma,\alpha, p^{\prime}}\left(
a\right) }\left\vert \left\langle \varphi k_{a}^{\sigma, \alpha, p^{\prime}},k_{a}^{\sigma,\alpha, p}\right\rangle _{\alpha}\right\vert \approx \left\vert \left\langle \varphi \widetilde{k_{a}^{\sigma,\alpha,p^{\prime}}},%
\widetilde{k_{a}^{\sigma,\alpha,p}}\right\rangle _{\alpha}\right\vert\leq \left\Vert \varphi \right\Vert _{K_{%
\mathcal{H}}}\ ,
\end{equation*}
Above we have used (see \eqref{p bounds} below):
$$
\left\Vert k_{a}^{\sigma,\alpha,p^{\prime}} \right\Vert_{B^{p}_{\sigma}(\mathbb{B}_n)}\approx \frac{1}{(1-\left\vert a\right\vert^2)^{\sigma}}\approx \sqrt{k_a^{\sigma,\alpha,p^{\prime}}(a)}.
$$

Let $WC_{p}^{\sigma }\left( 
\mathbb{B}_{n}\right) $ consist of those holomorphic functions in $f\in
B_{p}^{\sigma }\left( \mathbb{B}_{n}\right) $ that satisfy the weak or
one-box $\sigma $-Carleson condition:%
\begin{eqnarray*}
WC_{p}^{\sigma }\left( \mathbb{B}_{n}\right) &\equiv &\left\{ f\in
B_{p}^{\sigma }\left( \mathbb{B}_{n}\right) :\left\Vert f\right\Vert
_{WC_{p}^{\sigma }\left( \mathbb{B}_{n}\right) }<\infty \right\} ; \\
\left\Vert f\right\Vert _{WC_{p}^{\sigma }\left( \mathbb{B}_{n}\right)
;m}^{p} &\equiv &\sup_{Q}\frac{1}{\left\vert Q\right\vert ^{\frac{p\sigma}{n} }}%
\int_{S\left( Q\right) }\left\vert \left( 1-\left\vert z\right\vert
^{2}\right) ^{m+\sigma }\nabla ^{m}f\left( z\right) \right\vert ^{p}d\lambda
_{n}\left( z\right) ,
\end{eqnarray*}%
and where $Q$ is a nonisotropic ball on the sphere (see e.g. \cite[page 65]%
{Rud2}) and $\left\vert Q\right\vert $ is its surface measure. Here $m>\frac{%
n}{p}-\sigma $ and as usual, we will see below that the norms $\left\Vert
f\right\Vert _{WC_{p}^{\sigma }\left( \mathbb{B}_{n}\right) ;m}^{p}$ are
equivalent provided $p\left( m+\sigma \right) >n$, and so we can drop the
dependence on $m$. But first we establish the following standard equivalence
for one-box Carleson measures.

\begin{lemma}
\label{one box equiv}For $d\mu $ a positive Borel measure on the ball $%
\mathbb{B}_{n}$, $\sigma >0$ and $1<p<\infty $, we have 
\begin{equation*}
\label{wcm}
\left\Vert \mu\right\Vert_{WCM_{p}^{\sigma}}\equiv \sup_{a\in \mathbb{B}_{n}}\left( \int_{\mathbb{B}_{n}}\left\vert\widetilde{
k_{a}^{\sigma ,p^{\prime}}\left( z\right)} \right\vert ^{p}d\mu \left( z\right)
\right) ^{\frac{1}{p}}\approx \sup_{Q\subset \partial \mathbb{B}_{n}}\frac{%
\left( \int_{S\left( Q\right) }d\mu \left( z\right) \right) ^{\frac{1}{p}}}{%
\left\vert  Q \right\vert ^{\frac{\sigma}{n} }}.
\end{equation*}
\end{lemma}
We will call measures $\mu$ that satisfy Lemma \ref{one box equiv} weak Carleson measures and norm them via either expression.  This can be contrasted with standard Carleson measures which are normed by:
$$
\left\Vert \mu\right\Vert_{CM_p^\sigma}=\sup_{\left\Vert f\right\Vert_{B_p^\sigma(\mathbb{B}_n)}=1} \left(\int_{\mathbb{B}_n} \left\vert f(z)\right\vert^{p}d\mu(z)\right)^{\frac{1}{p}}.
$$

\begin{proof}
The proof of Lemma \ref{one box equiv} is standard using the pointwise
bounds 
\begin{eqnarray}
\left\vert k_{a}^{\sigma ,p^{\prime}}\left( z\right) \right\vert &\gtrsim &\left( 
\frac{1}{1-\left\vert a\right\vert ^{2}}\right) ^{2\sigma },\ \ \ \ \ z\in
S_{a}\ ,  \label{p bounds} \\
\left\vert k_{a}^{\sigma ,p^{\prime}}\left( z\right) \right\vert &\lesssim
&\left\vert \frac{1}{1-\overline{a}\cdot z}\right\vert ^{2\sigma },\ \ \ \ \
z\in \mathbb{B}_{n}\ ,  \notag
\end{eqnarray}%
which follow from (\ref{comput}), $\mathcal{R}_{\frac{n+1+\alpha }{p}-\sigma }^{\alpha }k_{w}^{\sigma ,\alpha ,p^{\prime}}\left( z\right) =\frac{1}{%
\left( 1-\overline{w}\cdot z\right) ^{\frac{n+1+\alpha }{p}+\sigma
}}$, since $\mathcal{R}_{\frac{n+1+\alpha }{p}-\sigma }^{\alpha }$
is essentially differentiation of order $t=\frac{n+1+\alpha }{p}%
-\sigma $, and so 
\begin{equation*}
k_{w}^{\sigma ,\alpha ,p^{\prime}}\left( z\right) =\left( \mathcal{R}_{\frac{%
n+1+\alpha }{p}-\sigma }^{\alpha }\right) ^{-1}\frac{1}{\left( 1-%
\overline{w}\cdot z\right) ^{\frac{n+1+\alpha }{p}+\sigma }}%
\approx \frac{1}{\left( 1-\overline{w}\cdot z\right) ^{\frac{n+1+\alpha }{%
p}+\sigma -t}}=\frac{1}{\left( 1-\overline{w}\cdot z\right)
^{2\sigma }}.
\end{equation*}
From this we also have the use approximation that:
$$
\widetilde{k_{a}^{\sigma,p^{\prime}}}(z)\approx \frac{(1-\left\vert a\right\vert^2)^{\sigma}}{(1-\overline{a}\cdot z)^{2\sigma}}.
$$
To show the inequality $\gtrsim $ we simply use the first inequality in (\ref%
{p bounds}). Conversely, to show the inequality $\lesssim $, we break up the integral 
\begin{equation*}
\left\Vert \widetilde{k_{a}^{\sigma ,p^{\prime}}}\right\Vert _{L^{p}\left( \mu \right)
}^{p}=\int_{\mathbb{B}_{n}}\left\vert \widetilde{k_{a}^{\sigma ,p^{\prime}}}\left( z\right)
\right\vert ^{p}d\mu \left( z\right)
\end{equation*}%
into geometric annuli $2^{\ell +1}S_{a}\setminus 2^{\ell }S_{a}$, where $%
S_{a}$ is the usual Carleson box associated with $a\in \mathbb{B}_{n}$. Then
we complete the proof using the one box condition on the Carleson boxes $%
2^{\ell +1}S_{a}$ together with the second inequality in (\ref{p bounds}),
and then summing up a geometric series.
\end{proof}

Now we show that $K_{p}^{\sigma }\left( \mathbb{B}_{n}\right) =H^{\infty
}\left( \mathbb{B}_{n}\right) \cap WC_{p}^{\sigma }\left( \mathbb{B}%
_{n}\right) $ with comparable norms. The corresponding result for the
multiplier algebra $M_{B_{p}^{\sigma }\left( \mathbb{B}_{n}\right) }$ is due
to Ortega and Fabrega \cite[Theorem 3.7]{OrFa}, and the proof there carries
over almost verbatim for weak Carleson measures in place of Carleson
measures, which we provide for the ease of the reader. It follows as a corollary of this result that the weak Carleson
measure condition is independent of $m>\frac{n}{p}-\sigma $.

\begin{proposition}
Let $\varphi \in H^{\infty }\left( \mathbb{B}_{n}\right) \cap B_{p}^{\sigma
}\left( \mathbb{B}_{n}\right) $, $\sigma >0$ and $m+\sigma >\frac{n}{p}$.  Then $\varphi \in K_{p}^{\sigma }\left( 
\mathbb{B}_{n}\right) $ if and only if $\varphi\in WC_{p}^{\sigma}(\mathbb{B}_n)$ i.e.
\begin{equation*}
d\mu \left( z\right) \equiv \left\vert \left( 1-\left\vert z\right\vert
^{2}\right) ^{m+\sigma }\varphi ^{\left( m\right) }\left( z\right)
\right\vert ^{p}d\lambda _{n}\left( z\right)
\end{equation*}%
is a weak $B_{p}^{\sigma }\left( \mathbb{B}_{n}\right) $-Carleson measure on 
$\mathbb{B}_{n}$.
\end{proposition}

\begin{proof}
Suppose that $\varphi \in H^{\infty }\left( \mathbb{B}_{n}\right) \cap
WC_{p}^{\sigma }\left( \mathbb{B}_{n}\right) $. In \cite{OrFa}, Ortega and
Fabrega use the notation $A_{\delta ,k}^{p}\left( \mathbb{B}_{n}\right) $ to
describe a space of holomorphic functions equivalent to $B_{p}^{\sigma
}\left( \mathbb{B}_{n}\right) $ when $\sigma =\frac{n+\delta }{p}-k$. They
define the norm on $A_{\delta ,k}^{p}\left( \mathbb{B}_{n}\right) $ by 
\begin{equation*}
\left\Vert f\right\Vert _{A_{\delta ,k}^{p}\left( \mathbb{B}_{n}\right)
}\equiv \left( \sum_{\left\vert \alpha \right\vert \leq k}\int_{\mathbb{B}%
_{n}}\left\vert D^{\alpha }f\left( z\right) \right\vert ^{p}\left(
1-\left\vert z\right\vert ^{2}\right) ^{\delta -1}dV\left( z\right) \right)
^{\frac{1}{p}}.
\end{equation*}%
They show (bottom of page 66 of \cite{OrFa}) that with $R$ the radial
derivative,%
\begin{equation}
\left\Vert \left( I+R\right) ^{k}\left( \varphi k_{a}^{\sigma ,p^{\prime}}\right)
\right\Vert _{A_{\delta ,0}^{p}\left( \mathbb{B}_{n}\right) }\lesssim
\left\Vert k_{a}^{\sigma ,p^{\prime}}\left( I+R\right) ^{k}\varphi \right\Vert
_{A_{\delta ,0}^{p}\left( \mathbb{B}_{n}\right) }+\sum_{\substack{ m+\ell
\leq k  \\ m>0}}\left\Vert \left( R^{\ell }\varphi \right) \left(
R^{m}k_{a}^{\sigma ,p^{\prime}}\right) \right\Vert _{A_{\delta ,0}^{p}\left( \mathbb{B%
}_{n}\right) }.  \label{OF}
\end{equation}%
Now the first term on the right side of (\ref{OF}) is controlled by the weak
Carleson norm $\left\Vert \varphi \right\Vert _{WC_{p}^{\sigma }\left( 
\mathbb{B}_{n}\right) }$ of $\varphi $. The needed estimates for the other
terms on the right hand side of (\ref{OF}) are obtained from the argument
used to prove Theorem 3.5 in \cite{OrFa} since only the weak Carleson
condition is needed here. Indeed we quote from the bottom of page 66 in \cite%
{OrFa}:\bigskip

``The estimates of the other terms can be obtained following the same
argument used to prove Theorem 3.5. Observe that the properties of $\varphi $
used in this proof were that $\varphi $ was a bounded holomorphic function
and that for $\delta -1-N-kp<0$%
\begin{equation*}
\int_{\mathbb{B}_{n}}\frac{\left\vert \left( I+R\right) ^{k}\varphi \left(
w\right) \right\vert ^{p}\left( 1-\left\vert w\right\vert ^{2}\right)
^{\delta -1}}{\left\vert 1-\overline{w}\cdot z\right\vert ^{n+1+N}}dV\left(
w\right) \lesssim \left( 1-\left\vert z\right\vert ^{2}\right) ^{\delta
-1-N-kp}.
\end{equation*}%
It is clear that they are consequences of $\varphi \in H^{\infty }\left( 
\mathbb{B}_{n}\right) $ and just testing the measure $d\mu $ on the function 
$\left( \frac{1}{1-\overline{w}\cdot z}\right) ^{\frac{n+1+N}{p}}$."

\bigskip

From this quote it is clear that we need only use the weak Carleson
condition when testing the measure $d\mu $ for the other terms on the right
side of (\ref{OF}).
\end{proof}

Now we turn to showing that $K_{p}^{\sigma }\left( \mathbb{B}_{n}\right) $
is an algebra for $\sigma >0$ and $1<p<\infty $.

\begin{theorem}
Let $\sigma >0$ and $1<p<\infty $. Then $K_{p}^{\sigma }\left( \mathbb{B}%
_{n}\right) $ is an algebra.
\end{theorem}

\begin{proof}
Fix $m>\frac{2n}{p}$. Note this choice is twice as large as need be, and
this will play a role in the proof. To show that $K_{p}^{\sigma }\left( 
\mathbb{B}_{n}\right) $ is an algebra, it is enough by polarization, $%
2fg=\left( f+g\right) ^{2}-f^{2}-g^{2}$, to show that $\left\Vert \varphi
^{2}\right\Vert _{K_{p}^{\sigma }\left( \mathbb{B}_{n}\right) }\leq
\left\Vert \varphi \right\Vert _{K_{p}^{\sigma }\left( \mathbb{B}_{n}\right)
}^{2}$. Now $\left\Vert \varphi ^{2}\right\Vert _{\infty }\leq \left\Vert
\varphi \right\Vert _{\infty }^{2}$ for $\varphi \in K_{p}^{\sigma }\left( 
\mathbb{B}_{n}\right) $ and using Lemma \ref{one box equiv} it remains to
show that 
\begin{equation}
\left\Vert \varphi ^{2}\right\Vert _{WC_{p}^{\sigma }\left( \mathbb{B}%
_{n}\right) }\lesssim \left\Vert \varphi \right\Vert _{K_{p}^{\sigma }\left( 
\mathbb{B}_{n}\right) }^{2}.  \label{RTS}
\end{equation}%
We have%
\begin{equation*}
\left\Vert \varphi ^{2}\right\Vert _{WC_{p}^{\sigma }\left( \mathbb{B}%
_{n}\right) }=\sum_{k=0}^{m-1}\left\vert \nabla ^{k}\left( \varphi
^{2}\right) \left( 0\right) \right\vert +\left( \sup_{Q}\frac{1}{\left\vert
Q\right\vert ^{\frac{p\sigma}{n} }}\int_{S\left( Q\right) }\left\vert \left(
1-\left\vert z\right\vert ^{2}\right) ^{m+\sigma }\nabla ^{m}\left( \varphi
^{2}\right) \left( z\right) \right\vert ^{p}d\lambda _{n}\left( z\right)
\right) ^{\frac{1}{p}},
\end{equation*}%
and 
\begin{equation*}
\nabla ^{m}\left( \varphi ^{2}\right) \left( z\right)
=\sum_{k=0}^{m}c_{m,k}\left( \nabla ^{m-k}\varphi \left( z\right) \right)
\left( \nabla ^{k}\varphi \left( z\right) \right) .
\end{equation*}%
Now 
\begin{eqnarray*}
&&\left( \int_{S\left( Q\right) }\left\vert \left( 1-\left\vert z\right\vert
^{2}\right) ^{m}\nabla ^{m}\left( \varphi ^{2}\right) \left( z\right)
\right\vert ^{p}\left( 1-\left\vert z\right\vert ^{2}\right) ^{\sigma
p}d\lambda _{n}\left( z\right) \right) ^{\frac{1}{p}} \\
&\leq &C\left( \int_{S\left( Q\right) }\left\vert \left( 1-\left\vert
z\right\vert ^{2}\right) ^{m}\nabla ^{m}\varphi \left( z\right) \right\vert
^{p}\left\vert \varphi \left( z\right) \right\vert ^{p}\left( 1-\left\vert
z\right\vert ^{2}\right) ^{\sigma p}d\lambda _{n}\left( z\right) \right) ^{%
\frac{1}{p}} \\
&&+C\sum_{k=1}^{m-1}\left( \int_{S\left( Q\right) }\left\vert \left(
1-\left\vert z\right\vert ^{2}\right) ^{m-k}\nabla ^{m-k}\varphi \left(
z\right) \right\vert ^{p}\left\vert \left( 1-\left\vert z\right\vert
^{2}\right) ^{k}\nabla ^{k}\varphi \left( z\right) \right\vert ^{p}\left(
1-\left\vert z\right\vert ^{2}\right) ^{\sigma p}d\lambda _{n}\left(
z\right) \right) ^{\frac{1}{p}} \\
&&+C\left( \int_{S\left( Q\right) }\left\vert \varphi \left( z\right)
\right\vert ^{p}\left\vert \left( 1-\left\vert z\right\vert ^{2}\right)
^{m}\nabla ^{m}\varphi \left( z\right) \right\vert ^{p}\left( 1-\left\vert
z\right\vert ^{2}\right) ^{\sigma p}d\lambda _{n}\left( z\right) \right) ^{%
\frac{1}{p}}\equiv I+II+III.
\end{eqnarray*}%
Terms $I=III$ are easily controlled by%
\begin{equation*}
I\leq C\left\Vert \varphi \right\Vert _{\infty }\left( \int_{S\left(
Q\right) }\left\vert \left( 1-\left\vert z\right\vert ^{2}\right) ^{m+\sigma
}\nabla ^{m}\varphi \left( z\right) \right\vert ^{p}d\lambda _{n}\left(
z\right) \right) ^{\frac{1}{p}}\leq C\left\Vert \varphi \right\Vert _{\infty
}\left\Vert \varphi \right\Vert _{WC_{p}^{\sigma }\left( \mathbb{B}%
_{n}\right) }\left\vert Q\right\vert ^{\frac{\sigma}{n} }.
\end{equation*}%
As for term $II$, fix $1\leq k\leq m-1$ for the moment, and assume without
loss of generality that $k\geq m-k$. Then we can use Cauchy's estimate 
\begin{equation*}
\left\vert \left( 1-\left\vert z\right\vert ^{2}\right) ^{m-k}\nabla
^{m-k}\varphi \left( z\right) \right\vert \lesssim \left\Vert \varphi
\right\Vert _{\infty }
\end{equation*}%
to obtain%
\begin{eqnarray*}
&&\left( \int_{S\left( Q\right) }\left\vert \left( 1-\left\vert z\right\vert
^{2}\right) ^{m-k}\nabla ^{m-k}\varphi \left( z\right) \right\vert
^{p}\left\vert \left( 1-\left\vert z\right\vert ^{2}\right) ^{k}\nabla
^{k}\varphi \left( z\right) \right\vert ^{p}\left( 1-\left\vert z\right\vert
^{2}\right) ^{\sigma p}d\lambda _{n}\left( z\right) \right) ^{\frac{1}{p}} \\
&\leq &\left\Vert \varphi \right\Vert _{\infty }\left( \int_{S\left(
Q\right) }\left\vert \left( 1-\left\vert z\right\vert ^{2}\right) ^{k+\sigma
}\nabla ^{k}\varphi \left( z\right) \right\vert ^{p}d\lambda _{n}\left(
z\right) \right) ^{\frac{1}{p}}\leq \left\Vert \varphi \right\Vert _{\infty
}\left\Vert \varphi \right\Vert _{WC_{p}^{\sigma }\left( \mathbb{B}%
_{n}\right) }\left\vert Q\right\vert ^{\frac{\sigma}{n} }
\end{eqnarray*}%
since $k\geq m-k$ implies $k\geq \frac{m}{2}>\frac{n}{p}>\frac{n}{p}-\sigma $
(this is where we use $m>\frac{2n}{p}$). Altogether then we have%
\begin{equation*}
\left( \int_{S\left( Q\right) }\left\vert \left( 1-\left\vert z\right\vert
^{2}\right) ^{m+\sigma }\nabla ^{m}\left( \varphi ^{2}\right) \left(
z\right) \right\vert ^{p}d\lambda _{n}\left( z\right) \right) ^{\frac{1}{p}%
}\lesssim \left\Vert \varphi \right\Vert _{\infty }\left\Vert \varphi
\right\Vert _{WC_{p}^{\sigma }\left( \mathbb{B}_{n}\right) }\left\vert
Q\right\vert ^{\frac{\sigma}{n} },
\end{equation*}%
which gives (\ref{RTS}): 
\begin{equation*}
\left\Vert \varphi ^{2}\right\Vert _{WC_{p}^{\sigma }\left( \mathbb{B}%
_{n}\right) }\lesssim \left\Vert \varphi \right\Vert _{\infty }\left\Vert
\varphi \right\Vert _{WC_{p}^{\sigma }\left( \mathbb{B}_{n}\right) }\lesssim
\left\Vert \varphi \right\Vert _{K_{p}^{\sigma }\left( \mathbb{B}_{n}\right)
}^{2}.
\end{equation*}
\end{proof}

In particular, when $\mathcal{H}=B_{2}^{\sigma }\left( \mathbb{B}_{n}\right) 
$ and $\sigma >0$, the kernel multiplier space $K_{\mathcal{H}%
}=K_{2}^{\sigma }\left( \mathbb{B}_{n}\right) $ is an algebra.

\subsection{The Corona Property for kernel multiplier algebras on the disk}

Here we prove the Corona Property for the one-dimensional algebras of kernel
multipliers $K_{p}^{\sigma}(\mathbb{D})$ for $0<\sigma < \frac{1}{p}$, $1<p<\infty$.

\begin{theorem}
\label{corona for kernel multipliers}Let $N\geq 2$, $1<p<\infty$ and $0<\sigma < \frac{1}{p%
}$ and suppose that $\varphi _{1},\ldots ,\varphi _{N}\in K_{p}^{\sigma
}\left( \mathbb{D}\right) $ with norm at most one satisfy%
\begin{equation*}
\max \left\{ \left\vert \varphi _{1}\left( z\right) \right\vert
^{2},\ldots ,\left\vert \varphi _{N}\left( z\right) \right\vert ^{2}\right\}
\geq c>0,\ \ \ \ \ z\in \mathbb{D}.
\end{equation*}%
Then there are a positive constant $C$ and $f_{1},\ldots ,f_{N}\in
K_{p}^{\sigma }\left( \mathbb{D}\right) $ satisfying 
\begin{eqnarray*}
\max \left\{ \left\Vert f_{1}\right\Vert_{K_{p}^{\sigma}(\mathbb{D})},\ldots
, \left\Vert f_{N}\right\Vert_{K_{p}^{\sigma}(\mathbb{D})}\right\} &\leq &C,\ \ \ \ \
z\in \mathbb{D} , \\
\varphi _{1}\left( z\right) f_{1}\left( z\right) +\cdots +\varphi _{N}\left(
z\right) f_{N}\left( z\right) &=&1,\ \ \ \ \ z\in \mathbb{D} .
\end{eqnarray*}
\end{theorem}

The Corona Property for the multiplier algebras $M_{B_{p}^{\sigma }\left( 
\mathbb{D}\right) }$ was obtained by Arcozzi, Blasi and Pau in \cite{ArBlPa}%
\ using the Peter Jones solution to the $\overline{\partial}$-equation, and we now adapt
these methods to prove the Corona Property for the algebras $%
K_{p}^{\sigma }\left( \mathbb{D}\right) $ of kernel multipliers. The key
point is that in the arguments in \cite{ArBlPa}, which in turn are a generalization of the results from \cite{Xia2} for the case $p=2$, we are always able to
substitute the weak or one-box Carleson condition for the actual Carleson
condition.

As in \cite{ArBlPa} let\ $dA_{s}\left( z\right) =\left( 1-\left\vert
z\right\vert ^{2}\right) ^{s}dA\left( z\right) $ where $dA\left( z\right) $
is normalized area measure on the disk $\mathbb{D}$ (the connection between $%
s$ there and $\sigma $ here is $s=p\sigma $). If 
$\left\vert g\left( z\right) \right\vert ^{p}dA_{p(1+\sigma)-2 }\left( z\right) $
is a weak Carleson measure for $B_{p}^{\sigma }\left( \mathbb{D}\right) $
with $0<\sigma <\frac{1}{p}$, then Lemma 6.3 in \cite{ArBlPa}\ shows that $%
\left\vert g\left( z\right) \right\vert dA\left( z\right) $ is a Carleson
measure for $H^{2}\left( \mathbb{D}\right) =B_{2}^{\frac{1}{2}}\left( 
\mathbb{D}\right) $ (or in fact $H^p(\mathbb{D})$ since Carleson measures for Hardy space are all the same!). Indeed,%
\begin{eqnarray*}
\int_{S\left( I\right) }\left\vert g\left( z\right) \right\vert dA\left(
z\right) &=&\int_{S\left( I\right) }\left\vert g\left( z\right) \right\vert
\left( 1-\left\vert z\right\vert ^{2}\right) ^{\sigma+1-\frac{2}{p} }\left( 1-\left\vert
z\right\vert ^{2}\right) ^{-\sigma -1+\frac{2}{p}}dA\left( z\right) \\
&\leq &\left( \int_{S\left( I\right) }\left\vert g\left( z\right)
\right\vert ^{p}\left( 1-\left\vert z\right\vert ^{2}\right) ^{p(\sigma+1)-2
}dA\left( z\right) \right) ^{\frac{1}{p}}\left( \int_{S\left( I\right)
}\left( 1-\left\vert z\right\vert ^{2}\right) ^{-(\sigma+1)\frac{p}{p-1}+\frac{2}{p-1} }dA\left( z\right)
\right) ^{\frac{p-1}{p}} \\
&\lesssim &\left( \left\vert I\right\vert ^{p\sigma }\right) ^{\frac{1}{p}%
}\left( \left\vert I\right\vert ^{(1-\sigma)\frac{p}{p-1} }\right) ^{\frac{p-1}{p}%
}=\left\vert I\right\vert .
\end{eqnarray*}%
Now we obtain from the Peter Jones solution to the $\overline{\partial}$-equation the
following `weak' version of Theorem 6.4 in \cite{ArBlPa}.

\begin{theorem}
\label{weak Jones}Let $1<p<\infty$ and $0<\sigma <\frac{1}{p}$. If $\left\vert g\left(
z\right) \right\vert ^{p}dA_{p(\sigma+1)-2 }\left( z\right) $ is a weak Carleson
measure, then there is $f\in C\left( \mathbb{D}\right) $ satisfying%
\begin{equation*}
\frac{\partial f}{\partial \overline{z}}=g\text{ on }\mathbb{D}\text{ in the
sense of distributions},
\end{equation*}%
and such that $f$ has radial limits $f^{\ast }$ a.e. on $\mathbb{T}$ with $%
f^{\ast }\in K\left( L_{p\sigma}^{p}\right) $.
\end{theorem}

Here $L_{p}^{p\sigma}$ is the Sobolev space on the circle $\mathbb{T}$%
: for $f\in L^{p}\left( \mathbb{T}\right) $,%
\begin{equation*}
\left\Vert f\right\Vert _{L_{p\sigma}^{p}\left( \mathbb{T}\right) }^{2}=\left\Vert
f\right\Vert _{L^{p}\left( \mathbb{T}\right) }^{p}+\int_{\mathbb{T}}\int_{%
\mathbb{T}}\frac{\left\vert f\left( e^{it}\right) -f\left( e^{iu}\right)
\right\vert ^{p}}{\left\vert e^{it}-e^{iu}\right\vert ^{2-p\sigma }}%
dudt\approx \left\Vert f\right\Vert _{L^{p}\left( \mathbb{T}\right)
}^{p}+\int \int_{\mathbb{D}}\left\vert \nabla Pf\left( z\right) \right\vert
^{p}\left( 1-\left\vert z\right\vert ^{2}\right) ^{p(\sigma+1)-2 }dA\left(
z\right) ,
\end{equation*}%
where $Pf$ is the Poisson extension of $f$ to the disk, and where the second
summands on each line are actually comparable. The space $K\left(
L_{p\sigma}^{p}\right) $ is the subspace of $L^{\infty }\left( \mathbb{T}\right) $
for which $\left\vert \nabla Pf\left( z\right) \right\vert ^{p}\left(
1-\left\vert z\right\vert ^{2}\right) ^{p(\sigma+1)-2 }dA\left( z\right) $ is a
weak Carleson measure on $B_{p}^{\sigma }\left( \mathbb{D}\right) $ normed by%
\begin{equation*}
\left\Vert f\right\Vert _{K\left( L_{p\sigma}^{p}\right) }=\left\Vert f\right\Vert
_{L^{\infty }\left( \mathbb{T}\right) }+\left\Vert \left\vert \nabla
Pf\left( z\right) \right\vert ^{p}\left( 1-\left\vert z\right\vert
^{2}\right) ^{p(\sigma+1)-2 }dA\left( z\right) \right\Vert _{WCM_{p}^{\sigma }}\ .
\end{equation*}%
A straightforward `weak' modification of Lemma 4.5 in \cite{ArBlPa} shows
that if $F\in C^{1}\left( \mathbb{D}\right) \cap L^{\infty }\left( \mathbb{D}%
\right) $ has radial limits $F^{\ast }$ existing a.e. on $\mathbb{T}$, then%
\begin{equation}
\left\Vert F^{\ast }\right\Vert _{K\left( L_{p\sigma}^{p}\right) }\lesssim
\left\Vert F^{\ast }\right\Vert _{L^{\infty }\left( \mathbb{T}\right)
}+\left\Vert \left\vert \nabla F\left( z\right) \right\vert ^{p}\left(
1-\left\vert z\right\vert ^{2}\right) ^{p(\sigma+1)-2 }dA\left( z\right)
\right\Vert _{WCM_{p}^{\sigma }}\ .
\label{weak criterion}
\end{equation}%
Indeed, Lemma 4.5 proved the analogue%
\begin{equation*}
\left\Vert F^{\ast }\right\Vert _{M\left( L_{p\sigma}^{p}\right) }\lesssim
\left\Vert F^{\ast }\right\Vert _{L^{\infty }\left( \mathbb{T}\right)
}+\left\Vert \left\vert \nabla F\left( z\right) \right\vert ^{p}\left(
1-\left\vert z\right\vert ^{2}\right) ^{p(\sigma+1)-2 }dA\left( z\right)
\right\Vert _{CM_{p}^{\sigma }}
\end{equation*}%
where the norm on the left is the multiplier norm of $F^{\ast }$, which is
equivalent to%
\begin{equation*}
\left\Vert F^{\ast }\right\Vert _{L^{\infty }\left( \mathbb{T}\right)
}+\left\Vert \left\vert \nabla PF^{\ast }\left( z\right) \right\vert
^{p}\left( 1-\left\vert z\right\vert ^{2}\right) ^{p(\sigma+1)-2 }dA\left(
z\right) \right\Vert _{CM_{p}^{\sigma }}.
\end{equation*}%
The passage from Carleson norms to weak Carleson norms here is routine because one simply repeats the proof with a reproducing kernel in place of the generic function used in their proof.

\begin{proof}
To prove Theorem \ref{weak Jones}, we note that since $d\mu \left( z\right)
\equiv g\left( z\right) dA\left( z\right) $ is a Carleson measure for $%
H^{2}\left( \mathbb{D}\right) $, we can invoke the Jones solution with $\nu =%
\frac{\mu }{\left\Vert \mu \right\Vert _{H^{2}\left( \mathbb{D}\right) -Car}}
$: 
\begin{eqnarray*}
u\left( z\right) &=&\int \int_{\mathbb{D}}K\left( \nu ,z,\zeta \right) d\mu
\left( \zeta \right) ,\ \ \ \ \ z\in \overline{\mathbb{D}}, \\
K\left( \nu ,z,\zeta \right) &\equiv &\frac{2i}{\pi }\frac{1-\left\vert
\zeta \right\vert ^{2}}{\left( z-\zeta \right) \left( 1-\overline{\zeta }%
z\right) }\exp \left\{ \int \int_{\left\vert \omega \right\vert \geq
\left\vert \zeta \right\vert }\left[ -\frac{1+\overline{\omega }z}{1-%
\overline{\omega }z}+\frac{1+\overline{\omega }\zeta }{1-\overline{\omega }%
\zeta }\right] d\nu \left( \omega \right) \right\} .
\end{eqnarray*}%
This solution $u\left( z\right) $ to $\frac{\partial f}{\partial \overline{z}%
}=g$ on $\mathbb{D}$ satisfies $u^{\ast }\in L^{\infty }\left( \mathbb{T}%
\right) $. We now show that in addition $u^{\ast }\in K\left( L_{p\sigma
}^{p} \right) $. For this purpose we consider the
function%
\begin{eqnarray*}
v\left( z\right) &=&\frac{2i}{\pi }\int \int_{\mathbb{D}}\frac{1-\left\vert
\zeta \right\vert ^{2}}{\left\vert 1-\overline{\zeta }z\right\vert ^{2}}\exp
\left\{ \int \int_{\left\vert \omega \right\vert \geq \left\vert \zeta
\right\vert }\left[ -\frac{1+\overline{\omega }z}{1-\overline{\omega }z}+%
\frac{1+\overline{\omega }\zeta }{1-\overline{\omega }\zeta }\right]
\left\vert g\left( \omega \right) \right\vert dA\left( \omega \right)
\right\} g\left( \zeta \right) dA\left( \zeta \right) ,
\end{eqnarray*}%
which has the same boundary values as $zu\left( z\right) $. Since $%
e^{-i\theta }\in K\left( L_{p\sigma }^{p}\left( \mathbb{T}\right) \right) $,
it suffices by (\ref{weak criterion}) to show that $v^{\ast }\in L^{\infty
}\left( \mathbb{T}\right) $ and that $\left\vert \nabla v\left( z\right)
\right\vert ^{p}\left( 1-\left\vert z\right\vert ^{2}\right) ^{p(\sigma+1)-2
}dA\left( z\right) $ is a weak Carleson measure for $B_{p}^{\sigma }\left( 
\mathbb{D}\right) $.

We begin with an estimate of A. Nicolau and J. Xiao for this particular
function $v$, see for example \cite{NiXi}%
\begin{equation}
\left\vert \nabla v\left( z\right) \right\vert \leq C\int \int_{\mathbb{D}}%
\frac{\left\vert g\left( \omega \right) \right\vert }{\left\vert 1-\overline{%
\omega }z\right\vert ^{2}}dA\left( \omega \right) .  \label{Nic and X}
\end{equation}%
Since $\left\vert g\left( z\right) \right\vert ^{p}dA_{p(\sigma+1)-2 }\left(
z\right) $ is a weak Carleson measure for $B_{2}^{\sigma }\left( \mathbb{D}%
\right) $, it now follows from (\ref{Nic and X}) and a `weak' modification
of Lemma E in \cite{ArBlPa}, which is a Carleson spreading lemma of Arcozzi, Blasi
and Pau, that $\left\vert \nabla v\left( z\right) \right\vert ^{p}\left(
1-\left\vert z\right\vert ^{2}\right) ^{p(\sigma+1)-2 }dA\left( z\right) $ is a
weak Carleson measure for $B_{p}^{\sigma }\left( \mathbb{D}\right) $ (see for example the analogous lemmas in \cite{ArRoSa2} and \cite{CoSaWi2} and references given there). So it
only remains to show that $v^{\ast }\in L^{\infty }\left( \mathbb{T}\right) $%
. But this is a consequence of the Jones trick that uses first%
\begin{eqnarray*}
&&\func{Re}\left\{ \int \int_{\left\vert \omega \right\vert \geq \left\vert
\zeta \right\vert }\left[ -\frac{1+\overline{\omega }z}{1-\overline{\omega }z%
}+\frac{1+\overline{\omega }\zeta }{1-\overline{\omega }\zeta }\right]
\left\vert g\left( \omega \right) \right\vert dA\left( \omega \right)
\right\} \leq 2\int \int_{\mathbb{D}}\frac{1-\left\vert \zeta \right\vert
^{2}}{\left\vert 1-\overline{\zeta }z\right\vert ^{2}}\left\vert g\left(
\omega \right) \right\vert dA\left( \omega \right) \leq C,
\end{eqnarray*}%
and then for $z\in \mathbb{T}$ that%
\begin{equation*}
\left\vert v\left( z\right) \right\vert \lesssim \int \int_{\mathbb{D}}\frac{%
1-\left\vert \overline{\zeta }z\right\vert ^{2}}{\left\vert 1-\overline{%
\zeta }z\right\vert ^{2}}e^{-\int \int_{\left\vert \omega \right\vert \geq
\left\vert \zeta \right\vert }\frac{1-\left\vert \overline{\omega }%
z\right\vert ^{2}}{\left\vert 1-\overline{\omega }z\right\vert ^{2}}%
\left\vert g\left( \omega \right) \right\vert dA\left( \omega \right)
}\left\vert g\left( \zeta \right) \right\vert dA\left( \zeta \right)
\lesssim 1,
\end{equation*}%
which completes the proof of Theorem \ref{weak Jones}.
\end{proof}

With Theorem \ref{weak Jones} in hand, it is now a routine matter to use the
Koszul complex, or the simplified version in dimension $n=1$, to prove the
Corona Theorem \ref{corona for kernel multipliers} for the algebras $K_{p}^{\sigma }\left( \mathbb{D}\right) $ when $0<\sigma <\frac{1}{p}$.  This proof strategy for weak Carleson measures can be seen in \cite{CoSaWi2}.

\subsection{Bergman spaces of solutions to generalized Cauchy-Riemann
equations}

Let $A\left( z\right) $ and $B\left( z\right) $ be two smooth real-valued
symmetric invertible $n\times n$ matrices defined on a domain $\Omega $ in $%
\mathbb{C}^{n}=\mathbb{R}^{2n}$, where we identify $z=\left( z_{j}\right)
_{j=1}^{n}\in \mathbb{C}^{n}$ with $\left( \left( x_{j}\right)
_{j=1}^{n},\left( y_{j}\right) _{j=1}^{n}\right) \in \mathbb{R}^{n}\times $ $%
\mathbb{R}^{n}$ under the correspondence $z_{j}\leftrightarrow x_{j}+iy_{j}$%
. Then the complex first order partial differential operator $P\equiv
A\left( z\right) \nabla _{x}+iB\left( z\right) \nabla _{y}$ on $\mathbb{C}%
^{n}$ is elliptic, and we can consider the solution space $\mathcal{N}_{P}$
of complex-valued functions (that are necessarily smooth by ellipticity):%
\begin{equation*}
\mathcal{N}_{P}\equiv \left\{ f\in C^{\infty }\left( \Omega ;\mathbb{C}%
\right) :Pf=0\text{ in }\Omega \right\} .
\end{equation*}

\begin{remark}
If $A=B$ is the $n\times n$ identity matrix, then $P=\overline{\partial }$.
Thus in this case $Pf=0$ is the system of Cauchy-Riemann equations on $%
\Omega $, and $\mathcal{N}_{P}$ is the linear space of holomorphic functions
in $\Omega $.
\end{remark}

The linear space $\mathcal{N}_{P}$ is actually an algebra since if $f,g\in 
\mathcal{N}_{P}$, then $P\left( \alpha f+\beta g\right) =\alpha P\left(
f\right) +\beta P\left( g\right) =0$ and $P\left( fg\right) =fP\left(
g\right) +gP\left( g\right) =0$. The real and imaginary parts $u$ of the
functions in $\mathcal{N}_{P}$ are also solutions of the following real
elliptic smooth coefficient divergence form second order equation in $%
\mathbb{R}^{2n}$:%
\begin{equation*}
P^{\ast }Pu=\func{div}\left[ 
\begin{array}{cc}
A^{\ast }A & 0 \\ 
0 & B^{\ast }B%
\end{array}%
\right] \func{grad}u=0.
\end{equation*}%
In particular, from the Schauder interior estimates (see e.g. Theorem 6.9 in 
\cite{GiTr}), solutions $u$ to $P^{\ast }Pu=0$ satisfy a local Lipschitz
inequality%
\begin{equation*}
\left\vert u\left( a\right) -u\left( b\right) \right\vert \leq
C_{K,A,B}\left( \int_{\Omega }\left\vert u\left( x\right) \right\vert
^{2}dx\right) ^{\frac{1}{2}},\ \ \ \ \ a,b\in K,
\end{equation*}%
for every compact subset $K$ of $\Omega $.

Now we define the $P$-Bergman space $A_{P}^{2}\left( \Omega \right) $ on $%
\Omega $ by%
\begin{equation*}
A_{P}^{2}\left( \Omega \right) \equiv \left\{ f\in \mathcal{N}_{P}:\frac{1}{%
\left\vert \Omega \right\vert }\int_{\Omega }\left\vert f\left( x\right)
\right\vert ^{2}dx<\infty \right\} ,
\end{equation*}%
with norm $\left\Vert f\right\Vert _{A_{P}^{2}\left( \Omega \right) }=\sqrt{%
\frac{1}{\left\vert \Omega \right\vert }\int_{\Omega }\left\vert f\left(
x\right) \right\vert ^{2}dx}$. The local Lipschitz inequality for real and
imaginary parts of solutions shows that point evaluations on $%
A_{P}^{2}\left( \Omega \right) $ are continuous, and so $A_{P}^{2}\left(
\Omega \right) $ is a pre-Hilbert function space. Moreover the Montel
property for $A_{P}^{2}\left( \Omega \right) $ is a standard consequence of
the local Lipschitz inequality and the Arzela-Ascoli theorem, and we thus
see that $A_{P}^{2}\left( \Omega \right) $ is a Hilbert function space on $%
\Omega $ with the Montel property. Using that $\mathcal{N}_{P}$ is an
algebra, it is now easy to see with $\mathcal{H}=A_{P}^{2}\left( \Omega
\right) $, that the multiplier algebra $M_{\mathcal{H}}$ of $\mathcal{H}$ is
isometrically isomorphic to $\mathcal{H}^{\infty }\left( \Omega \right) $
equipped with the sup norm $\left\Vert h\right\Vert _{\infty }\equiv
\sup_{a\in \Omega }$ $\left\vert h\left( a\right) \right\vert $. Thus $%
\mathcal{H}=A_{P}^{2}\left( \Omega \right) $ will be multiplier stable if
the kernel functions $k_{a}$ for $\mathcal{H}$ are bounded away from $0$ and 
$\infty $ and lower semicontinuous in $a$, i.e. 
\begin{equation}
\left\Vert k_{a}\right\Vert _{\infty }\text{ is lower semicontinuous in }%
a,\left\Vert \frac{1}{k_{a}}\right\Vert _{\infty }<\infty ,\ \ \ \ \ a\in
\Omega \text{.}  \label{k bound}
\end{equation}%
Thus provided (\ref{k bound}) holds, our alternate Toeplitz corona theorem
reduces the Corona Property for the Banach algebra $\mathcal{H}^{\infty
}\left( \Omega \right) $ of bounded solutions $f$ to $Pf=0$, to the Convex
Poisson Property for the $P$-Bergman space $A_{P}^{2}\left( \Omega \right) $%
. This generalizes the notion of holomorphicity for which such corona
theorems may hold. We do not pursue this direction any further here.

\end{document}